\documentclass[11pt]{article}   %17 dec. 2003--Pour TCS
\usepackage{amssymb, amsmath}
\usepackage{latexsym}
\usepackage{amsthm}
\usepackage{color}
\usepackage[dvips]{epsfig}

\usepackage[dvips]{graphics,graphicx}
\usepackage{mathrsfs}
\usepackage[usenames,dvipsnames]{pstricks}
 \usepackage{pst-grad} % For gradients
\usepackage{pst-plot}
\usepackage[latin1]{inputenc}
\usepackage{psfrag}
\usepackage{graphics}
\usepackage{graphicx}
\usepackage{amssymb}
\usepackage{latexsym}
\usepackage[dvips]{epsfig}

\newcommand{\K}{\mathbb K}

\theoremstyle{plain} \theoremstyle{remark}
\newtheorem{theo}{\bf Theorem}
\newtheorem{defi}{\bf Definition}
\newtheorem{lem}{\bf Lemma}
\newtheorem{cor}{\bf Corollary}
\newtheorem{prop}{\bf Proposition}
\newtheorem{rem}{\bf Remark}
\newtheorem{ex}{\bf Example}
\newcommand{\Int}{\mathbf{Int}}
\newcommand{\Mob}{\mathrm{M\ddot{o}b}}
\newcommand{\Pom}{\mathcal{P}}
\newcommand{\Pop}{\mathscr{P}}
\newcommand{\Cop}{\mathscr{C}}

\newcommand{\gm}{\mathbf{g}}
\newcommand{\dg}{\mathbf{dg}}

\newcommand{\stree}{\mathscr{T}}
\newcommand{\Ope}{\mathscr{O}}
\newcommand{\BB}{\mathbb{B}}
\newcommand{\KK}{\mathbb{K}}
\newcommand{\LL}{\mathbb{L}}
\newcommand{\TT}{\mathbb{T}}
\newcommand{\Ho}{\mathcal{H}}
\newcommand{\Sy}{\mathbb{S}}
\newcommand{\ZZ}{\mathbb{Z}}
\newcommand{\mo}{\mathscr{M}}

\newcommand{\fs}{\mathcal{F}}
\newcommand{\mrm}{\mathrm}
\newcommand{\tbf}{\textbf}
\newcommand{\seg}{\underline{\circ}}
\newcommand{\seb}{\underline{\bullet}}

\newcommand{\vect}{\mathrm{Vec}_{\mathbb{K}}}

\newcommand{\dgr}{\mathrm{dgVec}_{\mathbb{K}}}
\newcommand{\mgr}{\mathrm{gVec}_{\mathbb{K}}}
\newcommand{\be}{\begin{equation}}
\newcommand{\eeq}{\end{equation}}
\newcommand{\sinverse}{\langle -1 \rangle}

\newcommand{\Arb}{\mathscr{A}}
\newcommand{\Rel}{\mathscr{R}}
\newcommand{\Ch}{\mathrm{Ch}}
\newcommand{\coop}{\textrm{!`}}
\newcommand{\xx}{\mathbf{x}}

\newcommand{\mir}{\raisebox{0.3ex}{\scalebox{.6}{$\rightharpoonup$}}}
\newcommand{\mil}{\raisebox{0.3ex}{\scalebox{.6}{$\leftharpoonup$}}}
\newcommand{\mirs}{\raisebox{0.1ex}{\scalebox{.6}{$\rightharpoonup$}}}
\newcommand{\mils}{\raisebox{0.1ex}{\scalebox{.6}{$\leftharpoonup$}}}
\newcommand{\rarr}{\scalebox{.6}{$\rightarrow$}}
\newcommand{\larr}{\scalebox{.6}{$\leftarrow$}}
\newcommand{\Al}{\Lambda^{\larr}}
\newcommand{\Ar}{\Lambda^{\rarr}}

\newcommand{\Sc}{\mathscr{F}}

\newcommand{\Cob}{\mathrm{Cob}}
\newcommand{\Ba}{\mathcal{B}\mathrm{ar}}
\newcommand{\Bao}{\mathscr{B}\mathrm{ar}}
\newcommand{\led}{\langle\langle}
\newcommand{\rid}{\rangle\rangle}
\newcommand{\grader}{\underline{g}}
\begin{document}
\title{Koszul duality for monoids and the operad of enriched rooted
trees.}
\author{Miguel A. M\'endez\\Departamento de Matem\'atica\\
Instituto Venezolano de Investigaciones
Cient\'\i ficas\\
 A.P. 21827, Caracas 1020--A,
Venezuela \\and \\Universidad Central de Venezuela\\ Facultad de
Ciencias\\ Escuela de Matem\'atica.}
\date{}
\maketitle

%
%%%%%%%%%%%%%%%%%%%%%%%%%%%%%%%%%%%%%%%%%%%%%%%%%%%%%%%%%%%%%%%%%
%%%%%%%%%%%%%%%%%%%%%       DEBUT     %%%%%%%%%%%%%%%%%%%%%%%%%%%
%%%%%%%%%%%%%%%%%%%%%%%%%%%%%%%%%%%%%%%%%%%%%%%%%%%%%%%%%%%%%%%%%

\begin{abstract}
We introduce here the notion of Koszul duality for monoids in the
monoidal category of species with respect to the ordinary product.
To each Koszul monoid we associate a class of Koszul algebras in
the sense of Priddy, by taking the corresponding analytic functor.
The operad $\mathscr{A}_M$ of rooted trees enriched with a monoid
$M$ was introduced by the author many years ago. One special case
of that is the operad of ordinary rooted trees, called in the
recent literature the permutative non associative operad. We prove
here that $\mathscr{A}_M$ is Koszul if and only if the
corresponding monoid $M$ is Koszul. In this way we obtain a wide
family of Koszul operads, extending a recent result of Chapoton
and Livernet, and providing an interesting link between Koszul
duality for associative algebras and Koszul duality for operads.
\end{abstract}

%\keywords{Operads, Koszul duality, c-monoids, Enriched trees}
\thanks{Dedicated to the memory of Pierre Leroux.}

\section{Introduction}
The present paper is an overdue followup of the program initiated
in \cite{Miguel} and \cite{Julia-Miguel}, some years before the
introduction in \cite{G-K} of the notion of Koszul duality for
operads, and whose initial step in the above mentioned references
we proceed to describe.
 Recall that a Joyal (set) species is a functor from the category
 $\BB$ of finite sets and bijections, to the category
 $\mathcal{F}$ of finite sets and arbitrary functions. Let $M$
be a (set) species. Denoting by $|M[n]|$ the cardinal of the image
by $M$ of an arbitrary set of cardinal $n$, the
 exponential generating function of $M$ is defined by
\begin{equation}
M(x)=\sum_{n\geq 0}|M[n]|\frac{x^n}{n!}.
\end{equation}

  The program intended to overcome a gap in
the theory of species as originally introduced by Joyal in
\cite{Joyal2}. That was the lack of a combinatorial interpretation
for the formal power series with possible negative coefficients.
  Even though this task could be performed by means of the formal
differences between ordinary species (virtual species)
\cite{Joyal1}, the problem of finding {\em the inverses} by the
process of sieving known as M\"obius inversion on a partially
ordered set \cite{Rota} was open by then. The term {\em inverse}
means here inverse with respect to each one of the usual binary
operations between species: sum, product, and substitution. The
motivation for this approach was that many combinatorial
identities come in pairs. The two identities belonging to a pair
are usually said to be {\em inverses} to each other, because one
comes from the other by M\"obius inversion on a partially ordered
set (e.g., identities involving many families of polynomials of
binomial type and their umbral inverses \cite{Mullin, Rota1,
Reiner}).

A M\"obius species is defined as a functor from the category of
finite sets and bijections to the category of finite disjoint
union of finite partially ordered sets with unique maxima and
minima. All the main features of Joyal's set species are extended
to this context; the generating function of a M\"obius species is
computed by replacing the cardinality of a finite set by the
evaluation of the M\"obius function on partially ordered sets. In
this way, generating functions having negative coefficients are
allowed. To define M\"obius inverses for ordinary Joyal set
species, the concept we had to introduce was that of
left-cancellative monoid (c-monoid) in the context of certain
kinds of monoidal categories. We studied three instances of such
monoidal categories  on species, one for each of the
aforementioned operations. In order to construct the inverse of a
given species with respect to an operation one has to presuppose a
c-monoidal structure on it. For example, if a given set species
$M$ is a c-monoid in the monoidal category related to the
operation of product, we can define a M\"obius species $M^{-1}$
 whose M\"obius generating function
is the inverse with respect to the product of formal power series
of the generating function of $M$. A similar procedure can be
applied to the operations of sum and substitution in order to
define the M\"obius inverses with respect to each of these
operations. A c-monoid in the monoidal category of set species
with respect to the operation of substitution is a cancellative
set operad (c-operad from now on). In this way we introduced there
many c-operads and their associated families of partially ordered
sets, some of them rediscovered in the more recent literature. For
example: the c-operad of pointed sets $E^{\bullet}$ that induces
the poset of pointed partitions (\cite{Julia-Miguel} example
3.13.), studied in \cite{Chapoton1} (the permutative operad). We
also introduced the c-operad of $M$-enriched rooted trees
$\Arb_M$, $M$ being a c-monoid with respect to the product (see
\cite{Miguel} section 5. and \cite{Julia-Miguel} section 3). When
$M$ is the species of sets $E$, $\Arb_E=\Arb$ is the permutative
non-associative operad \cite{Livernet}. We proved by bijective
methods that the M\"obius species $\Arb_M^{\langle -1\rangle}$
(substitutional inverse of $\Arb_M$) and $XM^{-1}$ have the same
M\"obius valuation. In this way we related the substitutional
M\"obius inverse of $\Arb_M$ with the multiplicative M\"obius
inverse of $M$, prefiguring the main result of this paper (see
theorem \ref{main}). This result has an intimate connection with
the Lagrange inversion formula and the combinatorial proof given
by G. Labelle \cite{Gilbert}. Its corollary (corollary
\ref{genchapoton}) extends a recent result of Chapoton and
Livernet \cite{Chapoton}.

The following natural step in this program is to extend this
procedure of inversion from set to tensor species \cite{Joyal1}.
Koszul duality for operads can be seen as a way of doing that for
the operation of substitution. The M\"obius function is the Euler
characteristic of the complex associated to the chains in a
partially ordered set \cite{Rota, Bjorner}. As a matter of fact B.
Vallette
 rediscovered recently the construction of posets from c-operads
\cite{Vallette} and proved that a quadratic c-operad is Koszul if
and only if all the posets in the associated inverse M\"obius
species are Cohen-Macaulay. In this vein, we introduce here the
notion of Koszul duality for quadratic monoids with respect to the
product of species, and prove an analogous result for quadratic
c-monoids (see theorem \ref{C-M}). Given a Koszul monoid, we
obtain a family of Koszul algebras in the sense of Priddy
\cite{Priddy}, by taking the corresponding analytic functor. By
the Schur correspondence we can go back in this construction. In
this way we can translate many classical results about Koszul
duality for associative algebras to this context. However, the
present approach has the advantage of embodying Koszul duality for
associative algebras into the realm of representation theory of
the symmetric groups and symmetric functions.

In a forthcoming paper we shall explore the connections with Hopf
algebras induced by Hopf monoids as in \cite{Marcelo} and with the
antipode of incidence Hopf algebras induced by c-operads as in
\cite{Miguel}.

 As a conclusion we can say that from the
point of view of a combinatorialist, Koszul duality, for operads
or for product monoids, is a sophisticated method of defining the
inverses of some `good' generating functions. An inversion
procedure that goes deeper from the generating function, into the
structure of the algebraic and combinatorial objects that it
enumerates. For example, Andr\'e's generating function for
alternating permutations of even length ($\sec(x)$) \cite{Andre}
is obtained here as the (dimension) generating function of the
Koszul dual of a c-monoid introduced in \cite{Miguel}. The
generating function for alternating permutations of odd length
($\tan(x)$) is associated in the same way to the Koszul dual of a
module on this c-monoid. Moreover this construction gives, by
means of the orthogonal relations of Koszul duality, the ribbon
representations of the symmetric groups studied in \cite{Foulkes},
and in \cite{Foulkes0}, Thm. 6.1. (see also the recent paper of
Stanley \cite{Stanley-Andre}). These generating functions as well
as the character generating functions given in \cite{Carlitz0},
are instances of the inversion formulas associated to Koszul
duality (proposition \ref{genk}). The same procedure is applied
here to the combinatorial interpretation of the inverse of the
Bessel function $J_0(2x)$ given in \cite{Carlitz}. Many other
classical and new enumerative formulas can be obtained by the same
technique.

\section{Set, tensor, $\gm$-tensor, and $\dg$-tensor species}
Let $\KK$ be a field of characteristic zero. Let $\vect$ be the
category whose objects are vector spaces over $\mathbb{K}$ and
linear transformations as morphisms. Denote by $\mgr$ the category
whose objects are graded vector spaces
$V^{\centerdot}=\bigoplus_{i\in \ZZ}V^i$ and morphisms are linear
maps preserving the grading. $\mgr$ is a symmetric monoidal
category with respect to the tensor product of graded vector
spaces \be \label{tensor}\left(V^{\centerdot}\otimes
W^{\centerdot}\right)^n=\bigoplus_{i+j=n} V^i\otimes W^j. \eeq The
symmetry isomorphism is given by \be v\otimes w\mapsto
(-1)^{\mathrm{deg}(v)\mathrm{deg}(w)}w\otimes v.\eeq Denote by
$\dgr$ the category of differential graded vector spaces, or
complex over $\KK$. The objects of $\dgr$ are pairs
$(V^{\centerdot},d)$ where $V^{\centerdot}$ is a graded vector
space, and $d:V^{\centerdot}\rightarrow V^{\centerdot}$ is a
linear map of degree $1$ satisfying $d^2=0$. The tensor product
$(V^{\centerdot},d_1)\otimes (W^{\centerdot},d_2)$ is defined to
be $(V^{\centerdot}\otimes W^{\centerdot},d)$, where
$V^{\centerdot}\otimes W^{\centerdot}$ is as in equation
(\ref{tensor}) and $d$ is given by \be d(v\otimes w)=d_1(v)\otimes
w+(-1)^{\mathrm{deg}(v)}v\otimes d_2(w). \eeq

The direct sum of objects in the categories $\mgr$ and $\dgr$ is
defined in a trivial way. The cohomologies of a
$\mathrm{dg}$-vector space $(V^{\centerdot},d)$,
$H^i(V^{\centerdot},d)=\mathrm{Ker}d^i/\mathrm{Im}d^{i-1}$, $i\in
\ZZ$, defines a functor \be\begin{matrix}H:\dgr\rightarrow\mgr, &
H(V^{\centerdot},d):=\bigoplus_{i\in
\ZZ}H^{i}(V^{\centerdot},d).\end{matrix}\eeq It preserves tensor
products:
 \be \label{homologypreserve}H((V^{\centerdot},d_1)\otimes
 (W^{\centerdot},d_2))=H(V^{\centerdot},d_1)\otimes
 H(W^{\centerdot},d_2)\eeq

Denote by $\fs$ the category whose objects are finite sets and
whose morphisms arbitrary functions. Let $\BB$ be the groupoid
subjacent in $\fs,$ the objects of $\BB$ are finite sets and the
morphisms are bijective functions.

\begin{defi}When clear from the context, the common term
`species' will designate a covariant functor from $\BB$ to
$\mathcal{V}$, where $\mathcal{V}$ is either of the categories
$\fs$, $\vect$, $ \mgr$, or $\dgr$. Specifically, we shall call
$M:\BB\rightarrow \mathcal{V}$ a set, tensor, $\gm$-tensor, or a
$\dg$-tensor species if its codomain category $\mathcal{V}$ is
respectively $\fs$, $\vect$, $ \mgr$, or $\dgr$.
\end{defi}
For a species $F$, we denote by $F[U]$ the image under $F$  of the
finite set $U$ in the corresponding category. In this article we
only deal with finite dimensional (tensor, $\gm$-tensor, and
$\dg$-tensor) species, i.e., species where $F[U]$ is a finite
dimensional vector space for every $U$ in $\BB$. Similarly, for a
bijection between finite sets $\sigma:U\rightarrow V$, $F[\sigma]$
will denote the corresponding isomorphism
$F[\sigma]:F[U]\rightarrow F[V]$. A morphism $\alpha:F\rightarrow
G$ between two species of the same kind, is a natural
transformation between $F$ and $G$ as functors. Two isomorphic
species $F$ and $G$ will be considered as equal and we write
$F=G$. Unless otherwise stated, for $k\geq 0$, $F_k$ will denote
the subspecies of $F$ concentrated on sets of cardinality $k$. The
truncated species $\sum_{j\geq k} F_j$ will be denoted by
$F_{k^+}$. We also denote the species $F$ truncated in 1,
$F_{1^+}$, by $F_+$. By a standard construction, a tensor species
$T$ is equivalent to a sequence $\{T[n]\}_{n=0}^{\infty}$ of
representations of the symmetric groups (see \cite{Joyal1}). Let
$\mu$ and $\lambda$ be two partitions such that
$\mu\subseteq\lambda$. We denote by $\mathcal{S}_{\lambda/\mu}$
the Specht representation corresponding to the skew shape
$\lambda/\mu$, as well as its corresponding tensor species.

Given a set species $M$, we can construct a tensor species by
composing with the functor $l$ (linear span), that goes from $\fs$
to $\vect$, and sends each finite set $S$ to  $\KK\cdot S$, the
free $\KK$-vector space generated by $S$. The category $\vect$ can
be thought of as subcategory of $\mgr$ or of $\dgr$, by
considering a vector space as a graded vector space concentrated
in degree zero in the first case, or as a trivial complex
concentrated in degree zero in the second case. Similarly, the
category $\mgr$ is naturally imbedded into $\dgr$ by providing a
graded vector space with the zero differential. Then, a tensor
species can be thought of either as a $\gm$-tensor species or as a
$\dg$-tensor species, and a $\gm$-tensor species as a $\dg$-tensor
species. In other words, we have the following category imbeddings
\begin{equation}
\fs^{\BB}\subset\vect^{\BB}\subset\mgr^{\BB}\subset\dgr^{\BB}.
\end{equation}

In this article we frequently denote the classical set species and
their corresponding tensor, $\gm$-tensor, and $\dg$-tensor species
with the same symbol.
\subsection{Operations on species}
\begin{defi}\label{operations}Let $F$ and $G$ be two species of the same kind. We define
the operations of sum, product, Hadamard product, substitution and
derivative,
\begin{eqnarray}
(F+G)[U]&:=&F[U]\oplus G[U]\\
(F\cdot G)[U]&:=&\bigoplus_{U_1\uplus U_2=U}F[U_1]\otimes G[U_2]\\
(F\odot G)[U]&:=&F[U]\otimes G[U]\\
\label{subs}F(G)[U]=(F\circ G)[U]&:=&\bigoplus_{\pi\in
\Pi[U]}\left(\bigotimes_{B\in\pi}G[B]\right)\otimes F[\pi]\\
DF[U]=F'[U]&:=&F[U\uplus\{\star\}],\\
D^{k}F[U]&:=&F[U\uplus\{\star_1,\star_2,\dots,\star_k\}]\\
F^{\bullet}[U]&:=&\KK\cdot U\otimes F[U].
\end{eqnarray}
\noindent with $\{\star\}$ standing for any one-element set, and
$\{\star_1,\star_2,\dots,\star_k\}$ for any $k$-element set.\\
 A family of species $F_i$, $i\in I$, is said to be
sumable if for every finite set $U$, $F_i[U]=0$ for almost every
$i\in I$. We can then define the sum $\sum_{i\in I}F_i$ by
\begin{equation}
(\sum_{i\in I}F_i)[U]=\bigoplus_{i:F_i[U]\neq 0}F_i[U].
\end{equation}\end{defi}
 This definitions are valid for all the kinds of
species studied here, the tensor product has to be interpreted in
the corresponding category. In the case of set species, tensor
product has to be interpreted as cartesian product, and direct sum
as disjoint union. In the substitution (equation (\ref{subs})),
$\Pi[U]$ is the set of partitions of the set $U$, and we require
that $G[\emptyset]=0$. The tensor product over the blocks of a
partition in the right hand side of it has to be interpreted as an
unordered tensor product in the corresponding monoidal category.
Because all the monoidal categories in consideration are
symmetric, unordered tensor products have a precise meaning in
each context as coinvariants under the action of the symmetric
group: \be \bigotimes_{i\in I}V_i
:=\left(\bigoplus_{i_1,i_2,\dots,i_k}V_{i_1}\otimes V_{i_2}\otimes
\dots V_{i_k}\right)_{S_k},\eeq \noindent where the direct sum is
taken over all the total orderings of the set $I$.

A $\gm$-tensor species could be thought of as a summable family of
tensor species $\{F^{\underline{k}}| k\in \mathbb{Z}\}$. We
usually denote it with the symbol $F^{\grader}$,
$F^{\grader}=\sum_{k\in\mathbb{Z}}F^{\underline{k}}.$ In the same
vein, a $\dg$-tensor species could be thought of as a summable
family as above, plus a family of natural transformations
$d=\{d_k\}_{k\in\ZZ}$, $d_k:F^{\underline{k}}\rightarrow
F^{\underline{k+1}}$, $k\in \mathbb{Z}$, such that $d_{k+1}\circ
d_{k}=0$, for every $k\in \mathbb{Z}$.
\begin{defi}{\em Dual species.} Let $F$ be a tensor species. The
 dual $F^*$ of $F$ is defined by
\begin{eqnarray}
F^*[U]&=&(F[U])^*\mbox{, for every finite set $U$,}\\
F^*[f]h&=&h\circ F[f^{-1}]\mbox{, for every bijection
$f:U\rightarrow V$ and every $h\in F^*[U]$}.
\end{eqnarray}
For a $\gm$-tensor species $F^{\grader}$, the dual $\gm$-species
$(F^{\grader})^*$  is obtained by dualizing each component and
reversing the grading:
$(F^{\grader})^{*\;\underline{k}}:=(F^{\underline{-k}})^*,\; k\in
\ZZ.$ For a $\dg$-tensor species $(F^{\grader},d)$, its dual is
defined to be $((F^{\grader})^*,d^*)$, where $d^*_k$ is the
adjoint of $d_{-k-1}$, $k\in \ZZ$.
\end{defi}

\begin{ex}
Recall that the species $1$ (empty set indicator), and $X$
(singleton species), are defined as follows

\be \begin{matrix}1[U]=\begin{cases}\{U\}&\mbox{if\, $U$ is the empty set}\\
\emptyset&\mbox{otherwise}\end{cases}&
\; X[U]=\begin{cases}{U}&\mbox{if\, $|U|=1$}\\
\emptyset&\mbox{otherwise}.\end{cases}\end{matrix} \eeq \item
\noindent Observe that $1$ and $X$ are respectively the identities
for the operations of product and substitution. We have that
$1\cdot F=F\cdot 1=F$ for every species $F$, and $F(X)=X(F)=F$ for
every species such that $F[\emptyset]=0$
\end{ex}

For a tensor species $F$, the following two generating functions
are defined: the generating function of the dimensions and the
Frobenius character (see \cite{Joyal1}),
\begin{eqnarray}\label{gen}
F(x)&=&\sum_{n=0}^{\infty}\mathrm{dim}(F[n])\frac{x^n}{n!},\\
\label{Macd}\Ch F(\xx)&=&\sum_{n=0}^{\infty}\sum_{\alpha\vdash n
}\mathrm{tr}F[\alpha]\frac{p_{\alpha}(\xx)}{z_{\alpha}}.
\end{eqnarray}
In the right hand side  of equation (\ref{Macd}) we use
Macdonald's notation for the power sum symmetric function (see
\cite{Macdonald}), as in the rest of this article when dealing
with symmetric functions. When restricted to set species, we get
the exponential generating function of the cardinals $|F[n]|$ in
the first case, and the cycle index series $Z_F(p_1,p_2,\dots)$ in
the second case. These definitions are easily extended to $\gm$
and $\dg$-tensor species, by taking respectively the
Euler-Poincar\'e characteristic and the alternating sum of traces
\be\begin{matrix}\chi(F^{\grader}[n])=\sum_{i\in\ZZ}(-1)^i\mathrm{dim}F^{i}[n],&
\mathrm{tr}F^{\grader}[\alpha]:=\sum_{i\in\ZZ}(-1)^i\mathrm{tr}F^i[\alpha]\end{matrix}\eeq
\noindent instead of $\mathrm{dim}F[n]$, and of
$\mathrm{tr}F[\alpha]$. More explicitly, we have
\begin{eqnarray}
F^{\grader}(x)&=&\sum_{k\in \ZZ}(-1)^k F^{\underline{k}}(x)\\
\Ch F^{\grader}(\xx)&=&\sum_{k\in \ZZ}(-1)^k \Ch
F^{\underline{k}}(\xx).
\end{eqnarray}

The operations of above are preserved by taking generating
functions;
\begin{eqnarray}
(F+G)(x)&=&F(x)+G(x)\\\Ch (F+G)(\xx)&=&\Ch F(\xx)+\Ch G(\xx)\\
(F\cdot G)(x)&=&F(x)G(x)\\\Ch (F\cdot G)(\xx)&=&\Ch F(\xx)\cdot\Ch G(\xx)\\
(F\odot G)(x)&=&F(x)\odot G(x)\label{hadamard}\\\Ch(F\odot
G)(\xx)&=&\Ch
F(\xx)\odot \Ch G(\xx)\label{internal}\\
F(G)(x)&=&F(G(x)) \\\Ch F(G)(\xx)&=&\Ch F(\xx)\ast \Ch
G(\xx)\label{plethysm}\\ DF(x)&=&F'(x)\\
\Ch DF(\xx)&=&\frac{\partial \Ch F(\xx)}{\partial p_1}\\
F^{\bullet}(x)&=&xF'(x)\\
\Ch F^{\bullet}(\xx)&=&p_1(\xx)\frac{\partial \Ch F(\xx)}{\partial
p_1}\end{eqnarray} \noindent In the right hand side of equations
(\ref{hadamard}), and
 (\ref{internal}), $\odot$ means respectively
Hadamard product of formal power series and internal product of
symmetric functions. In the right hand side of equation
(\ref{plethysm}), $\ast$ means plethysm of symmetric functions.
\begin{ex}  The set species $1$ and $X$ have as generating
functions $1(x)=1$, $\Ch 1(\xx)=1$, $X(x)=x$, and $\Ch
X(\xx)=p_1(\xx)$. The set species $E$ (exponential, or uniform
species), and the species $\LL$ of totally ordered sets, are
defined as follows,
\begin{eqnarray}E[U]&=&\{U\}, \\ \LL[U]&=&\{l| l:[n]\rightarrow U, \mbox{ where $n=|U|$
and $l$ being a bijection}\}, \mbox{ for every $U\in
\BB$}.\end{eqnarray} It is clear that the family $X^k\mbox{,
 }k\geq 0,$ is summable and that \begin{equation}\LL=\sum_{k\geq
 0}X^k.
\end{equation}
\end{ex}

Their generating functions are: \begin{eqnarray} E(x)=e^x,&\,&\Ch
E(\xx)=\sum_{n=0}^{\infty}h_n(\xx)=\prod_{n=1}^{\infty}\frac{1}{1-x_n}\\
\LL(x)=\frac{1}{1-x},&\,& \Ch
\LL(\xx)=\frac{1}{1-p_1(\xx)}=\frac{1}{1-\sum_{n=1}^{\infty}x_n}.\end{eqnarray}
\begin{ex} The tensor species $\Lambda$, defined by
\be\begin{matrix}\Lambda[U]=\bigwedge^{|U|}(\KK \cdot U),&\mbox{ }
U\in \BB,\end{matrix}\eeq \noindent has as generating functions
\be \Lambda(x)=e^x,\,\mbox{ and }
\Ch\Lambda(\xx)=\sum_{k=0}^{\infty}e_k(\xx)=\prod_{n=1}^{\infty}(1+x_n).\eeq
\end{ex}

\begin{defi}{\em The functor $\Ho$.} Let $F$ be a $\dg$-tensor species. Let
$\Ho F=H\circ F$ be the $\gm$-tensor species defined as the
functorial composition of the cohomology functor $H$ with $F$.
$\Ho$ is a functor from the category ${\dgr}^{\BB}$ of
$\dg$-tensor species, to the category ${\mgr}^{\BB}$ of
$\gm$-tensor species.
\end{defi} It is not difficult to check, since $H$ preserves
tensor products (equation (\ref{homologypreserve})), that $\Ho$
preserves all the operations in definition \ref{operations}:
\begin{equation}\label{hompre}\begin{matrix}\Ho(F+G)=\Ho F+\Ho G, &  \Ho (F\cdot
G)=\Ho F\cdot\Ho G,\\\Ho (F\odot G)=\Ho F\odot\Ho G, &\Ho
(F(G))=\Ho F(\Ho G).\end{matrix}
\end{equation}
Because of the identities
\be\begin{matrix}\chi(F^{\grader}[n])=\chi(H(F^{\grader}[n])),&\mbox{and}&
\mathrm{tr}F^{\grader}[\alpha]=\mathrm{tr}HF^{\grader}[\alpha],\end{matrix}\eeq
the generating functions (\ref{gen}), and (\ref{Macd}) are also
preserved by $\Ho$, \be\label{preservation}\begin{matrix}\Ho
M(x)=M(x), &\Ch\Ho  M(\xx)=\Ch M(\xx)\end{matrix} .\eeq
\section{Additive, multiplicative and substitutional  inverses} Denote by $s^m$, $m\in \ZZ$, the
shift endofunctor of the category $\dgr$; $s^m:\dgr\rightarrow
\dgr$ defined by $(s^m(V^{\centerdot}))^i=V^{i+m}$,
$s^m(d_{V^{\centerdot}})=(-1)^{m}d_{V^{\centerdot}}.$ The shift
 can also be defined (by restriction) as an endofunctor of $\mgr$.
\begin{defi}{\em Additive Inverses.} Let $F$ be a species of any kind. $F$ can be always thought of as a
$\dg$-tensor species or as a $\gm$-tensor species. There are two
kinds of `additive inverses'
of $F$. %\raisebox{0.5ex}{$\mil$}$F$
 \begin{eqnarray}\mil F&:=&s^{-1}\circ F\\\mir F&:=&s^1\circ F
 \end{eqnarray}
\end{defi}
Clearly, $(\mil F)(x)=-F(x)=(\mir F)(x)$ and $\Ch(\mil
F)(\xx)=\Ch(\mir F)(\xx)=-\Ch F(\xx)$.   There are two other
interesting ways of shifting $F$. Define
$F^{\scriptscriptstyle{\leftarrow}}$ and $F^{\rarr}$ as follows,
\begin{eqnarray}F^{\larr}&:=&
\sum_{k=0}^{\infty}s^{-k}\circ F_k,\\
F^{\rarr}&:=&\sum_{k=0}^{\infty}s^{k}\circ F_k.
 \end{eqnarray}

\begin{prop} We have the identities,
\begin{eqnarray}\label{sml}F(\mil X)&=&\Lambda^{\larr}\odot F,\\
\label{smr}F(\mir X)&=&\Lambda^{\rarr}\odot F.
\end{eqnarray}
\end{prop}
\begin{proof}
Let us prove equation (\ref{sml}), the proof of (\ref{smr}) is
similar. By the definition of substitution we have that \be F(\mil
X)[U]=\left(\bigotimes_{b\in U}(\mil X)[\{b\}]\right)\otimes
F[U]=\left(\bigotimes_{b\in U}s^{-1}\KK\cdot\{b\}\right)\otimes
F[U]=\left(s^{-|U|}\Lambda[U]\right)\otimes F[U].\eeq The last
step because we have the isomorphism (see \cite{G-K} lemma
(3.2.9)) \be\label{49} \bigotimes_{j\in
J}s^{-1}V_j^{\centerdot}=s^{-|J|} \left(\bigotimes_{j\in
J}V_j^{\centerdot}\right)\otimes \Lambda[J].\eeq \end{proof}

Their generating functions are: \begin{eqnarray}F(\mil X)(x)&=&
F(\mir X)(x)=\sum_{k=0}^{\infty}(-1)^k F_k(x)=F(-x)\\\Ch F(\mil
X)(\xx)&=&\Ch F(\mir X)(\xx)=\sum_{k=0}^{\infty}(-1)^k\Ch
F_k(\xx)\odot e_k(\xx)=\Ch F(-\xx)\odot
\sum_{k=0}^{\infty}e_k(\xx).\end{eqnarray}
\begin{ex} {\em Koszul Complexes.} We have the identities
\begin{eqnarray}
E(\mil X)=\Al\odot E=\Al,&&\, E(\mir X)=\Lambda^{\rarr}\odot
E=\Lambda^{\rarr}.
\end{eqnarray}
Let $P=\mil X+X$ and $Q=\mir X+X$ be the $\dg$-tensor species with
differential $d_{-1}:X=P^{\underline{-1}}\rightarrow
P^{\underline{o}}=X$ (resp. $d'_{0}:X=Q^{\underline{0}}\rightarrow
Q^{\underline{1}}=X$) in each case the trivial isomorphism and
zero for $k\neq -1$ (respectively $k\neq 0$). As $\gm$-tensor
species $E(P)$ and $E(Q)$ are  respectively equal
to \begin{eqnarray}E(\mil X+X)&=&E(\mil X)\cdot E=\Al\cdot E,\\
E(\mir X+X)&=&E\cdot E(\mir X)=E\cdot\Ar.\end{eqnarray} As
$\dg$-tensor species, the differentials that comes from $P$ and
$Q$ by using the definition of substitution of species are easily
seen to be:

\begin{eqnarray} d_k:((\Al)^{\underline{k}}\cdot
E)[U]&\rightarrow & ((\Al)^{\underline{k+1}}\cdot E)[U]\mbox{, $k=-|U|,\dots,0$}\\
d_k(a_1\wedge\dots \wedge a_k\otimes U_2)&=&\sum_{i=1}^k
(-1)^{i-1}a_1\wedge\dots \wedge\widehat{a_i}\wedge\dots \wedge
a_k\otimes U_2\cup\{a_i\}\\d'_k:(
E\cdot(\Ar)^{\underline{k}})[U]&\rightarrow & (
E\cdot(\Ar)^{\underline{k+1}})[U]\mbox{, $k=0,\dots,|U|$}\\
d_k'(U_1\otimes a_1\wedge a_2\wedge\dots \wedge a_k)&=&\sum_{a\in
U_1}(U_1-\{a\})\otimes a\wedge a_1\wedge a_2\wedge\dots \wedge a_k
\end{eqnarray}

By equation \ref{hompre}, $\Ho E(\mil X +X)=\Ho E(\mir X +
X)=E(0)=1$.
\end{ex}

Let $F$ be of the form $F=1+F_+$. The inverse of its generating
function is \be
F(x)^{-1}=(1+F_+(x))^{-1}=1-F_+(x)+F_+(x)^2-\dots=\LL(-F_+(x)).\eeq
This motivates the following definition.
\begin{defi}{\em Multiplicative Inverses.} Let $F$ be as above.
We  define two `multiplicative inverses' of $F$, \begin{eqnarray}
\label{inversamul}F^{\mils 1}&:=&\LL(\mil F_+),\\
F^{\mirs 1}&:=&\LL(\mir F_+).\end{eqnarray}Obviously, \be F^{\mils
1}(x)=F^{\mirs 1}(x)=F(x)^{-1}\eeq\be\Ch F^{\mils 1}(\xx)=\Ch
F^{\mirs 1}(\xx)=\sum_{k=0}^{\infty} (-1)^k (\Ch
F_+(\xx))^k=\left(\Ch F(\xx)\right)^{-1}.\eeq
\end{defi}
Denote by $\Sc$ the set species of commutative Schr\"{o}der trees,
or generalized commutative parenthesizations. It satisfies the
implicit equation \be \Sc=X+E_{2^+}(\Sc).\eeq

The structures of $\Sc[U]$ are rooted trees whose leaves are
labelled with the elements of $U,$ and whose internal vertices,
each one with at least two sons, are unlabelled. More generally,
let $G$ be a species of the form $G=X+G_{2^+}$. The species of
$G$-enriched Schr\"{o}der trees (see for example \cite{shroeder,
Joyal2}) is the solution (fixed point) of the implicit equation
\be \Sc_{G}=X+G_{2^+}(\Sc_G).\label{schroder}\eeq  $\Sc_{G}$ is
the free operad generated by $G$. This kind of structures should
be called Hipparchus trees, who computed in the second century
B.C. the number of composed propositions out of $10$ simple
propositions (see \cite{Stanley} p. 213, and
\cite{Stanley-Hipparcus}). In modern combinatorial language, the
number of generalized non-commutative parenthesizations (or
bracketings) using $10$ undistinguishable symbols. In our
notation: $|\Sc_{\LL_+}[10]|/10!$.
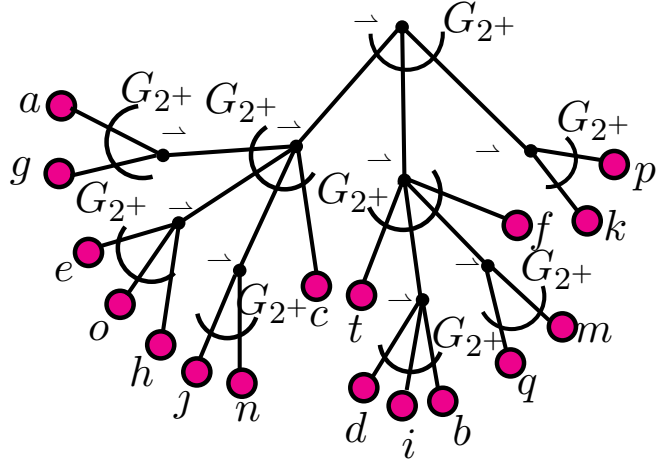
\begin{figure}\begin{center}
% Generated with LaTeXDraw 1.9.5
% Wed Aug 27 15:00:52 BOT 2008
% \usepackage[usenames,dvipsnames]{pstricks}
% \usepackage{epsfig}
% \usepackage{pst-grad} % For gradients
% \usepackage{pst-plot} % For axes
\scalebox{1.5} % Change this value to rescale the drawing.
{
\begin{pspicture}(0,-2.08375)(6.2028127,2.0221875)
\psdots[dotsize=0.12](3.6009376,1.78375)
\psdots[dotsize=0.12](2.6609375,0.72375)
\psdots[dotsize=0.12](3.6209376,0.42375)
\psdots[dotsize=0.12](4.7409377,0.68375)
\psdots[dotsize=0.12](3.7809374,-0.63625)
\psdots[dotsize=0.12](4.3609376,-0.33625)
\psdots[dotsize=0.12](1.4809375,0.64375)
\psdots[dotsize=0.12](1.6209375,0.04375)
\psdots[dotsize=0.12](2.1609375,-0.37625)
\pscircle[linewidth=0.04,dimen=outer,fillstyle=solid,fillcolor=Magenta](0.5609375,0.48375){0.14}
\pscircle[linewidth=0.04,dimen=outer,fillstyle=solid,fillcolor=Magenta](0.8209375,-0.21625){0.14}
\pscircle[linewidth=0.04,dimen=outer,fillstyle=solid,fillcolor=Magenta](1.4609375,-1.03625){0.14}
\pscircle[linewidth=0.04,dimen=outer,fillstyle=solid,fillcolor=Magenta](2.1809375,-1.37625){0.14}
\pscircle[linewidth=0.04,dimen=outer,fillstyle=solid,fillcolor=Magenta](1.1009375,-0.67625){0.14}
\pscircle[linewidth=0.04,dimen=outer,fillstyle=solid,fillcolor=Magenta](1.7809376,-1.27625){0.14}
\pscircle[linewidth=0.04,dimen=outer,fillstyle=solid,fillcolor=Magenta](0.5809375,1.08375){0.14}
\pscircle[linewidth=0.04,dimen=outer,fillstyle=solid,fillcolor=Magenta](3.9609375,-1.53625){0.14}
\pscircle[linewidth=0.04,dimen=outer,fillstyle=solid,fillcolor=Magenta](2.8409376,-0.51625){0.14}
\pscircle[linewidth=0.04,dimen=outer,fillstyle=solid,fillcolor=Magenta](3.2609375,-1.41625){0.14}
\psline[linewidth=0.036cm](3.6009376,1.80375)(2.6209376,0.68375)
\psline[linewidth=0.036cm](3.5809374,0.40375)(3.5609374,0.40375)
\psline[linewidth=0.036cm](4.7209377,0.68375)(3.6009376,1.78375)
\psline[linewidth=0.036cm](2.6609375,0.72375)(1.6209375,0.04375)
\psline[linewidth=0.036cm](2.6409376,0.68375)(2.1609375,-0.35625)
\psline[linewidth=0.036cm](2.6809375,0.68375)(2.8409376,-0.39625)
\psline[linewidth=0.036cm](3.6009376,0.40375)(3.2609375,-0.47625)
\psline[linewidth=0.036cm](3.6209376,0.42375)(3.7609375,-0.59625)
\psline[linewidth=0.036cm](3.6209376,0.42375)(4.3409376,-0.31625)
\psline[linewidth=0.036cm](3.6209376,0.44375)(4.5209374,0.08375)
\psline[linewidth=0.036cm](2.6209376,0.72375)(1.4809375,0.64375)
\pscircle[linewidth=0.04,dimen=outer,fillstyle=solid,fillcolor=Magenta](5.2409377,0.06375){0.14}
\pscircle[linewidth=0.04,dimen=outer,fillstyle=solid,fillcolor=Magenta](5.4809375,0.56375){0.14}
\pscircle[linewidth=0.04,dimen=outer,fillstyle=solid,fillcolor=Magenta](4.6209373,0.02375){0.14}
\pscircle[linewidth=0.04,dimen=outer,fillstyle=solid,fillcolor=Magenta](3.6009376,-1.57625){0.14}
\pscircle[linewidth=0.04,dimen=outer,fillstyle=solid,fillcolor=Magenta](3.2409375,-0.59625){0.14}
\psline[linewidth=0.036cm](1.4409375,0.66375)(0.6609375,1.06375)
\psline[linewidth=0.036cm](1.4209375,0.64375)(0.6609375,0.46375)
\psline[linewidth=0.036cm](1.6209375,0.04375)(0.9409375,-0.15625)
\psline[linewidth=0.036cm](1.6009375,0.04375)(1.1609375,-0.59625)
\psline[linewidth=0.036cm](1.6209375,0.02375)(1.4809375,-0.91625)
\psline[linewidth=0.036cm](2.1209376,-0.41625)(1.8009375,-1.17625)
\psline[linewidth=0.036cm](2.1609375,-0.35625)(2.1609375,-1.27625)
\pscircle[linewidth=0.04,dimen=outer,fillstyle=solid,fillcolor=Magenta](4.5609374,-1.19625){0.14}
\pscircle[linewidth=0.04,dimen=outer,fillstyle=solid,fillcolor=Magenta](5.0209374,-0.87625){0.14}
\usefont{T1}{ptm}{m}{n}
\rput(1.5823437,0.81375){$\mirs$}
\psline[linewidth=0.036cm](3.7609375,-0.63625)(3.3409376,-1.31625)
\psline[linewidth=0.036cm](3.7809374,-0.63625)(3.6409376,-1.43625)
\psline[linewidth=0.036cm](3.7809374,-0.63625)(3.9209375,-1.43625)
\psline[linewidth=0.036cm](4.3409376,-0.33625)(4.5209374,-1.07625)
\psline[linewidth=0.036cm](4.3609376,-0.31625)(4.9009376,-0.81625)
\psline[linewidth=0.036cm](4.7409377,0.68375)(5.1609373,0.14375)
\psline[linewidth=0.036cm](4.7409377,0.70375)(5.3609376,0.60375)
\usefont{T1}{ptm}{m}{n}
\rput(3.2823439,1.81375){$\mirs$}
\usefont{T1}{ptm}{m}{n}
\rput(1.6423438,0.15375){$\mirs$}
\usefont{T1}{ptm}{m}{n}
\rput(3.4023438,0.57375){$\mirs$}
\usefont{T1}{ptm}{m}{n}
\rput(2.6023438,0.87375){$\mirs$}
\usefont{T1}{ptm}{m}{n}
\rput(1.9823438,-0.30625){$\mirs$}
\usefont{T1}{ptm}{m}{n}
\rput(3.5823438,-0.66625){$\mirs$}
\usefont{T1}{ptm}{m}{n}
\rput(4.362344,0.65375){$\mirs$}
\usefont{T1}{ptm}{m}{n}
\rput(4.182344,-0.34625){$\mirs$}
\psline[linewidth=0.036cm](3.6009376,1.78375)(3.6209376,0.44375)
\psarc[linewidth=0.036](3.6409376,1.72375){0.32}{180.0}{3.3664606}
\usefont{T1}{ptm}{m}{n}
\rput(4.302344,1.83375){$G_{ 2^+}$}
\usefont{T1}{ptm}{m}{n}
\rput(5.3123436,0.95375){$G_{2^+}$}
\psarc[linewidth=0.036](4.8909373,0.61375){0.25}{266.1859}{49.085617}
\psarc[linewidth=0.04](2.5609374,0.64375){0.3}{122.47119}{329.03625}
\psarc[linewidth=0.04](1.3009375,0.76375){0.32}{97.125015}{285.9454}
\psarc[linewidth=0.04](3.6109376,0.31375){0.33}{203.1986}{36.469234}
\psarc[linewidth=0.04](2.0509374,-0.70625){0.27}{191.30994}{341.56506}
\psarc[linewidth=0.04](3.6809375,-0.99625){0.28}{188.1301}{338.19858}
\usefont{T1}{ptm}{m}{n}
\rput(2.1923437,1.11375){$G_{2^+}$}
\usefont{T1}{ptm}{m}{n}
\rput(1.4423437,1.21375){$G_{ 2^+}$}
\psarc[linewidth=0.04](1.3809375,-0.17625){0.3}{141.3402}{312.27368}
\usefont{T1}{ptm}{m}{n}
\rput(1.0423437,0.23375){$G_{ 2^+}$}
\usefont{T1}{ptm}{m}{n}
\rput(2.4623437,-0.66625){$G_{ 2^+}$}
\usefont{T1}{ptm}{m}{n}
\rput(4.202344,-0.96625){$G_{ 2^+}$}
\usefont{T1}{ptm}{m}{n}
\rput(3.1823437,0.31375){$G_{ 2^+}$}
\psarc[linewidth=0.036](4.5709376,-0.60625){0.29}{208.30075}{18.434948}
\usefont{T1}{ptm}{m}{n}
\rput(4.992344,-0.36625){$G_{2^+}$}
\usefont{T1}{ptm}{m}{n}
\rput(0.30234376,1.11375){$a$}
\usefont{T1}{ptm}{m}{n}
\rput(4.132344,-1.78625){$b$}
\usefont{T1}{ptm}{m}{n}
\rput(2.8623438,-0.78625){$c$}
\usefont{T1}{ptm}{m}{n}
\rput(3.1923437,-1.76625){$d$}
\usefont{T1}{ptm}{m}{n}
\rput(0.60234374,-0.38625){$e$}
\usefont{T1}{ptm}{m}{n}
\rput(4.842344,-0.02625){$f$}
\usefont{T1}{ptm}{m}{n}
\rput(0.22234374,0.51375){$g$}
\usefont{T1}{ptm}{m}{n}
\rput(1.2923437,-1.26625){$h$}
\usefont{T1}{ptm}{m}{n}
\rput(3.6623437,-1.90625){$i$}
\usefont{T1}{ptm}{m}{n}
\rput(1.6723437,-1.52625){$j$}
\usefont{T1}{ptm}{m}{n}
\rput(5.492344,0.03375){$k$}
\usefont{T1}{ptm}{m}{n}
\rput(5.302344,-0.92625){$m$}
\usefont{T1}{ptm}{m}{n}
\rput(2.2323437,-1.62625){$n$}
\usefont{T1}{ptm}{m}{n}
\rput(0.91234374,-0.90625){$o$}
\usefont{T1}{ptm}{m}{n}
\rput(5.7523437,0.45375){$p$}
\usefont{T1}{ptm}{m}{n}
\rput(4.7123437,-1.42625){$q$}
\usefont{T1}{ptm}{m}{n}
\rput(3.1923437,-0.88625){$t$}
\end{pspicture}
}

\end{center}

\caption {Graphical representation of the  substitutional
inverse}\label{substinverse}
\end{figure}

 Let us introduce some notation in order to give an explicit
description of the object $\Sc_{G}[U],$ for  a finite set $U$. For
a tree $t\in\Sc[U]$, we denote by $\mathrm{Iv}(t)$ the set of
internal vertices of $t$. For a vertex $v\in \mathrm{Iv}(t)$
denote by $t_v$ the subtree of $t$ having $v$ as a root and as
vertices all de descendants of $v$ in $t$. Let $U_v$ be the set of
leaves of $t_v$. We shall identify the internal (unlabelled)
vertex $v$ with $U_v$. The set $U_v$ will be used as a label for
$v$. In the same way we identify  $\mathrm{Iv}(t)$ with the set
$\{U_v|v\in \mathrm{Iv}(t)\}$. Let $\{v_1,v_2,\dots,v_k\}$ be the
set of sons of $v$ in $t$. We denote by $\pi_v$ the partition
$\{U_{v_i}|i=1,2,\dots,k\}$ of the set $U_v$ of leaves of $t_v$,
generated by the leaves of the set of trees
$\{t_{v_i}|i=1,2,\dots,k\}$ attached to $v$. Consistently, we
identify the partition $\pi_v$ with the set of sons of $v$. The
object $\Sc_{G}[U]$ is explicitly given by the
formula\be\label{explic} \Sc_{G}[U]=\bigoplus_{t\in
\Sc[U]}\bigotimes_{v\in\mathrm{Iv}(t)}G_{2^+}[\pi_v]. \eeq $\Sc_G$
has a natural grading $\Sc_G^{\grader}=\sum_{k=0}^{\infty}
\Sc_G^{\underline{k}}$, where $\Sc_G^{\underline{0}}=X$ and for
$k\geq 1$, $\Sc_G^{\underline{k}}$  is the species of the
$G$-enriched Shr\"oder trees with exactly $k$ internal vertices.

From equation (\ref{schroder}), the generating function $\Sc_G(x)$
 is the solution of the implicit equation \be \Sc_G(x)=x+G_{2^+}(\Sc_G(x)).\eeq Then, we
 have that \be\Sc_G(x)-G_{2^+}(\Sc_G(x))=x,\eeq and equivalently, that
 the substitutional inverse  $(x-G_{\geq 2}(x))^{\langle-1\rangle}$
 of $(x-G_{2^+}(x)),$ is $\Sc_G(x).$  Similarly, the Frobenius
 character, $\Ch\Sc_G(\xx)$ is the plethystic
 inverse of $(p_1(\xx)-\Ch G_{2^+}(\xx))$;
 \be \Ch\Sc_G(\xx)=(p_1(\xx)-\Ch G_{2^+}(\xx))^{\langle-1\rangle}.\eeq This motivates the following definition.
\begin{defi}{\em Substitutional Inverse.} Let $G$ be a species as above. We define two `substitutional
inverses' of $G$, \begin{eqnarray} G^{\langle\mils 1\rangle}&:=&\Sc_{\mils G}, \\
G^{\langle\mirs 1\rangle}&:=&\Sc_{\mirs G}.\end{eqnarray}
\end{defi}

It is easy to prove that
 \be G^{\langle\mils
1\rangle}(x)=G^{\langle\mirs
1\rangle}(x)=G^{\langle-1\rangle}(x),\eeq \noindent and that
\be\Ch G^{\langle\mils 1\rangle}(\xx)=\Ch G^{\langle\mirs
1\rangle}(\xx)=(\Ch G (\xx))^{\langle-1\rangle}.\eeq

\section{Monoidal categories on species}
%Escribir mas aqui(diagramas etc.), ver libro de Marcelo.

 The category
$\mathcal{V}^{\BB}$ is monoidal  with respect to the product of
species, having $1$ as identity object. A monoid in
 $(\mathcal{V}^{\BB},\cdot\,,1)$ is a species $M$ together with
morphisms $\nu:M \cdot M\rightarrow M$ and $e:1\rightarrow M$,
satisfying the associative and unity properties (see
\cite{Maclane} for the general definition of monoidal categories
and of a monoid in a monoidal category).
 In this article we only consider monoids of the form
 $M=1+M_+$. Observe that in that case the morphism $e:1\rightarrow M$ is unique.
Then, its monoidal structure is completely determined by giving
the morphism $\nu:M\cdot M\rightarrow M$ satisfying the
associative and identity properties: for every finite set $U$ \be
\nu(\nu(m_{U_1}\otimes m_{U_2})\otimes m_{U_3})=\nu(m_{U_1}\otimes
\nu(m_{U_2}\otimes m_{U_3}))\mbox{, } m_{U_1}\otimes
m_{U_2}\otimes m_{U_3}\in (M.M.M)[U], \eeq \noindent and, \be
\nu(e\otimes m)=\nu(m\otimes e)=m\; \mbox{, }e\otimes m\in
(M_1\cdot M)[U]=(1\cdot M)[U]\mbox{, }m\otimes e\in (M\cdot
M_1)[U]=(M\cdot 1)[U].\eeq

 Given a monoid $(M,\nu)$, a
species $N$ is called a (right) $M$-module if there exists a right
action $\tau:N.M\rightarrow M$, of $M$  on $N$ that satisfies the
pseudoassociative and identity properties: for every finite set
$U$ we have \be \tau(\tau(n_{U_1}\otimes m_{U_2})\otimes
m_{U_3})=\tau(n_{U_1}\otimes \nu(m_{U_2}\otimes m_{U_3}))\mbox{, }
n_{U_1}\otimes m_{U_2}\otimes m_{U_3}\in (N.M.M)[U], \eeq
 \noindent and, \be \tau(n\otimes
e)=n\; \mbox{, for }n\otimes e\in (N\cdot M_1)[U]=(N\cdot
1)[U].\eeq

In a standard way we define the dual notions of comonoid an
comodule with respect to a comonoid.

 Let us denote by $\mathcal{V}^{\BB}_+$ the category of species
$F$ satisfying $F[\emptyset]=0$. $\mathcal{V}^{\BB}_+$ is a
monoidal category with respect to the substitution of species and
having $X$ as identity object.

 An {\em operad} $(\Ope,\eta, e)$ is defined to be
a monoid in the monoidal category $(\mathcal{V}^{\BB}_+, X,
\circ)$. Equivalently, $\eta:\Ope(\Ope)\rightarrow \Ope$ and
$e:X\rightarrow \Ope$ are morphisms in $\mathcal{V}^{\BB}_+$
(natural transformations) that satisfy the associativity and
identity properties. $\Ope$ will be called a set-operad when
$\mathcal{V}=\fs$. Similarly, $\Ope$ will be called an operad, a
$\gm$-operad or a $\dg$-operad if the category $\mathcal{V}$ is
respectively $\vect$, $\mgr$ or $\dgr$.

 We only
consider here operads of the form $\Ope=X+\Ope_{2^+}$, whose
structure is completely determined by giving the morphism
$\eta:\Ope(\Ope)\rightarrow\Ope$.

\subsection{Koszul duality for quadratic monoids and modules}
\begin{defi} {\em Quadratic Monoids}.
Let $F$ be a tensor species such that $F[\emptyset]=\emptyset$.
The species $$\LL(F)=1+F+F^2+F^3\dots$$ is the free monoid
generated by
 $F$. Let $R$ be a subspecies of $F^2$ and let $\mathcal{R}_M=\langle R\rangle$
be the monoid ideal generated by $R$ in $\LL(F)$. Explicitly
\begin{equation}\mathcal{R}_M=\sum_{k=0}^{\infty}\mathcal{R}_M^{\underline{k}},
\end{equation}
\noindent where
\begin{equation}
\mathcal{R}_M^{\underline{k}}=\sum_{i=0}^{k-2}F^i R
F^{k-2-i}\subseteq F^k.
\end{equation}
The monoid
$M=\LL(F)/\mathcal{R}_M=1+F+F^2/R+F^3/\mathcal{R}_M^{\underline{3}}+\dots$,
will be called the quadratic monoid with generators in $F$,
quadratic relations in $R_M$, and denoted by $M=\mathcal{M}(F,R)$.
 There is a natural grading on the monoid $M$, the
corresponding graded monoid will be denoted by $M^{\grader}$,
$M^{\underline{k}}=F^k/\mathcal{R}_M^{\underline{k}}$.
\end{defi}
\begin{defi}{\em Quadratic $M$-modules}. Let $M$ be a quadratic monoid as above
and $G$ an arbitrary tensor species. The species $$G\cdot \LL(F)
=G+G\cdot F +G\cdot F^2 +\dots$$ is the free (right)
$\LL(F)$-module generated by $G$. Let $R_N\subseteq G.F$ be a
subspecies of $G.F$, and $\mathcal{R}_{M,N}$ the submodule
generated in $G\cdot \LL(F)$ by $\mathcal {R}_M$ and $R_N$;
\begin{equation}
\mathcal{R}_{M,N}=\sum_{k=0}^{\infty}\mathcal{R}_{M,N}^{\underline{k}},
\end{equation}
\noindent where
\begin{eqnarray}
\mathcal{R}_{M,N}^{\underline{0}}&=&0\subseteq G\mbox{,
}\\\mathcal{R}_{M,N}^{\underline{1}}&=&R_N\subseteq G\cdot
F\mbox{, }
\\\mathcal{R}_{M,N}^{\underline{k}}&=&R_N\cdot F^{k-1}+\sum_{i=0}^{k-2}F^i
R_M F^{k-2-i}\cdot G + \subseteq G\cdot F^k\mbox{, }k\geq 2.
\end{eqnarray}
The $M$-module \begin{equation}N=\left(\LL(F)\cdot
G\right)/\mathcal{R}_{M,N}=G+(F\cdot G)/R_N+(F^2\cdot
G)/\mathcal{R}_{M,N}^{\underline{2}}+\dots\end{equation}\noindent
will be called the quadratic $M$-module generated by $G$ with
relations in $R_N$, and denoted by
\begin{equation}N=\mathscr{M}_M(G,R_N).\end{equation} $N$ has a
natural grading, the corresponding $M^{\grader}$-graded module
will be denoted by $N^{\grader}$,
\begin{equation}\label{relmod}N^{\underline{k}}:=(
G\cdot F^k)/\mathcal{R}_{M,N}^{\underline{k}}\mbox{
}(k=0,1,2,\dots).\end{equation}
\end{defi}
\begin{defi}{\em Quadratic duality.} Let $M=\mathcal{M}(F,R_M)$ be
a quadratic monoid. Define the quadratic dual of $M$ by \be
M^{!.}:=\mathcal{M}(F^*,R_M^{\perp}),\end{equation} \noindent
where $R_M^{\perp}$ is the annihilator of $R_M$ in $(F^2)^*
=(F^*)^2$. Let $N=\mo_M(G,R_N)$ be a quadratic $M$-module. The
quadratic dual of $N$ is the $M^{!.}$-module defined by
\begin{equation}
N^{!.}=\mo_{M^{!.}}(G^*,R_N^{\perp}),
\end{equation}
\noindent where $R_{N}^{\perp}$ is the annihilator of $R_N$ in
$(F\cdot G)^*=F^*\cdot G^*$. The dual $M^{\coop .}:=(M^{!.})^*$ is
a comonoid. Similarly, for a quadratic $M$-module $N$, $N^{\coop
.}:=(N^{!.})^*$ is an $M^{\coop }$-comodule.
\end{defi}
Observe that the species of relations $R_M$ and $R_N$ are
respectively the kernels of $$\nu^{\underline{2}}:F\cdot
F\rightarrow M^{\underline{2}}\mbox{, and
}\tau^{\underline{1}}:G\cdot F\rightarrow N^{\underline{1}}.$$
Then, the orthogonal relations are respectively the images of
$(\nu^{\underline{2}})^*$ and $(\tau^{\underline{1}})^*$.

\begin{ex}
$\LL$ is a quadratic monoid, $\nu:\LL\cdot\LL\rightarrow \LL$
being the concatenation of linear orders. $\LL=\mathcal{M}(X,0)$,
then
$\LL^{!.}=\mathcal{M}(X^*,(X^*)^2)=\LL(X^*)/\langle(X^*)^2\rangle=1+X^*.$
\end{ex}

\begin{ex}The species $E$ is a monoid. There is a unique natural transformation
$\nu:E\cdot E\rightarrow E$ sending each product $U_1\otimes
U_2\in (E\cdot E)[U]$, to $U_1\uplus U_2=U$. $(E,\nu)$ is
quadratic, $E=\mathcal{M}(X,R_E)$, where
\begin{equation}R_E[\{a,b\}]=\KK\{a\otimes b-b\otimes a\}\subset
X^2[\{a,b\}].\end{equation} $R_E^{\perp}=\KK\{a^*\otimes
b^*+b^*\otimes a^*\}$, and clearly its quadratic dual is
$\Lambda^*$. Its monoidal structure given by
$\nu':\Lambda^*\cdot\Lambda^*\rightarrow \Lambda^*$, $\nu'$ being
the concatenation of wedge monomials.
\\Consider the species $E_{j^+}$, for a fixed integer $j$, $j\geq
1$. $E_{j^+}$ is an $E$-module, $\tau:E_{j^+}\cdot E\rightarrow
E_{j^+}$ being the restriction of $\nu$. Similarly, the species
$\Lambda_{j^+}$ is a $\Lambda$-module, with
$\tau':\Lambda_{j^+}\cdot\Lambda\rightarrow \Lambda_{j^+}$ being
the restriction of $\nu'$.
\end{ex}

We have the following proposition.
\begin{prop}\label{hookmodules}
For every $j\geq 1$ we have that $E_{j^+}$ is a quadratic
$E$-module whose dual (a $\Lambda^*$-module) is isomorphic to the
following sum of Specht  representations with hook shapes,
\begin{equation}
E_{j^+}^{!.}=\sum_{k\geq 0}\mathcal{S}_{(j,1^{k})}.
\end{equation}
Moreover,
$(E_{j^+}^{!.})^{\underline{k}}=\mathcal{S}_{(j,1^{k})}$. In a
similar way, $\Lambda_{j^+}$ is a quadratic $\Lambda$-module. Its
dual (an $E^*$-module) is isopmorphic to a sum of Specht
representations with hook shapes, as follows
\begin{equation}\label{hookdual}
\Lambda_{j^+}^{!.}=\sum_{k\geq 0}\mathcal{S}_{(k,1^j)},
\end{equation}
\noindent $(\Lambda_{j^+}^{!.})^{\underline{k}}$ being isomorphic
to $\mathcal{S}_{(k,1^{j})}$.\end{prop}
\begin{proof}

The generator of $E_{j^+}$ is obviously the species $E_j$. Let $U$
be a set of cardinal $j+1$. By the definition of product of tensor
species, $( E_{j}\cdot X)[U]$ is the vector space with basis $$\{
(U-\{a\})\otimes a|a\in U\}.$$ Since
$\tau^{\underline{1}}((U-\{a\})\otimes a)=U-\{a\}\cup \{a\}=U$,
for every $a\in U$, $R_{E_{j^+}}^{\perp}[U]$ is the
one-dimensional vector space generated by
\begin{equation}\label{garnir}\sum_{a\in U}
 (U-\{a\})^*\otimes a^*.\end{equation}
 The consequence of relation (\ref{garnir}) is the Garnir element of a (dual)
 Specht module of shape $(j,1)$. The figure \ref{Garnir} is a tableau form
 of the consequences of equation (\ref{garnir}) for $j=3$ and $U=\{a,b,c,d\}$.
\begin{figure}
\begin{center}
% Generated with LaTeXDraw 2.0.0
% Mon Sep 15 09:39:46 VET 2008
% \usepackage[usenames,dvipsnames]{pstricks}
% \usepackage{epsfig}
% \usepackage{pst-grad} % For gradients
% \usepackage{pst-plot} % For axes
\scalebox{1.5} % Change this value to rescale the drawing.
{
\begin{pspicture}(0,-0.42)(7.381875,0.42)
\definecolor{color157}{rgb}{0.8,0.0,1.0}
\usefont{T1}{ptm}{m}{n}
\rput(0.2309375,-0.1915625){a*}
\usefont{T1}{ptm}{m}{n}
\rput(0.5975,0.21){c}
\psframe[linewidth=0.04,dimen=outer](0.38,0.0)(0.0,-0.38)
\psline[linewidth=0.04cm,linecolor=color157](0.4,0.0)(0.4,0.36)
\psline[linewidth=0.04cm,linecolor=color157](0.8,-0.02)(0.8,0.38)
\psframe[linewidth=0.04,dimen=outer](2.16,0.02)(1.78,-0.36)
\psline[linewidth=0.04cm,linecolor=color157](2.18,0.0)(2.18,0.4)
\psline[linewidth=0.04cm,linecolor=color157](2.62,0.0)(2.62,0.4)
\psframe[linewidth=0.04,dimen=outer](3.96,0.02)(3.58,-0.36)
\psline[linewidth=0.04cm,linecolor=color157](3.98,0.02)(3.98,0.38)
\psline[linewidth=0.04cm,linecolor=color157](4.38,0.0)(4.38,0.4)
\psframe[linewidth=0.04,dimen=outer](5.78,0.02)(5.4,-0.36)
\psline[linewidth=0.04cm,linecolor=color157](5.8,0.02)(5.8,0.38)
\psline[linewidth=0.04cm,linecolor=color157](6.2,0.0)(6.2,0.4)
\usefont{T1}{ptm}{m}{n}
\rput(0.20421875,0.21){b}
\usefont{T1}{ptm}{m}{n}
\rput(2.8045313,0.21){d}
\usefont{T1}{ptm}{m}{n}
\rput(3.81125,-0.19){c*}
\usefont{T1}{ptm}{m}{n}
\rput(1.4614062,0.05){$+$}
\usefont{T1}{ptm}{m}{n}
\rput(6.9714065,0.11){$=0$}
\usefont{T1}{ptm}{m}{n}
\rput(3.2814062,0.05){$+$}
\usefont{T1}{ptm}{m}{n}
\rput(5.081406,0.05){$+$}
\psframe[linewidth=0.04,linecolor=color157,dimen=outer](1.24,0.4)(0.0,-0.04)
\psframe[linewidth=0.04,linecolor=color157,dimen=outer](3.02,0.42)(1.78,-0.02)
\psframe[linewidth=0.04,linecolor=color157,dimen=outer](4.82,0.42)(3.58,-0.02)
\psframe[linewidth=0.04,linecolor=color157,dimen=outer](6.64,0.42)(5.4,-0.02)
\usefont{T1}{ptm}{m}{n}
\rput(2.0078125,-0.19){b*}
\usefont{T1}{ptm}{m}{n}
\rput(4.206406,0.21){b}
\usefont{T1}{ptm}{m}{n}
\rput(6.0065627,0.21){b}
\usefont{T1}{ptm}{m}{n}
\rput(2.388125,0.21){c}
\usefont{T1}{ptm}{m}{n}
\rput(6.3879685,0.21){c}
\usefont{T1}{ptm}{m}{n}
\rput(4.6045313,0.21){d}
\usefont{T1}{ptm}{m}{n}
\rput(5.621406,-0.19){d*}
\usefont{T1}{ptm}{m}{n}
\rput(1.0045313,0.21){d}
\usefont{T1}{ptm}{m}{n}
\rput(1.993125,0.2084375){a}
\usefont{T1}{ptm}{m}{n}
\rput(3.793125,0.2084375){a}
\usefont{T1}{ptm}{m}{n}
\rput(5.593125,0.2084375){a}
\end{pspicture}
}

\end{center}\caption{Garnir relation on
$\mathcal{S}_{(3,1)}[\{a,b,c,d\}]$.}\label{Garnir}
\end{figure}
Then, $(E_j X)^*/R_{E_{j^+}}^{\perp}=\mathcal{S}_{(j,1)}$. For a
set $U$, $|U|=k+j$,  $( E_j\cdot X^{k})[U]$ is the vector space
generated by vectors of the form $ U_2\otimes l$, where
$l=(u_1,u_2,\dots,u_k)$ is a linear order of length $k$ over a set
$U_2$, $|U_2|=k$, $U_1$ is
 a set of cardinal $j$ and $U_1\uplus U_2=U$. Recall that
 $\mathcal{R}_E^{\perp}=\mathcal{R}_{\Lambda^*}$
 is the species of anticommuting relations. The space of relations
 $\mathcal{R}^{\underline{k}}_{\Lambda^*,E^{!.}_{j^+}}[U]$ is
 generated by
 \begin{eqnarray}
  U_2^*\otimes (u_1,\dots,u_i,u_{i+1},\dots,u_k)^*&+&
 U_2^*\otimes (u_1,\dots,u_{i+1},u_i,\dots,u_k)^*\mbox{, $i=1,2\dots,k-1$}\\
 U_2^*\otimes(u_1,\dots,u_k)^*&+&\sum_{u\in
U_2}(U_2\cup \{u_1\}-\{u\})^*\otimes (u,u_2,\dots,u_k)^*
 \end{eqnarray}

It is not difficult to see that $( E_j\cdot
X^{k})^*[U]/(\mathcal{R}^{\underline{k}}_{\Lambda^*,E^{!.}_{j^+}}[U])$
is isomorphic to the Specht module of hook shape
  $(j,1^{k})$. Equivalently, $(E_{j^+}^{!.})^{\underline{k}}$ is
  isomorphic to $\mathcal{S}_{(j,1^k)}$.

Equation (\ref{hookdual}) is proved in a similar way.\end{proof}

  \subsection{\label{Koszul.}Bar and Cobar constructions}
 Let $M=\mathcal{M}(F,R)$ be a quadratic monoid. Let
$M^{\grader}$ be the associated graded monoid.
 Consider the free $\gm$-comonoid generated by $\mil
(M^{\grader})$

\be\label{inverba}\Ba(M):=\LL(\mil
M^{\grader}_+)=(M^{\grader})^{\mil 1}.\eeq For a quadratic
$M$-module $N=\mo_M(G,R_N)$ and $N^{\grader}$ its associated
graded $M^{\grader}$-module, define the $\gm$-species $\Ba(M,N)$
by \be\label{inverbamod}\Ba(M,N):= N^{\grader}\cdot\LL(\mil
M^{\grader}_+)=N^{\grader}(M^{\grader})^{\mil 1}.\eeq

From equations (\ref{inverba}) and (\ref{inverbamod}) we get
\be\label{grado}
\Ba(M)^{\underline{k}}=\sum_{r=1}^{\infty}\left(\sum_{\j_1+j_2+\dots+j_r=k+r}M_+^{\underline{j_1}}
M_+^{\underline{j_2}}\dots M_+^{\underline{j_r}}\right),\mbox{
}k\geq 0,\eeq \noindent and \be
\Ba(M,N)^{\underline{k}}=\sum_{r=1}^{\infty}\left(\sum_{\j_1+j_2+\dots+j_r+j_{r+1}=k+r}
N^{\underline{j_1}} M_+^{\underline{j_2}}\dots
M_+^{\underline{j_r}}M_+^{\underline{j_{r+1}}}\right)\mbox{,
}k\geq 0.\eeq Hence $\Ba(M)^{\underline{0}}=\LL(F)$ and
$\Ba(M,N)^{\underline{0}}= G\cdot\LL(F)$.

 We now transform $\Ba(M)$ and $\Ba(M,N)$ into $\dg$-tensor species, by
providing them with differentials. For a fixed finite set $U$ let
us define $d_U:\Ba(M)[U]\rightarrow \Ba(M)[U]$ and
$d_U^{N}:\Ba(M,N)\rightarrow \Ba(M,N)$ by
\begin{equation}\label{Badif} d_U(m_{1}\otimes
m_{2}\otimes\dots\otimes m_{r})=\sum_{i=1}^{r-1}
(-1)^{i-1}m_{1}\otimes m_{2}\otimes\dots\otimes \nu(m_{i}\otimes
m_{i+1})\otimes\dots\otimes m_{r}\end{equation}\begin{equation}
d_U^N(n_1\otimes m_2\otimes\dots\otimes m_{r+1})=\tau(n_1\otimes
m_2)\otimes\dots \otimes m_{r}\otimes m_{r+1}-n_1\otimes
d_U(m_{2}\otimes\dots\otimes m_{r+1}),
\end{equation} \noindent
where $m_{1}\otimes m_{2}\otimes\dots\otimes m_{r}$, is a generic
decomposable element of $\Ba(M)^{\underline{k}}[U]$ for some
$k\geq 0$, and similarly with respect to $n_1\otimes m_2\otimes
\dots\otimes m_{r+1}. $\\ The cobar $\dg$-tensor species $\Cob(M)$
and  $\Cob(M,N)$ are defined to be respectively the dual of
$\Ba(M)$ and $\Ba(M,N)$;
\begin{equation}
\Cob(M)=\Ba(M)^*=(\LL(\mir(M^{\grader})^*),d^*)=(((M^{\grader})^*)^{\mir
1},d^*)\end{equation}\begin{equation}
\Cob(M,N)=\Ba(M,N)^*=((N^{\grader})^*\cdot\LL(\mir(M^{\grader})^*)
,(d^N)^*)=((N^{\grader})^*(\cdot(M^{\grader})^*)^{\mir
1},(d^N)^*)\end{equation}

The generating functions of $\Ba(M)$, $\Ba(M,N)$, $\Cob(M)$ and
$\Cob(M,N)$ are
\begin{eqnarray}\label{genba}
\Ba(M)(x)=\Cob(M)(x)&=&\left(M^{\grader}(x)\right)^{-1}\\
\label{chba}\Ch\Cob(M)(\xx)= \Ch\Ba(M)(\xx)&=&\left(\Ch
M^{\grader}(\xx)\right)^{-1}\\ \label{genbamod}
\Ba(M,N)(x)=\Cob(M,N)(x)&=&\frac{N^{\grader}(x)}{M^{\grader}(x)}\\
\label{chbamod} \Ch\Ba(M,N)(\xx)=\Ch\Cob(M,N)(\xx)&=&\frac{\Ch
N^{\grader}(\xx)}{\Ch M^{\grader}(\xx)}.
\end{eqnarray}

From equation (\ref{grado}), $\Cob(M)^{\underline{-1}}=\sum_{n\geq
2}\sum_{i=0}^{n-2}(F^*)^{i}(M^{\underline{2}})^*(F^*)^{n-i-2}$,
and $\Cob(M)^{\underline{0}}=\LL(F^*)$. From that, the image of
$(d)^*_{-1}$ is equal to \be\mathrm{Im}(d)^*_{-1}=\sum_{n\geq
2}\sum_{i=0}^{n-2}(F^*)^{i}\nu^*(M^{\underline{2}})^*(F^*)^{n-i-2}=\sum_{n\geq
2}\sum_{i=0}^{n-2}(F^*)^{i}R^{\perp}(F^*)^{n-i-2}=\langle
R^{\perp}\rangle\subseteq \LL(F^*)\eeq \noindent and we have
$H^{0} (\Cob(M))=\LL(F^*)/\langle R^{\perp}\rangle=M^{!.}$. In a
similar way, for $N$ as above, we get  $H^0(\Cob(M,N))=N^{!.}$. By
duality we have $H^0(\Ba(M))=(M^{!.})^*=M^{\coop .}$,
$H^0(\Ba(M,N))=(N^{!.})^*=N^{\coop .}$.

Observe that the bar and cobar complexes are the sum of
subcomplexes $$\Ba(M)=\sum_k\Ba^{(k)}(M)\mbox{,  }
\Ba(M,N)=\sum_k\Ba^{(k)}(M,N),$$ $$\Cob(M)=\sum_k
\Cob^{(k)}(M)\mbox{,  } \Cob(M,N)=\sum_{k}\Cob^{(k)}(M,N).$$ For
example, the vectors in $\Ba^{(k)}(M)$ (respectively
$\Ba^{(k)}(M,N)$) are those of degree $k$ before the shifting:
\begin{eqnarray} 0 &\rightarrow & F^k \stackrel{d}{\rightarrow}
\sum_{i=0}^{k-2}F^i
M^{\underline{2}}F^{k-i-2}\stackrel{d}{\rightarrow}\dots\stackrel{d}{\rightarrow}
M^{\underline{k}}\rightarrow 0,\\
 0
&\rightarrow &  G\cdot F^k \stackrel{d}{\rightarrow}
N^{\underline{1}}F^{k-1}+\sum_{i=0}^{k-2}G\cdot F^i
M^{\underline{2}}F^{k-i-2}\stackrel{d}{\rightarrow}\dots\stackrel{d}{\rightarrow}
N^{\underline{k}}\rightarrow 0.
\end{eqnarray} Clearly
$H^0\Ba^{(k)}(M)=(M^{\coop .})^{\underline{k}}$ and
$H^0\Ba^{(k)}(M,N)=(N^{\coop .})^{\underline{k}}.$
\begin{defi}{\em Koszul Monoids and modules.} A quadratic monoid $M$
is said to be {\em Koszul} if $H^{i}(\Ba(M))=0,\; i>0,$
(equivalently, if $H^{i}(\Cob(M))=0,\; i<0$). In other words, if
the $\gm$-tensor species $\Ho\Ba(M)$ and $\Ho\Cob(M)$ are both
concentrated in degree zero. For a Koszul monoid $M$, a quadratic
$M$-module $N$ is said to be Koszul if $\Ho\Ba(M,N)$ is
concentrated in degree zero (equivalently, if $\Ho\Cob (M,N)$ is
concentrated in degree zero).
\end{defi}
\begin{prop}\label{genk}Let $M$ be a Koszul monoid and $N$ a Koszul $M$-module. Then
\begin{eqnarray}
M^{!.}(x)&=&(M^{\grader}(x))^{-1}\\\Ch M^{!.}(\xx)&=&(\Ch
M^{\grader}(\xx))^{-1}\\
N^{!.}(x)&=&\frac{N^{\grader}(x)}{M^{\grader}(x)}\\
\Ch N^{!.}(\xx)&=&\frac{\Ch N^{\grader}(\xx)}{\Ch
M^{\grader}(\xx)}
\end{eqnarray}
\end{prop}
\begin{proof}The result follows from equations (\ref{genba})-(\ref{chbamod}),
and (\ref{preservation}).
\end{proof}

The previous proposition gives the following necessary conditions
for a quadratic monoid to be Koszul.
\begin{prop}Let $M$ be a quadratic monoid. If $M$ is Koszul then
the coefficients of the exponential generating function
$(M^{\grader}(x))^{-1}$ are non-negative integers. More generally,
the coefficient of the symmetric function $(\Ch
M^{\grader}(\xx))^{-1}$, when expanded in terms of Schur
functions, are non-negative integers.
\end{prop}

\subsection {Koszul analytic algebras}
 We study here the connection of Koszul monoids in species with
 classical Koszul algebras by means of the Schur correspondence between tensor species
 and analytic functors. This correspondence is an
 equivalence of categories when the chacteristic of the field $\K$ is zero.

  Let $\widetilde{F}():\vect\rightarrow
\mathcal{V}$ be the analytic functor associated to a species
$F:\BB\rightarrow \mathcal{V}$ ($\mathcal{V}=\vect,\, \mgr,\,
\mbox{or }$ $\dgr$); \be
\widetilde{F}(V)=\bigoplus_{n=0}^{\infty}\left( F[n]\otimes
V^{\otimes n}\right)_{S_n}.\eeq \noindent It is well known that
the $\widetilde{F.G}(V)=\widetilde{F}(V)\otimes\widetilde{G}(V)$.
Then, a quadratic monoid $(M,\nu)$ defines an analytic functor
$(\widetilde{M},\tilde{\nu})$, with a structure of graded algebra
when evaluated on a finite dimensional vector space $V$, \be
\widetilde{\nu}_{\scriptscriptstyle{V}}:\widetilde{M}(V)\otimes\widetilde{M}(V)\rightarrow
\widetilde{M}(V).\eeq
 Assuming that $\widetilde{M^{\underline{k}}}(V)$ is finite dimensional for every $k\geq 0$,
 its Hilbert series is defined by
 \be \widetilde
{M}(V,t)=\sum_{k=0}^{\infty}\mathrm{dim}\widetilde{M^{\underline{k}}}(V)
t^k.\eeq \noindent  Since
$\mathrm{dim}\widetilde{M^{\underline{k}}}(V)=\Ch
M^{\underline{k}}(\overbrace{1,1,\dots,1}^n,0,0,\dots)$,
$n=\mathrm{dim}V$ (see \cite{Macdonald}), we have

\be \widetilde{M}(V,t)=\sum_{k=0}^{\infty}\Ch
M^{\underline{k}}(\overbrace{1,1,\dots,1}^n,0,0,\dots)t^k. \eeq

\noindent

A species $F$ is called {\em polynomial} if the number of integers
$k$, such that $F[k]\neq 0$, is finite.
 In this section we assume that all the quadratic monoids and modules are generated by polynomial species.
 The reader may verify that in this case, for every finite dimensional space $V$, the corresponding
 algebra $(\widetilde{M}(V),\tilde{\nu}_V)$ is quadratic with
 generators in $W=\widetilde{F}(V)$ (a finite dimensional vector
 space), and relations in $\widetilde{R}_M(V)\subseteq W\otimes
 W=\widetilde{F}(V)\otimes\widetilde{F}(V)$. In a similar way,
 for every quadratic $M$-module $N=\mo_M(G,R_M)$, the corresponding $\widetilde{M}(V)$-module
 $\widetilde{N}(V)$ is quadratic with generators
 in $Q=\widetilde{G}(V)$, and relations in $R_N(V)\subseteq
 \widetilde{F}(V)\otimes \widetilde {G}(V)=W\otimes Q$.

The bar and cobar constructions in section \ref{Koszul.} become
the classical bar and cobar constructions for the algebra
$(\widetilde{M}(V),\tilde{\nu}_{\scriptscriptstyle{V}})$ and
module $(\widetilde{N}(V),\tilde{\tau}_V)$. Similarly with respect
to the definitions of Koszul algebra, Koszul module, and quadratic
duals (see \cite{Priddy, quadal}).

We invoke Schur classical equivalence between representations of
the symmetric groups and homogeneous polynomial representations of
the general linear group (see \cite{Macdonald}). It is well known
that this equivalence can be restated by saying that the category
of tensor species and the category of analytic functors are
equivalent (see \cite{Joyal1}). So, we can go backwards in this
construction and obtain the following proposition.
\begin{prop} \label{translation} The monoid $(M,\nu)$ is quadratic
(resp. Koszul) if and only if
$(\widetilde{M}(V),\tilde{\nu}_{\scriptscriptstyle{V}})$ is a
quadratic (resp. Koszul) algebra, for every finite dimensional
vector space $V$. In a similar way, an $M$-module $(N,\tau)$ is
quadratic (resp. Koszul) if and only if
$(\widetilde{N}(V),\tilde{\tau}_V)$ is a quadratic (resp. Koszul)
$\widetilde{M}(V)$-module for every finite dimensional vector
space $V$.
\end{prop}
Quadratic algebras $A$ and $A^!$ are Koszul simultaneously (see
\cite{quadal} Corollary 3.3.). By the previous proposition we
obtain the following.
\begin{prop} Let $M$ be a quadratic monoid. Then, $M$ is Koszul
if and only if its quadratic dual $M^!$ is so.
\end{prop}
\begin{defi} For a monoid $(M,\nu)$ we call the pair
$(\widetilde{M},\tilde{\nu})$ an {\em analytic algebra}. It will
be called quadratic (resp. Koszul) if the corresponding monoid
$(M,\nu)$ is quadratic (resp. Koszul).
\end{defi}

\begin{ex}
Consider the monoids $\LL$,  $E$ and $\Lambda$. The corresponding
analytic algebras are the tensor $(\TT)$, symmetric $(\Sy)$ and
exterior $(\bigwedge)$ algebras. The three of them are Koszul.
Then $\LL$, $E$ and $\Lambda$ are Koszul monoids. The Koszul dual
of $E$ is $\Lambda^*$, and the Koszul dual of $\LL$ is
$1^*+X^*=1+X$.
\end{ex}

\begin{defi}
Let $M^{\grader}$ be  a positively graded monoid and $k$ a
positive integer. We define the Veronese power $M_{(k)}$ as the
tensor species $M_{(k)}=\sum_{j=0}^{\infty}M^{\underline{kj}}$.
$M_{(k)}$ inherits the monoidal structure from $M^{\grader}$ and
is grading is given by
$M_{(k)}^{\underline{r}}=M^{\underline{kr}}$. Let $N^{\grader}$ be
a graded $M^{\grader}$-module. For every $k\geq 0$ we define the
truncated $M$-module $N^{[k]}$ as
$(N^{[k]})^{\underline{j}}=N^{\underline{k+j}}$.
\end{defi}
\begin{prop}\label{Veronese}Let $M$ and $N$ be as above. If the monoid
$M$ is quadratic (Koszul),
 then $M_{(k)}$ is quadratic (respectively Koszul). Similarly with
 respect to $N$, if the $M$-module $N$ is quadratic (Koszul), then
 $N^{[k]}$ is quadratic (Koszul) for each $k\geq 0$.
\end{prop}
\begin{proof}The Veronese power of a quadratic (Koszul) algebra is
quadratic (Koszul) (see \cite{Backelin}). Let $A$ be an algebra
and $B$ and $A$-module. Then if $A$ and $B$ are quadratic or
Koszul, then the same is true for $B^{[k]}$ (see \cite{quadal},
proposition 1.1.). The result then follows by proposition
\ref{translation}.
\end{proof}
By this proposition, the modules in proposition \ref{hookmodules}
are Koszul.

\begin{defi}{\tbf{Segre and Manin products of monoids}}. Let $M^{\grader}$
and $N^{\grader}$ be two graded monoids. The Segre product $M\seg
N$ is the graded monoid defined by \be (M\seg
N)^{\underline{k}}=M^{\underline{k}}\cdot N^{\underline{k}}\mbox{,
}k\geq 0\eeq \noindent and obvious product. If $M$ and $N$ are
quadratic, the black circle or Manin product is defined as follows
\be M\seb N=(M^{!.}\seg N^{!.})^{!.}. \eeq
\end{defi}
The black circle product for quadratic algebras was introduced by
Manin in \cite{Manin}.
\begin{prop}If $M$ and $N$ are quadratic (resp. Koszul), then
$M\seg N$  is quadratic (resp. Koszul).
\end{prop}
\begin{proof}Similar to the proof of proposition \ref{Veronese}.
See \cite{Backelin} for the analogous result about Koszulness of
the Segre product of Koszul algebras.
\end{proof}
Note that if $M$ and $N$ are quadratic, $M\seg N$ is generated by
$M^{\underline{1}}\cdot N^{\underline{1}}$, and  $$R_{M\seg
N}=R_M\cdot (N^{\underline{1}})^2+(M^{\underline{1}})^2\cdot
R_N.$$ As a consequence, $$R_{M\seb N}=R_M\cdot R_N.$$

\section{M\"obius species, c-monoids, and Cohen-Macaulyness}

A {\em partially ordered set} or {\em poset} $(P,\leq)$ is a set
together with a partial order. We usually refer to a poset by its
underlying set $P$. If $P$ and $Q$ are two posets we define the
direct product $P\times Q$ as the poset with underlying set the
cartesian product of the respective underlying sets with order
relation  $(p,q)\leq_{P\times Q} (p',q')$ if $p\leq_P p'$ and
$q\leq_Q q'$. The direct sum or coproduct  $P\sqcup Q$ is the
poset with underline set the disjoint union of the respective
underlying sets and such that $x\leq y$ in $P\sqcup Q$ either if
$x,y\in P$ and $x\leq_P y$ or $x,y\in Q$ and $x\leq_Q y$.

Denote respectively by $\mrm{Max}(P)$ and $\mrm{Min}(P)$ the set
of maximal and minimal elements of $P$. When $\mrm{Min}(P)$ has
one element, $\hat{0}$ denotes that unique element. Similarly, if
$\mrm{Max}(P)$ has only one element it is denoted by $\hat{1}$. A
set with $\hat{0}$ and $\hat{1}$ is called {\em bounded}.

An interval $[x,y]$ in $P$, $x,y\in P$ is the subposet
$[x,y]=\{z|x\leq z\leq y\}$. We say that $y$ covers $z$ if
$|[x,y]|=2$ and denote it by $x\prec y$. A {\em chain} is a
totally ordered subset $x_0<x_1<\dots<x_l$ of $P$. Its length is
defined to be $l$. A {\em maximal chain} of an interval $[x,y]$ is
one of the form $$x=x_0\prec x_1\prec\dots\prec x_l=y.$$

A poset is {\em pure} if for every interval $[x,y]$, all the
maximal chains have the same length. A bounded poset that is pure
is called a {\em graded poset}. The length of a maximal chain is
called the rank of $P$ ($\mrm{rk}(P)$).

We follow the conventions in \cite{Vallette} for the definition of
order complexes, homology of posets and Cohen-Macaulay posets. We
refer the reader to \cite{Bjorner} and the more recent notes
\cite{Wachs} for tools and techniques on poset topology.

Denote by $\Delta(P)$ the set of chains $x_0<x_1<\dots<x_l$ of $P$
such that $x_0\in\mrm{Min}(P)$ and $x_l\in\mrm{Max}(P)$.
$\Delta(P)=\uplus_{l}\Delta_l(P)$, $\Delta_l(P)$ being the set of
chains of length $l$. Define
$\partial_l:\KK\Delta_l(P)\rightarrow\Delta_{l-1}(P)$,
\begin{equation*}
\partial_i(x_0<x_1<\dots<x_l)=\sum_{i=1}^{l-1}(-1)^{i-1}(x_0<x_1<\dots<x_{i-1}<x_{i+1}<\dots<x_l).
\end{equation*}
The chain complex $(\KK\Delta(P),\partial)$ is called the {\em
order complex} of $P$. The homology of $P$ with coefficients in
$\KK$ is the homology of the complex $\KK\Delta(P)$. It is denoted
by $H_*(P,\KK)$.
\begin{defi} Let $P$ be a graded poset. It is said to be
Cohen-Macaulay over $\KK$ if for every interval $[x,y]$ of $P$ the
homology is concentrated in top rank: \be H_r([x,y],\KK)=0\mbox{,
for $r\neq \mrm{rk}([x,y])$.}\eeq
\end{defi}

\begin{defi}Let $\Int$ be the category of finite families of finite posets with $\hat{0}$ and
$\hat{1}$, where for two objects $A$ and $B$ of $\Int$, a morphism
$f:A\rightarrow B$ is an isomorphism of posets $f:\coprod_{I\in
A}I\rightarrow \coprod_{I'\in B}I'$. A {\em M\"obius} species is a
functor $P:\BB\rightarrow \Int$. For an object $A\in \Int$, define
the  M\"obius cardinality $|A|_{\mu}=\sum_{I\in A}\mu(I).$ If
$A=\emptyset$, we set $\mu(A)=0$. For a permutation
$\sigma:U\rightarrow U$, let
$\mathrm{Fix}_P[\sigma]=\{I_{\sigma}|I\in M[U]\mbox{,
}P[\sigma](I)=I\}$, where $I_{\sigma}$ is the subposet of elements
of $I$ fixed by $\sigma$; $I_{\sigma}=\{x\in I|P[\sigma]x=x\}$ The
M\"obius cardinal $|\mathrm{Fix}P[\sigma]|_{\mu}|$ only depends on
the cycle type of $\sigma$. This leads to analogous of the
generating functions for ordinary species. \be\Mob
P(x)=\sum_{n=0}^{\infty}|P[n]|_{\mu}\frac{x^n}{n!}, \eeq \be \Ch
P(\xx)=\sum_
{\alpha}|\mathrm{Fix}P[\alpha]|_{\mu}\frac{p_{\alpha}(\xx)}{z_{\alpha}}.\eeq
$P$ is said to be Cohen-Macaulay if for every finite set $U$, all
the posets in $P[U]$ are Cohen-Macaulay.
\end{defi}

Let  $(M,\nu)$ be a monoid in the monoidal category of set-species
$(\fs^{\BB},\cdot, 1)$. Recall that the elements of the set
$(M\cdot M)[U]$ are ordered pairs of the form $(m_{U_1},
m_{U_2})$, $U_1\uplus U_2=U$, where $m_{U_i}\in M[U_i]\mbox{, }
i=1,2$. $(M,\nu)$ is called a c-monoid if it satisfies the left
cancellation law \be
\nu(m_{U_1},m_{U_2})=\nu(m_{U_1},m'_{U_2})\Rightarrow
m_{U_2}=m'_{U_2}.\eeq
\subsection{The posets associated to a $c$-monoid}
 From a c-monoid $(M,\nu)$ we can define a
partially ordered set $$(\biguplus_{U_1\subseteq U}
M[U_1],\leq_{\nu})=((E\cdot M)[U],\nu).$$ for each finite set $U$.
The relation $\leq_{\nu}$ on the set defined by  \be
m_{U_1}\leq_{\nu} m_{U_2} \mbox{ if there exists }m'_{U'_2},\;
U_1\uplus U'_2=U_2 \mbox{ such that
 }\nu(m_{U_1}, m'_{U'_2})=m_{U_2}.
 \eeq
 This posets have a zero the unique element of $M[\emptyset]$. All the elements of
 $M[U]$ are maximal in $(\biguplus_{U_1\subseteq U}
M[U_1],\leq_{\nu})$, but in general, an element $m_1\in M[U_1]$
$U_1\neq U$ could be also maximal. To avoid this situation we
define $\Pom_M[U]$ to be the subposet of $(\biguplus_{U_1\subseteq
U} M[U_1],\leq_{\nu})$ whose set of maximals elements is equal to
$M[U]$,
$$\Pom_M[U]:=\{m_1| \mbox{ such that there exists $m\in M[U]$ with }m_1\leq_{\nu}
m\}.$$ In other words, $\Pom_M[U]=\bigcup_{m\in M[U]}[\hat{0},m]$.

 The partial order of $\Pom_M[U]$ is functorial, for every bijection
 $f:U\rightarrow V$ between finite sets we define $\Pom_M[f]:\Pom_M[U]\rightarrow \Pom_M[V]$
 by
 \be\label{actionP} \Pom_M[f]m_1=M[f_{U_1}]m_1 =(EM)[f]m_1\eeq
 \noindent  where $m_1\in M[U_1]$ and $f_{U_1}$ is the
 restriction of $f$ to $U_1$. From the cancellative property in $c$-monoiods
 we obtain the following theorem and corollaries (see \cite{Julia-Miguel}, theorem
 3.2.)
 \begin{theo}The family of posets $\{\Pom_M[U]|U\mbox{ a finite
 set}\}$ satisfies the following properties:
 \begin{enumerate}
 \item $\Pom_M[f]:\Pom_M[U]\rightarrow\Pom_M[V]$ is an order
 isomorphism.
 \item $\Pom_M[U]$ has a $\hat{0}$ the unique element of
 $M[\emptyset]$.
 \item For a finite set $U_1\subseteq U$, and $m_1\in M[U_1]$ an element of $\Pom_M[U]$,
 the order coideal
 $$\mathcal{C}_{m_1}=\{m_2\in\Pom_M[U]|m_2\geq m_1\},$$ is isomorphic to
 $\Pom_M[U-U_1]$.
 \end{enumerate}
\end{theo}
\begin{cor}If $f:U\rightarrow V$ is a bijection, and $[\hat{0},m]$, $m\in M[U]$,  an
interval of $\Pom_M[U]$, then, the restriction of $\Pom_M[f]$ to
$[\hat{0},m]$, $\Pom_M[f]|_{[\hat{0},m]}:[\hat{0},m]\rightarrow
[\hat{0},M[f]m]$ is an isomorphism of posets.
\end{cor}

\begin{cor}\label{important}
Every interval $[m_1,m_2]$  of $\Pom_M[U]$, $m_i\in M[U_i]$,
$i=1,2$, $U_1\subseteq U_2\subseteq U$, is isomorphic to the
interval $[\hat{0},m_2']$ of $\Pom_M[U_2-U_1]$, $m_2'$ being the
unique element of $M[U_2-U_1]$ such that $\nu(m_1,m_2')=m_2.$
\end{cor}

  Now we can define the inverse M\"obius species $M^{-1}.$
 \begin{defi} Let $M$ be a c-monoid a s above. Define the M\"obius species
 $$M^{-1}[U]=\{[\hat{0},m]|m\in M[U]\},$$
 for a bijection $f:U\rightarrow U'$, the isomorphism
 $M^{-1}[f]:\coprod_{m\in M[U]}[\hat{0},m]\rightarrow
\coprod_{m'\in M[U']}[\hat{0},m']$ is given by
 \be \label{action}M^{-1}[f]m_1=M[f_{U_1}]m_1.\eeq
 where $U_1$ is the subset of $U$ such that  $m_1\in M[U_1]$. It
 is clear that if $m_1\in[\hat{0},m]$, then $M^{-1}[f]m_1\in
 [\hat{0},M[f]m]$.
 \end{defi}
 By M\"obius inversion in the poset $\Pom_M[U]$ we obtain (see
 \cite{Julia-Miguel}, theorem 3.3.)
 \be \label{Mobgen}\Mob M^{-1}(x)=(M(x))^{-1}.
\eeq Let $\Pom_{M}[U,\sigma]$ be the subposet of $\Pom_M[U]$ of
the elements fixed by the permutation $\sigma:U\rightarrow U$,
under the action given by equation (\ref{actionP}). In a similar
way, by M\"obius inversion on $\Pom_{M}[U,\sigma]$, we obtain
\be\label{MobChar} \Ch M^{-1}(\xx)=(\Ch M(\xx))^{-1}.\eeq

\begin{prop}
Assume that $M=\mathcal{M}(F,R)$ is a quadratic $c$-monoid. Then,
every interval $[\hat{0},m]$ with $m\in M^{\underline{k}}[U]$ is
graded with rank $k$.
\end{prop}
\begin{proof}
If $m\in M^{\underline{k}}[U]$, then, there exist $f_1,
f_2,\dots,f_k$ elements of $F$ such that $$m=\overline{f_1\otimes
f_2\otimes\dots\otimes f_k}\in F^{k}[U]/\langle R\rangle.$$ Every
maximal chain in $[\hat{0},m]$ has to be of the form
$$\hat{0}\prec m_1\prec m_2\prec\dots\prec m_k=m$$
where $m_i=\nu(m_{i-1},f_i')$ for $i=1,2,\dots,k$ and $f'_1\otimes
f'_2\otimes\dots\otimes f'_k\in\overline{f_1\otimes
f_2\otimes\dots\otimes f_k}.$\end{proof} From the previous
proposition we easily obtain the corollary.
\begin{cor}The poset $\Pom_M^{\underline{k}}[U]:=\bigcup_{m\in
M^{\underline{k}}[U]}[\hat{0},m]$, is pure.
\end{cor}

\begin{theo}\label{C-M0}

 Let
$M=\mathcal{M}(F,R)$ be a quadratic c-monoid. Then, the tensor
species $U\mapsto H_k(P^{\underline{k}}_M[U],\KK)$ is isomorphic
to $(M^{\coop .})^{\underline{k}}$. Moreover, $M$ is Koszul if and
only if
 for every finite set $U$ and every $k$, the homology $ H_*(\Pom^{\underline{k}}_M[U],\KK)$
 is concentrated in top rank ;
 $H_r(\Pom^{\underline{k}}_M[U],\KK)=0$, for $r\neq k$.

\end{theo}
\begin{proof} We reverse the enumeration in the complex
$\KK\Delta_*(\Pom^{\underline{k}}_M[U])$ to define an appropriate
$\dg$-tensor species. Let $C^{(k),\grader}$ be defined by \be
C^{(k),\underline{r}}[U]:=\KK
\Delta_{k-r}(\Pom^{\underline{k}}_M[U]), \eeq \noindent where
$r=0,1,2,\dots,k.$ The differential
$\hat{d}_r:C^{(k),\underline{r}}\rightarrow
C^{(k),\underline{r+1}}$ is equal to $\partial_{k-r}.$  Let \be
m_{U_1}\otimes m_{U_2}\otimes\dots\otimes m_{U_r}\in
\Ba^{(k)}(M)^{\underline{j}}[U]\mbox{,   } m_{U_i}\in
M[U_i]\mbox{, }i=1,2,\dots,r.\eeq  Let $V_i=\biguplus_{s=1}^i
U_s$, and define recursively, $$\overline{m}_{V_1}=m_{U_1}\mbox{,
  }\overline{m}_{V_i}=\nu(\overline{m}_{V_{i-1}},m_{U_i})\mbox{,
 $i=2,\dots,r$}.$$ It is clear that
$\hat{0}<_{\nu}\overline{m}_{V_1}<_{\nu}\overline{m}_{V_2}<_{\nu}\dots<_{\nu}\overline{m}_{V_r}$
is a chain in $C^{(k),\underline{k-r}}[U]$. By the left
cancellation law, the correspondence \be \alpha:m_{U1}\otimes
m_{U_2}\otimes\dots\otimes m_{U_r}\mapsto
\hat{0}<_{\nu}\overline{m}_{V_1}<_{\nu}\overline{m}_{V_2}<_{\nu}\dots<_{\nu}\overline{m}_{V_r}\in
M[U] \eeq
 \noindent is bijective between basic elements of
 $\Ba(M)^{(k),\underline{k-r}}[U]$ and $C^{(k),\underline{k-r}}[U]$
 respectively. Since $$ \alpha(m_{U1}\otimes\dots\otimes\nu(m_{U_{i}},m_{U_{i+1}})
 \otimes\dots\otimes
m_{U_r})=
\hat{0}<_{\nu}\overline{m}_{V_1}<_{\nu}\overline{m}_{V_2}<_{\nu}\dots<_{\nu}
\overline{m}_{V_{i-1}}<_{\nu}\overline{m}_{V_{i+1}}<_{\nu}\overline{m}_{V_r}.$$
\noindent $\alpha$ extends by linearity to an isomorphism of
$\dg$-tensor species $\alpha:\Ba^{(k)}(M)\rightarrow C^{(k)}.$
Because of this we have
$H_k(\Pom_M^{\underline{k}}[U],\KK)=H^0(C^{(k),\grader}[U])=H^0\Ba^{(k)}(M)=(M^{\coop
.})^{\underline{k}}$, and that $\Ho\Ba^{(k)}(M)$ is concentrated
in degree zero if and only if $\Ho C^{(k),\grader}$ is so for
every $k=0,1,2,\dots.$\end{proof}
\begin{lem} Let $M=\mathcal{M}(F,R)$ be a quadratic c-monoid as above.
 The inverse species $M^{-1}$ is Cohen-Macaulay if and only if
for every finite set $U$ and every poset $[\hat{0},m]\in
M^{-1}[U]$, the homology $H_*([\hat{0},m])$ is concentrated in top
rank $k=\mathrm{rk}([\hat{0},m])$.
\end{lem}
\begin{proof}
The result follows directly from corollary \ref{important}.
\end{proof}
\begin{theo}\label{C-M}Let $M$ be a c-monoid as above. Then, $M$
is Koszul if and only if the M\"obius species $M^{-1}$ is
Cohen-Macaulay.
\end{theo}\begin{proof}We have the identity
\be H_r(\Pom^{\underline{k}}_M[U],\KK)=\bigoplus_{m\in
M^{\underline{k}}[U]}H_{r}([\hat{0},m],\KK)\mbox{,
 }r=0,1,2,\dots,k.\eeq Then, $H_r(\Pom^{\underline{k}}_M[U],\KK)= 0$ if and only if
$H_r([\hat{0},m],\KK)=0$ for every $m\in M^{\underline{k}}[U]$.
The result follows from the previous lemma and theorem \ref{C-M0}.
\end{proof}
If $M$ is a Koszul quadratic c-monoid, since $((M^{\coop
.})^{\grader}(x))^{-1}=M(x)$ and $((\Ch M^{\coop
.})^{\grader})^{-1}(\xx)=\Ch M(\xx)$, we have
\begin{eqnarray} \Mob
M^{-1}(x)&=&(M^{\coop
.})^{\grader}(x),\\
\Ch M^{-1}(\xx)&=& \Ch(M^{\coop .})^{\grader}(\xx).
\end{eqnarray}
Moreover, these identities can be refined a little more. Using
standard arguments, we can prove that the M\"obius cardinalities
have the following interpretations:
\begin{eqnarray}\label{mdim}\sum_{m\in
M^{\underline{k}}[n]}\mu(\hat{0},m)&=&(-1)^k \mathrm{dim}(M^{\coop .})^{\underline{k}}[n],\\
\label{mtr}\sum_{m\in
M^{\underline{k}}[n],M[\sigma]m=m}\mu([\hat{0},m]_{\sigma})&=&(-1)^k\mathrm{tr}(M^{\coop
.})^{\underline{k}}[\sigma].\end{eqnarray}
\begin{ex}For the  c-monoid $E$, $\Pom_E[U]=[\emptyset, U]=\mathscr{B}(U)$ is
the Boolean algebra of the parts of $U$. It is well known that
$\mathscr{B}(U)$ is Cohen-Macaulay. This is another way of
obtaining that $E$ is a Koszul monoid (over $\mathbb{Z}$).
\end{ex}
\subsection{The posets associated to a c-module}
\begin{defi} Let $M$ be a c-monoid and $N$ a right (set) $M$-module
such that $N[\emptyset]=\emptyset$.
 $N$ is said to be a c-($M$)-module if it satisfies the left
cancellation law: $\tau(n,m)=\tau(n,m')\Rightarrow m=m'$.
\end{defi}
Define a partial order on $E.N[U]=\biguplus_{U_1\subseteq
U}N[U_1]$ by: $n_1\leq_{\tau} n_2$ if there exists $m\in
M[U_2-U_1]$ such that $\tau(n_1,m)=n_2$. $E.N[U]$ does not have  a
$\hat{0}$. In that case define $\widehat{E.N}[U]=\hat{0}+E.N[U]$,
then trim the maximal elements not in $N[U]$, to obtain \be
\Pom_{M,N}[U]=\bigcup_{n\in N[U]}[\hat{0},n]\eeq \noindent
$[\hat{0},n]$ being an interval of $\widehat{E.N}[U]$. As in the
case of posets induced by $c$-monoids, the order is functorial,
and we have the following theorem and corollaries.

 \begin{theo}The family of posets $\{\Pom_{M,N}[U]|U\mbox{ a finite
 set}\}$ satisfies the following properties:
 \begin{enumerate}
 \item For a bijection $f:U\rightarrow V$, $\Pom_{M,N}[f]:\Pom_{M,N}[U]\rightarrow
 \Pom_{M,N}[V]$ is an order
 isomorphism.

 \item For a finite set $U_1\subseteq U$, and $n_1\in N[U_1]$ an element of $\Pom_{M,N}[U]$,
 the order coideal
 $$\mathcal{C}_{n_1}=\{n_2\in\Pom_{M,N}[U]|n_2\geq n_1\},$$ is isomorphic to
 $\Pom_{M,N}[U-U_1]$.
 \end{enumerate}
\end{theo}
\begin{cor}If $f:U\rightarrow V$ is a bijection, and $[\hat{0},n]$, $n\in N[U]$,  an
interval of $\Pom_{M,N}[U]$, then, the restriction of
$\Pom_{M,N}[f]$ to $[\hat{0},n]$,
$\Pom_{M,N}[f]|_{[\hat{0},n]}:[\hat{0},n]\rightarrow
[\hat{0},N[f]n]$ is an isomorphism of posets.
\end{cor}

\begin{lem} Let $[n_1,n_2]$ be an
interval of $\Pom_{M,N}[U]$, $n_i\in N[U_i]$, $U_1\subseteq
U_2\subseteq U$. Then, $[n_1,n_2]$ is isomorphic to $[\hat{0},m]$,
$m\in M[U_2-U_1]$ being the unique element such that
$\tau(n_1,m)=n_2$.
\end{lem}

\begin{theo}Let $Q_{N,M}$ be the M\"obius species defined as \be
Q_{N,M}[U]=\{[\hat{0},n]|n\in N[U]\}.\eeq \noindent Then
\begin{eqnarray}\Mob Q_{N,M}(x)&=&-N(x)\Mob
M^{-1}(x)=-\frac{N(x)}{M(x)}\\
\Ch Q_{N,M}(\xx)&=&-\frac{\Ch N(\xx)}{\Ch M(\xx)}
\end{eqnarray}
\end{theo}
\begin{proof} \begin{eqnarray} \sum_{n\in N[U]}\sum_{\hat{0}\leq n_1\leq n}
\mu(n_1,n)&=&\sum_{n\in N[U]}\mu(\hat{0},n)+\sum_{n\in
N[U]}\sum_{\hat{0}<n_1\leq n}\mu(n_1,n)=0\\
&=&|Q_{N,M}[U]|_{\mu}+\sum_{n\in N[U]}\sum_{\emptyset\neq
U_1\subseteq
U}\sum_{n_1\in N[U_1]}\mu(n_1,n)=0\\
&=&|Q_{N,M}[U]|_{\mu}+\sum_{\emptyset\neq U_1\subseteq
U}|N[U_1]|\sum_{m\in M[U-U_1]}\mu(\hat{0},m)=0\\
&=&|Q_{N,M}[U]|_{\mu}+\sum_{\emptyset\neq U_1\subseteq
U}|N[U_1]||M^{-1}[U-U_1]|_{\mu}=0.\end{eqnarray} From that, we get
\be \Mob Q_{N,M}(x)+N(x)\Mob M^{-1}(x)=0.\eeq The proof of the
character formula is similar, by M\"obius inversion on the poset
$$\Pom_{M,N}[U,\sigma]=\{n\in \Pom_{M,N}[U]|\Pom_{M,N}[\sigma]n=n\}.$$
\end{proof}
Analogously to the c-monoid case we have the following results.
\begin{prop}The poset $\Pom_{M,N}^{\underline{k}}[U]:=\bigcup_{n\in
N^{\underline{k-1}}[U]}[\hat{0},n]$ is pure with
$\mrm{rk}([\hat{0},n])=k$ for every $n$ maximal in
$\Pom_{M,N}^{\underline{k}}[U]$.
\end{prop}

\begin{theo}\label{C-Mod}

 Let $M$ be a Koszul c-monoid, and
$N$ a quadratic c-module on $M$. Then, the tensor species
$U\mapsto H_k(P^{\underline{k}}_{M,N}[U],\KK)$ is isomorphic to
$(N^{\coop .})^{\underline{k-1}}$. Moreover, $N$ is Koszul if and
only if
 for every finite set $U$ and every $k$, the homology $ H_*(\Pom^{\underline{k}}_{M,N}[U],\KK)$
 is concentrated in top rank;
 $H_r(\Pom^{\underline{k}}_{M,N}[U],\KK)=0$, for $r\neq k$.
\end{theo}

\begin{cor}The $M$-module $N$ is Koszul if and only if the M\"obius species $Q_{N,M}$
is a Cohen-Macaulay. Moreover we have
\begin{eqnarray} \Mob
Q_{N,M}(x)&=&-(N^{\coop
.})^{\grader}(x),\\
\Ch Q_{M,N}(\xx)&=& -\Ch(N^{\coop .})^{\grader}(\xx),\\
\label{moddim}\sum_{n\in
N^{\underline{k}}[n]}\mu(\hat{0},m)&=&(-1)^{k+1} \mathrm{dim}(N^{\coop .})^{\underline{k}}[n],\\
\label{modtr}\sum_{n\in
N^{\underline{k}}[n],N[\sigma]n=n}\mu([\hat{0},n]_{\sigma})&=&(-1)^{k+1}\mathrm{tr}(N^{\coop
.})^{\underline{k}}[\sigma].\end{eqnarray}
\end{cor}
\begin{ex} {\em Truncated Posets.} Let $M=\mathcal{M}(F,R)$ be a Koszul $c$-monoid.
For a fixed positive integer $l$, $M^{[l]}=\sum_{j\geq l}
M^{\underline{j}}$ is a cancellative Koszul $M$-module.
$\Pom_{M,M^{[l]}}[U]$ is the poset obtained from $\Pom_{M}[U]$ by
removing the elements of rank $j$, $0<j< l$. The orthogonal
relation $R_{M^{[l]}}^{\perp}[U]$ is the linear span of elements
of the form $$\sum_{\nu(m_1,f)=m}m_1^*\otimes f^*\in
(M^{\underline{l}})^*\cdot F^*,$$ where in the sum, $m$ is in
$M^{\underline{l+1}}[U]$. For $M=E$, $\Pom_{E,E^{[l]}}[U]$ is the
truncated Boolean algebra $\mathscr{B}^{[l]}[U]$. Its top homolgy
is given in the first part of proposition \ref{hookmodules} (see
\cite{Solomon} for a general result on the top homology of rank
selected Boolean algebras).
\end{ex}
\section{Applications: Andr\'e's alternating sequences, Bessel
functions, and generalizations.} The Andr\'e generating function
for alternating sequences \cite{Andre} is the prototypical example
of the combinatorial applications of Koszul duality for product
monoids and modules. Define the species $\mathrm{Cosh}$ and
$\mathrm{Sinh}$ to be respectively the even and odd parts of $E$,
 \begin{eqnarray}\mathrm{Cosh}&=&\sum_{k=0}^{\infty}
 E_{2k}=E_{(2)}\\
 \mathrm{Sinh}&=&\sum_{k=0}^{\infty}E_{2k+1}=E_{(2)+1}
\end{eqnarray}

$\mathrm{Cosh}$ is a c-monoid (see \cite{Julia-Miguel}, Example
3.8.), and $\mathrm{Sinh}$ is a $\mathrm{Cosh}$-module. Since $E$
is a Koszul monoid and hence a Koszul $E$-module, by proposition
\ref{Veronese}, both are quadratic and Koszul.
$\mathrm{Cosh}=\mathcal{M}(E_2, R_{\mathrm{Cosh}})$, where for a
set $U$ of cardinal $4$, the vector space of relations
$R_{\mathrm{Cosh}}[U]\subset E_2^2[U]$ is generated by the vectors
of the form $U_1\otimes U_2-U_1'\otimes U_2'$,
$|U_1|=|U_2|=|U_1'|=|U_2'|=2$, and such that $U_1\uplus
U_2=U_1'\uplus U_2'=U$. Its annihilator
$R_{\mathrm{Cosh}}^{\perp}[U]\subset (E_2^*)^2[U]$ is the
one-dimensional vector space generated by \be\label{granir3}
\sum_{U_1\uplus U_2=U}U^*_1\otimes U^*_2.\eeq

In a similar way
$\mathrm{Sinh}=\mo_{\mathrm{Cosh}}(X,R_{\mathrm{Sinh}})$, where
$R_{\mathrm{Sinh}}[U]\subset (X\cdot E_2)[U]$, $|U|=3$, is the
vector space generated by the differences $ a\otimes (U-\{a\})
-a'\otimes (U-\{a'\})  \mbox{, }a,a'\in U.$ Its annihilator is the
one-dimensional vector space generated by

\begin{equation}\label{garnir4}\sum_{a\in U}
 a^*\otimes (U-\{a\})^*.\end{equation}

For a finite set $U=\{1,2,\dots,2k\}$ of even cardinal,
$\mathrm{Cosh}^{!.}[U]=\LL(E_2)[U]/\langle
R_{\mathrm{Cosh}}^{\perp}\rangle[U]$ is the vector space generated
by vectors of the form $U_1^*\otimes U_2^*\otimes\dots\otimes
U_k^*$, $\biguplus_{i=1}^k U_i=U$, and $|U_i|=2$,  under the
relations \be\label{garnir3} \sum_{U_i\uplus
U_{i+1}=H}U^*_i\otimes U^*_{i+1}=0,\;H\subseteq U\mbox{, $|H|=4,$
for $i=1,2,\dots, k-1.$}\eeq

 Let $\tau_{n}$ be the border strip $B_{\alpha}$,
$\alpha$ being the composition $(\overbrace{2,2,\dots,2}^k)$ if
$n=2k$ is even and $\alpha=(\overbrace{2,2,\dots,2}^{k},1)$ if
$n=2k+1$ is odd ($B_{\alpha}$ as defined in \cite{Stanley} p.
383). The relations (\ref{garnir3}) are the Garnir elements of the
skew Specht representation whose shape is $\tau_{2k}$. Then \be
\mrm{Cosh}^{!.}=\sum_{k=0}^{\infty}\mathcal{S}_{\tau_{2k}}.\eeq

In the same way we prove that \be
\mrm{Sinh}^{!.}=\sum_{k=0}^{\infty}\mathcal{S}_{\tau_{2k+1}}.\eeq
The basis of $\mathcal{S}_{\tau_{2k}}[2k]$, and
$\mathcal{S}_{\tau_{2k+1}}[2k+1]$ are in bijection with the
standard tableaux of shapes $\tau_{2k}$ and $\tau_{2k+1}$
respectively. They can be identified with the alternating
permutations of length $2k$ and $2k+1$ respectively.

We have the generating functions
\begin{eqnarray}
\mathrm{Cosh}(x)&=&\sum_{k=0}^{\infty}\frac{x^{2k}}{(2k)!}=\cosh(x)\\
\mathrm{Sinh}(x)&=&\sum_{k=0}^{\infty}\frac{x^{2k+1}}{(2k+1)!}=\sinh(x)\\
\mathrm{Cosh}^{\grader}(x)&=&\sum_{k=0}^{\infty}(-1)^k\frac{x^{2k}}{(2k)!}=\cos(x)\\
\mathrm{Sinh}^{\grader}(x)&=&\sum_{k=0}^{\infty}(-1)^k\frac{x^{2k+1}}{(2k+1)!}=\sin(x)\\
\Ch \mrm{Cosh}^{\grader}(x)&=&\sum_{k=0}^{\infty}(-1)^k
h_{2k}(\xx)\\ \Ch
\mrm{Sinh}^{\grader}(x)&=&\sum_{k=0}^{\infty}(-1)^k
h_{2k+1}(\xx)\end{eqnarray}

Denote by $\mathcal{E}_k$ the Euler number, the number of
alternating permutations of length $k$.  From proposition
\ref{genk} we obtain
\begin{eqnarray}
\mrm{Cosh}^{!.}(x)&=&\frac{1}{\cos(x)}=\sec(x)=\sum_{k=0}^{\infty}\mathcal{E}_{2k}\frac{x^{2k}}{2k!}\\
\mrm{Sinh}^{!.}(x)&=&\frac{\sin(x)}{\cos(x)}=\tan(x)=\sum_{k=0}^{\infty}\mathcal{E}_{2k+1}\frac{x^{2k+1}}{(2k+2)!}\\
\Ch\mrm{Cosh}^{!.}(\xx)&=&\frac{1}{\sum_{k=0}^{\infty}(-1)^k
h_{2k}(\xx)}=\sum_{k=0}^{\infty}S_{\tau_{2k}}(\xx)\\
\Ch\mrm{Sinh}^{!.}(\xx)&=&\frac{\sum_{k=0}^{\infty}(-1)^k
h_{2k+1}(\xx)}{\sum_{k=0}^{\infty}(-1)^k
h_{2k}(\xx)}=\sum_{k=0}^{\infty}S_{\tau_{2k+1}}(\xx).
\end{eqnarray}

 All of this can be straightforwardly generalized to the monoid
$E_{(n)}$, and its modules $E_{(n)+r}$, $r=1,2,\dots,n-1$.

Observe that $\Pom_{\mathrm{Cosh}}[U]$ is the poset of subsets of
$U$ having even cardinality. By theorem \ref{C-M},
$\Pom_{\mathrm{Cosh}}[U]$ is Cohen-Macaulay for every $U$. This is
indeed a consequence of a classical result:
$\Pom_{\mrm{Cosh}}[2k]$ is the rank restricted Boolean algebra
$\mathscr{B}_S[2k]$, $\mathscr{B}_S[2k]=\{V||V|\in S\}$,
$S=\{0,2,4,\dots,2k\}$, hence it is Cohen Macaulay (see
\cite{Bjorner} Theorem 5.2).  We have \be\Ch H_k(
P_{\mathrm{Cosh}}[2k],\KK)(\xx)=S_{\tau_{2k}}(\xx).
 \eeq\noindent
Since $\mrm{dim}\mathcal{S}_{\tau_{r}}[r]= \mathcal{E}_{r}$, by
equations (\ref{mdim}) and (\ref{moddim}), we can compute the
M\"obius function
\begin{eqnarray*}\mu(\Pom_{\mathrm{Cosh}}[2k])&=&(-1)^k\mathcal{E}_{2k},\\
\mu(\Pom_{\mathrm{Cosh,Sinh}}[2k+1])&=&(-1)^{k+1}\mathcal{E}_{2k+1}.\end{eqnarray*}
See \cite{Stanley0} Proposition 3.16.4. or
\cite{Stanley-Binomial}, Corollary 3.5., for a q-generalization of
this M\"obius function formula, in the context of binomial posets.

The combinatorics of Bessel functions is the simplest and most
natural application of the Koszulness of the Segre and Manin
products of Koszul monoids. Consider the Segre product of $E$ with
itself, $E\seg E$. The elements of $E\seg E[U]$ are pairs of sets
$(U_1,U_2)$, $U_1\uplus U_2=U$, $|U_1|=|U_2|$. $E\seg E$ is
clearly a c-monoid. Its generating functions are

\begin{eqnarray} (E\seg
E)(x)&=&\sum_{k=0}^{\infty}E_k(x)^2=\sum_{k=0}^{\infty}\frac{x^{2k}}{k!^2}=I_0(2x)\\
\Ch(E\seg E)(\xx)&=&\sum_{k=0}^{\infty}h_k^2(\xx)\\
(E\seg
E)^{\grader}(x)&=&\sum_{k=0}^{\infty}(-1)^k\frac{x^{2k}}{k!^2}=J_0(2x)\\
\Ch(E\seg E)^{\grader}(\xx)&=&\sum_{k=0}^{\infty}(-1)^k
h_k^2(\xx).
\end{eqnarray}
\noindent where $J_0(x)$ and $I_0(x)$ are respectively the
ordinary and hyperbolic Bessel functions. The species of
generators is $X\cdot X=X^2$.
 The species of relations $R[U]\subset X^4[U]=(X^2\cdot X^2)[U]$ is the vector space generated
 by the differences:
\be  l_1\otimes l_2-l_3\otimes l_4\eeq \noindent where $l_1\otimes
l_2, l_3\otimes l_4\in X^2[U_1]\otimes X^2[U_2]$, $U_1\uplus
U_2=U.$ $R^{\perp}$ is the vector space generated by the set\be
\left\{\sum_{ l_1\otimes l_2\in X^2[U_1]\otimes
X^2[U_2]}l^*_1\otimes l^*_2|U_1\uplus U_2=U\right\}\eeq For
$U=\{1,2,3,4\}$, $U_1=\{1,2\}$, $U_2=\{3,4\}$, the relation
becomes the following identity in $(\Lambda\seb\Lambda)[4]=(E\seg
E)^{!.}[4]=(X^2\cdot X^2)[4]/R^{\perp}[4],$

\begin{eqnarray}\label{rr}
&&\overline{(1,2)^*\otimes (3,4)^*}+\overline{(2,1)^*\otimes
(3,4)^*}+\overline{(1,2)^*\otimes (4,3)^*}+\overline{(2,1)^*\otimes(4,3)^*}=0,\\
&&\overline{(1,2)^*\otimes (3,4)^*}=-\overline{(2,1)^*\otimes
(3,4)^*}-\overline{(1,2)^*\otimes
(4,3)^*}-\overline{(2,1)^*\otimes (4,3)^*}.\end{eqnarray}

By this relations, the vector space $(E\seg E)^{!.}[2k]=(X^k\cdot
X^k)[2k]/\mathcal{R}^{\perp}[2k]$ has a basis that can be
identified with the set of elements of the form $l^*_1\otimes
l^*_2\in (X^k\cdot X^k)[2k]$, where in the pair
$(l_1,l_2)=(a_1a_2\dots a_k, b_1b_2\dots b_k)$ the configuration
$a_i<a_{i+1}$, $b_i<b_{i+1}$ is not allowed (forbidden rise-rise
configuration). From that, we recover the result of \cite{Carlitz}
that the coefficient $\mathfrak{f}_{2k}$ in the expansion \be
(E\seg
E)^{!.}(x)=J_0^{-1}(x)=\sum_{k=0}^{\infty}\mathfrak{f}_{2k}\frac{x^{2k}}{k!^2}\eeq
\noindent counts the pairs of linear orders $(l_1,l_2)\in
X^k[k]\times X^k[k]$,  with forbidden rise-rise configuration. A
similar interpretation can be obtained from the coefficients of
the monomials in the expansion of the character of $(E\seg
E)^{!.}$ \be \Ch (E\seg E)^{!.}=\frac{1}{\sum_{k=0}^{\infty}(-1)^k
h_k^2(\xx)}. \eeq

The ordinary generating function of the Manin product $E\seb
E=(\Lambda\seg\Lambda)^{!}$ is also $J_0^{-1}(x)$. Its Frobenius
character is \be \Ch (E\seb
E)(\xx)=\frac{1}{\sum_{k=0}^{\infty}(-1)^k e_k^2(\xx)}.\eeq See
\cite{Carlitz} for a combinatorial interpretation of that. As
before, the species of generators of $\Lambda\seg\Lambda$ is
$X\cdot X=X^2$.

The species of relations $R[U]\subset X^4[U]=(X^2\cdot X^2)[U]$
 is the vector space generated
 by the differences:
\be  l_1\otimes l_2- \mathrm{sig}\sigma.\mathrm{sig}\tau(\sigma
l_1\otimes \tau l_2)\eeq \noindent where $l_1\otimes l_2\in
X^2[U_1]\otimes X^2[U_2]$, $U_1\uplus U_2=U,$ $\sigma$ and $\tau$
being permutations of $U_1$ and $U_2$ respectively. $R^{\perp}$ is
the vector space generated by the set \be \left\{\sum_{
\sigma,\tau }\mathrm{sig}\sigma.\mathrm{sig}\tau (\sigma
l^*_1\otimes \tau l^*_2)|U_1\uplus U_2=U\right\}\eeq For
$U=\{1,2,3,4\}$, $U_1=\{1,2\}$, $U_2=\{3,4\},$ we have the
identity in $(E\seb E)[4]=(\Lambda\seb
\Lambda)^{!.}[4]=$$(X^2\cdot X^2)[4]/R^{\perp}[4],$

\be \overline{(1,2)^*\otimes
(3,4)^*}=-\overline{(2,1)^*\otimes(3,4)^*}-\overline{(1,2)^*\otimes
(4,3)^*}+\overline{(2,1)^*\otimes (4,3)^*}.\eeq

As before, the basis of $(E\seb E)[2k]$ is also in bijection with
pairs of linear orders with forbidden rise-rise configuration.

All of this can be generalized to arbitrarily many Segre and Manin
products of E: $E\seg E\seg\dots\seg E$ and $E\seb E\seb\dots\seb
E$.
\begin{figure}\begin{center}
% Generated with LaTeXDraw 2.0.0
% Sun Sep 21 15:00:32 VET 2008
% \usepackage[usenames,dvipsnames]{pstricks}
% \usepackage{epsfig}
% \usepackage{pst-grad} % For gradients
% \usepackage{pst-plot} % For axes
\scalebox{1} % Change this value to rescale the drawing.
{
\begin{pspicture}(0,-1.858125)(9.762813,1.858125)
\definecolor{color632}{rgb}{0.8,0.0,0.8}
\usefont{T1}{ptm}{m}{n}
\rput(4.842969,-1.6303124){$(\emptyset,\emptyset)$}
\usefont{T1}{ptm}{m}{n}
\rput(4.862344,1.6696875){$(\{a,b\},\{c,d\})$}
\usefont{T1}{ptm}{m}{n}
\rput(1.0723437,-0.0103125){$(\{a\},\{c\})$}
\usefont{T1}{ptm}{m}{n}
\rput(3.5923438,-0.0103125){$(\{a\},\{d\})$}
\usefont{T1}{ptm}{m}{n}
\rput(6.112344,-0.0103125){$(\{b\},\{c\})$}
\usefont{T1}{ptm}{m}{n}
\rput(8.632343,-0.0103125){$(\{b\},\{d\})$}
\psline[linewidth=0.04cm,linecolor=color632](4.7829375,-1.4603125)(1.3629375,-0.2603125)
\psline[linewidth=0.04cm,linecolor=color632](4.7829375,-1.4603125)(3.5429375,-0.2003125)
\psline[linewidth=0.04cm,linecolor=color632](4.8029375,-1.4403125)(6.0829377,-0.2203125)
\psline[linewidth=0.04cm,linecolor=color632](4.8229375,-1.4603125)(8.602938,-0.2203125)
\psline[linewidth=0.04cm,linecolor=color632](1.0629375,0.2196875)(4.8229375,1.4996876)
\psline[linewidth=0.04cm,linecolor=color632](3.5629375,0.1596875)(4.8429375,1.5196875)
\psline[linewidth=0.04cm,linecolor=color632](6.1029377,0.1796875)(4.8429375,1.4996876)
\psline[linewidth=0.04cm,linecolor=color632](8.602938,0.1796875)(4.8429375,1.4996876)
\end{pspicture}
}

\end{center} \caption{Interval of the scale poset $\Pom_{E\seg
E}[\{a,b,c,d\}].$}\label{scale}
\end{figure}

\subsection{The scale posets}
 We call $\Pom_{E\seg E}[U]$ the {\em scale poset}. The elements of it
 are balanced ordered pairs of subsets of $U$, $(V_1,V_2)$,
$|V_1|=|V_2|$, $V_1\cap V_2=\emptyset$, ordered by simultaneous
inclusion: $(V'_1,V'_2)\leq (V_1,V_2)$ if $V'_1\subseteq V_1$ and
$V'_2\subseteq V_2$. The maximal elements are the balanced ordered
pairs of sets $(U_1,U_2)$, $U_1\uplus U_2=U$, and all the posets
$[(\emptyset,\emptyset), (U_1,U_2)]\in
 (E\seg E)^{-1}[U]$ are isomorphic.
 By theorem \ref{C-M}, they are Cohen-Macaulay.
 The character of the top homology $\Ch
H_k(\Pom_{E\seg E}[2k],\KK)(\xx)$ is the homogeneous part of
degree $2k$ in the expansion of the symmetric function
$\left(\sum_{j=0}^{\infty}(-1)^j h_j^2(\xx)\right)^{-1}$. By
equation (\ref{mdim}), the M\"obius cardinality \be |(E\seg
E)^{-1}[2k]|_{\mu}=\sum_{U_1\uplus
U_2=[2k]}\mu([(\emptyset,\emptyset),(U_1,U_2)])=\left(\begin{matrix}2k\\k\end{matrix}\right)
\mu[(\emptyset,\emptyset),([k],[k])] \eeq \noindent is equal to
$(-1)^k\mrm{dim}(E\seg E)^{!.}[2k]
=(-1)^k\left(\begin{matrix}2k\\k\end{matrix}\right)\mathfrak{f}_{2k}$.
Then, \be \mu[(\emptyset,\emptyset),([k],[k])]=(-1)^k
\mathfrak{f}_{2k},\eeq \noindent and we have \be\Mob (E\seg
E)^{-1}(x)=
\frac{1}{I_0(2x)}=\sum_{k=0}^{\infty}(-1)^k\mathfrak{f}_{2k}\frac{x^{2k}}{k!^2}.\eeq

The scale posets can be generalized by considering, for fixed $k$
and $n$, the $c$-monoid $$\mrm{Lib_{(k,n)}}=\overbrace{E_{(k)}\seg
E_{(k)}\seg\dots\seg E_{(k)}}^{n}.$$ The elements of
$\Pom_{\mrm{Lib_{(k,n)}}}[U]$ are $n$-tuples of sets having the
same cardinal, a multiple of $k$, and ordered by simultaneous
inclusion. $\mrm{Lib}^{-1}_{(k,n)}$ is Cohen-Macaulay, and we have
the generating functions
\begin{eqnarray}
\Mob
\mrm{Lib_{(k,n)}}^{-1}(x)&=&\left(\sum_{r=0}^{\infty}\frac{x^{nkr}}{(rk)!^n}\right)^{-1},\;
\mrm{Lib}_{(k,n)}^{\coop .}(x)=\left(\sum_{r=0}^{\infty}(-1)^r\frac{ x^{nkr}}{(rk)!^n}\right)^{-1}\\
\Ch
\mrm{Lib_{(k,n)}}^{-1}(\xx)&=&\frac{1}{\sum_{r=0}^{\infty}h_{kr}^n(\xx)},\;\Ch
\mrm{Lib_{(k,n)}}^{\coop
.}(\xx)=\frac{1}{\sum_{r=0}^{\infty}(-1)^r h_{kr}^n(\xx)}.
\end{eqnarray}

\subsection{c-Operads}

Let $R$ and $S$ be two set-species such that
$S[\emptyset]=\emptyset$. Recall that the elements of the
substitution $R(S)[U]$ are pairs of the form $(a,r)$, $a$ being an
assembly of $S$-structures, $$a=\{s_B\}_{B\in \pi}\mbox{, }s_B\in
S[B]\mbox{, $B\in\pi$}\mbox{, $\pi$ a partition of $U$,}$$ and
$r\in R[\pi]$.

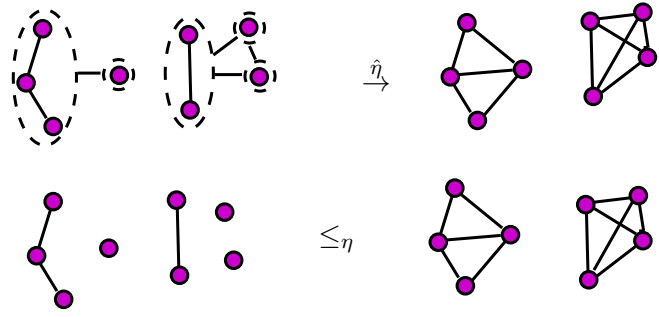
\begin{figure}
\begin{center}
% Generated with LaTeXDraw 1.9.5
% Sun Sep 14 15:30:45 VET 2008
% \usepackage[usenames,dvipsnames]{pstricks}
% \usepackage{epsfig}
% \usepackage{pst-grad} % For gradients
% \usepackage{pst-plot} % For axes
\scalebox{1} % Change this value to rescale the drawing.
{
\begin{pspicture}(0,-2.04)(8.58,2.04)
\definecolor{color838b}{rgb}{0.8,0.0,0.8}
\pscircle[linewidth=0.04,dimen=outer,fillstyle=solid,fillcolor=color838b](0.55,-0.61){0.13}
\pscircle[linewidth=0.04,dimen=outer,fillstyle=solid,fillcolor=color838b](0.33,-1.33){0.13}
\pscircle[linewidth=0.04,dimen=outer,fillstyle=solid,fillcolor=color838b](0.69,-1.91){0.13}
\pscircle[linewidth=0.04,dimen=outer,fillstyle=solid,fillcolor=color838b](2.19,-0.59){0.13}
\pscircle[linewidth=0.04,dimen=outer,fillstyle=solid,fillcolor=color838b](1.29,-1.23){0.13}
\pscircle[linewidth=0.04,dimen=outer,fillstyle=solid,fillcolor=color838b](2.23,-1.59){0.13}
\pscircle[linewidth=0.04,dimen=outer,fillstyle=solid,fillcolor=color838b](2.83,-0.75){0.13}
\psline[linewidth=0.04cm,fillcolor=color838b](0.52,-0.7)(0.36,-1.22)
\psline[linewidth=0.04cm,fillcolor=color838b](0.38,-1.44)(0.6,-1.8)
\psline[linewidth=0.04cm,fillcolor=color838b](2.2,-0.7)(2.22,-1.48)
\usefont{T1}{ptm}{m}{n}
\rput(4.311406,-1.11){$\leq_{\eta}$}
\pscircle[linewidth=0.04,dimen=outer,fillstyle=solid,fillcolor=color838b](5.89,-0.43){0.13}
\pscircle[linewidth=0.04,dimen=outer,fillstyle=solid,fillcolor=color838b](5.67,-1.15){0.13}
\pscircle[linewidth=0.04,dimen=outer,fillstyle=solid,fillcolor=color838b](6.03,-1.73){0.13}
\pscircle[linewidth=0.04,dimen=outer,fillstyle=solid,fillcolor=color838b](6.63,-1.05){0.13}
\psline[linewidth=0.04cm,fillcolor=color838b](5.86,-0.52)(5.7,-1.04)
\psline[linewidth=0.04cm,fillcolor=color838b](5.72,-1.26)(5.94,-1.62)
\psline[linewidth=0.04cm,fillcolor=color838b](5.98,-0.52)(6.52,-0.96)
\psline[linewidth=0.04cm,fillcolor=color838b](5.78,-1.12)(6.5,-1.06)
\psline[linewidth=0.04cm,fillcolor=color838b](6.12,-1.66)(6.56,-1.14)
\pscircle[linewidth=0.04,dimen=outer,fillstyle=solid,fillcolor=color838b](0.41,1.69){0.13}
\pscircle[linewidth=0.04,dimen=outer,fillstyle=solid,fillcolor=color838b](0.19,0.97){0.13}
\pscircle[linewidth=0.04,dimen=outer,fillstyle=solid,fillcolor=color838b](0.55,0.39){0.13}
\pscircle[linewidth=0.04,dimen=outer,fillstyle=solid,fillcolor=color838b](2.35,1.61){0.13}
\pscircle[linewidth=0.04,dimen=outer,fillstyle=solid,fillcolor=color838b](1.43,1.07){0.13}
\pscircle[linewidth=0.04,dimen=outer,fillstyle=solid,fillcolor=color838b](2.37,0.61){0.13}
\pscircle[linewidth=0.04,dimen=outer,fillstyle=solid,fillcolor=color838b](3.29,1.05){0.13}
\psline[linewidth=0.04cm,fillcolor=color838b](0.38,1.6)(0.22,1.08)
\psline[linewidth=0.04cm,fillcolor=color838b](0.24,0.86)(0.46,0.5)
\psline[linewidth=0.04cm,fillcolor=color838b](2.36,1.5)(2.38,0.72)
\usefont{T1}{ptm}{m}{n}
\rput(4.8414063,1.09){$\stackrel{\hat{\eta}}{\rightarrow}$}
\pscircle[linewidth=0.04,dimen=outer,fillstyle=solid,fillcolor=color838b](6.05,1.77){0.13}
\pscircle[linewidth=0.04,dimen=outer,fillstyle=solid,fillcolor=color838b](5.83,1.05){0.13}
\pscircle[linewidth=0.04,dimen=outer,fillstyle=solid,fillcolor=color838b](6.19,0.47){0.13}
\pscircle[linewidth=0.04,dimen=outer,fillstyle=solid,fillcolor=color838b](7.69,1.79){0.13}
\pscircle[linewidth=0.04,dimen=outer,fillstyle=solid,fillcolor=color838b](6.79,1.15){0.13}
\pscircle[linewidth=0.04,dimen=outer,fillstyle=solid,fillcolor=color838b](7.73,0.79){0.13}
\pscircle[linewidth=0.04,dimen=outer,fillstyle=solid,fillcolor=color838b](8.45,1.31){0.13}
\psline[linewidth=0.04cm,fillcolor=color838b](6.02,1.68)(5.86,1.16)
\psline[linewidth=0.04cm,fillcolor=color838b](5.88,0.94)(6.1,0.58)
\psline[linewidth=0.04cm,fillcolor=color838b](7.7,1.68)(7.72,0.9)
\psline[linewidth=0.04cm,fillcolor=color838b](6.14,1.68)(6.68,1.24)
\psline[linewidth=0.04cm,fillcolor=color838b](5.94,1.08)(6.66,1.14)
\psline[linewidth=0.04cm,fillcolor=color838b](6.28,0.54)(6.72,1.06)
\psline[linewidth=0.04cm,fillcolor=color838b](7.84,0.84)(8.36,1.24)
\psline[linewidth=0.04cm,fillcolor=color838b](8.36,1.4)(7.8,1.74)
\psellipse[linewidth=0.04,linestyle=dashed,dash=0.16cm
0.16cm,dimen=outer](0.44,1.04)(0.44,0.92)
\psellipse[linewidth=0.04,linestyle=dashed,dash=0.16cm
0.16cm,dimen=outer](2.36,1.1)(0.34,0.76)
\pscircle[linewidth=0.04,linestyle=dashed,dash=0.16cm
0.16cm,dimen=outer](1.42,1.08){0.22}
\pscircle[linewidth=0.04,linestyle=dashed,dash=0.16cm
0.16cm,dimen=outer](3.3,1.06){0.22}
\psline[linewidth=0.04cm,fillcolor=color838b](0.88,1.1)(1.18,1.1)
\psline[linewidth=0.04cm,fillcolor=color838b](2.7,1.08)(3.06,1.08)
\pscircle[linewidth=0.04,dimen=outer,fillstyle=solid,fillcolor=color838b](3.15,1.71){0.13}
\pscircle[linewidth=0.04,linestyle=dashed,dash=0.16cm
0.16cm,dimen=outer](3.16,1.7){0.22}
\psline[linewidth=0.04cm](2.68,1.34)(3.0,1.58)
\pscircle[linewidth=0.04,dimen=outer,fillstyle=solid,fillcolor=color838b](8.39,1.91){0.13}
\psline[linewidth=0.04cm](8.28,1.88)(7.78,1.82)
\psline[linewidth=0.04cm](8.4,1.82)(8.46,1.44)
\psline[linewidth=0.04cm](8.46,1.44)(8.4,1.44)
\pscircle[linewidth=0.04,dimen=outer,fillstyle=solid,fillcolor=color838b](7.63,-0.65){0.13}
\pscircle[linewidth=0.04,dimen=outer,fillstyle=solid,fillcolor=color838b](7.67,-1.65){0.13}
\pscircle[linewidth=0.04,dimen=outer,fillstyle=solid,fillcolor=color838b](8.39,-1.13){0.13}
\psline[linewidth=0.04cm,fillcolor=color838b](7.64,-0.76)(7.66,-1.54)
\psline[linewidth=0.04cm,fillcolor=color838b](7.78,-1.6)(8.3,-1.2)
\psline[linewidth=0.04cm,fillcolor=color838b](8.3,-1.04)(7.74,-0.7)
\pscircle[linewidth=0.04,dimen=outer,fillstyle=solid,fillcolor=color838b](8.33,-0.53){0.13}
\psline[linewidth=0.04cm](8.22,-0.56)(7.72,-0.62)
\psline[linewidth=0.04cm](8.34,-0.62)(8.4,-1.0)
\psline[linewidth=0.04cm](8.4,-1.0)(8.34,-1.0)
\pscircle[linewidth=0.04,dimen=outer,fillstyle=solid,fillcolor=color838b](2.95,-1.39){0.13}
\psline[linewidth=0.04cm](3.16,1.46)(3.24,1.28)
\psline[linewidth=0.04cm](8.32,1.82)(7.78,0.88)
\psline[linewidth=0.04cm](8.26,-0.6)(7.74,-1.52)
\end{pspicture}
}
 \end{center}\caption{Order induced on  $\mathscr{G}=E(\mathscr{G}_c)$ by the $c$-operad  $\mathscr{G}_c$
 of simple,
connected graphs.}\label{ograph}
\end{figure}

\begin{defi}{\em c-operads.} Let $(\mathscr{C},\eta)$ be a set-operad.
We say that $\mathscr{C}$ is a $c$-operad if it satisfies the left
cancellation law: For every pair of elements $(a,c), (a,c')\in
\mathscr{C}(\mathscr{C})[U]$, if $\eta(a,c)=\eta(a,c')$, then
$c=c'$. $\eta$ induces a natural transformation
$\hat{\eta}:E(\mathscr{C})(\mathscr{C})\rightarrow E(\mathscr{C})$
defined as the composition  $\hat{\eta}=E(\eta)\circ \alpha$, \be
E(\mathscr{C})(\mathscr{C})\stackrel{\alpha}{\rightarrow}
E(\mathscr{C}(\mathscr{C}))\stackrel{E(\eta)}{\rightarrow}E(\mathscr{C}).\eeq
\noindent $\alpha$ being the associativity isomorphism for the
substitution of species.\end{defi}
 The
elements of $E(\mathscr{C})(\mathscr{C})[U]$ are pairs of
assemblies $(a_1,a_2)$, where $a_1=\{c_B\}_{B\in\pi}$ is an
assembly of $\mathscr{C}$-structures over the set $U$, and $a_2$
is an assembly of $\mathscr{C}$-structures on the set $\pi$. More
explicitly, if $a_2=\{\tilde{c}_{B'}|B'\in\pi'\}$, $\pi'$ is the
partition induced by the assembly $a_2$ (a partition on the blocks
of $\pi$), then $a_1$ can be written as an union of subassemblies
$a_1=\biguplus_{B'\in\pi'}a_1^{B'}$, $a_1^{B'}=\{c_B\in a_1|B\in
B'\}$. $\hat{\eta}$ is explicitly evaluated as follows
$$\hat{\eta}(a_1,a_2)=\{\eta(a_1^{B'},\tilde{c}_{B'})|B'\in \pi'\}.$$

 $\hat{\eta}$ inherits the
left cancellation law: if
$\hat{\eta}(a_1,a_2')=\hat{\eta}(a_1,a_2'')$, then $a_2'=a_2''$.
We can define a partially ordered set
$(E(\mathscr{C})[U],\leq_{\eta})$, for every finite set $U$. The
partial order $\leq_{\eta}$ defined by \be a_1\leq_{\eta}
a_2,\mbox{ if there exists $a_2'\in\mathscr{C}[\pi]$, such that
}\eta(a_1,a_2')=a_2,\eeq \noindent $\pi$ being the partition
associated to the assembly $a_1$. See figure \ref{ograph} for the
partial order induced by the operad of simple, connected graphs
$\mathscr{G}_c$, on the simple graphs
$\mathscr{G}=E(\mathscr{G}_c)$ (assemblies of connected graphs)
(\cite{Julia-Miguel}, Example 3.19.)

The assemblies with only one structure $\{c_U\}$,
$c_U\in\mathscr{C}[U]$, are maximal elements of this poset.
However, as in the c-monoid case, in general an assembly with more
than one element could be maximal. We define the poset
$\mathscr{P}_{\mathscr{C}}[U]$ to be the subposet of
$(E(\mathscr{C})[U],\leq_{\eta})$ whose maximal elements are the
assemblies with only one structure. Explicitly
$$
\mathscr{P}_{\mathscr{C}}[U]:=\bigcup_{c_U\in\mathscr{C}[U]}[\hat{0},\{c_U\}].$$
The order on $\Pop_{\Cop}[U]$ is also functorial. For a bijection
$f:U\rightarrow U'$, the bijection
$$\Pop_{\Cop}[f]:\Pop_{\Cop}[U]\rightarrow\Pop_{\Cop}[U']\mbox{, }\Pop_{\Cop}[f]a=E(\Cop)[f]a$$
 is an order isomorphism.

From the properties of $c$-operads the following theorem follows
(see \cite{Julia-Miguel} Theorem 3.4.)
\begin{theo}The family $\{\Pop_{\Cop}[U]| \mbox{ $U$ a finite
set}\}$ satisfies the following properties:
\begin{enumerate}
\item $\Pop_{\Cop}[U]$ has a $\hat{0}$ equal to the assembly of
elements all in $\Cop_1=X$. \item For $a_1\in \Pop_{\Cop}[U],$ the
order coideal $\mathcal{C}_{a_1}$ is isomorphic to
$\Pop_{\Cop}[\pi]$, $\pi$ being the partition induced by the
assembly $a_1$. Every interval $[a_1,a_2]$ in $\Pop_{\Cop}[U]$ is
isomorphic to $[\hat{0},a_2']$ of $\Pop_{\Cop}[\pi]$, $a_2'$ being
the unique assembly such that $\hat{\eta}(a_1,a_2')=a_2$. \item If
$a_1=\{c_B|B\in\pi\}$, the interval $[\hat{0},a_1]$ is isomorphic
to the direct product
$$\prod_{B\in \pi}[\hat{0},\{c_B\}],$$
where $[\hat{0},\{c_B\}]$ is an interval of $\Pop_{\Cop}[B]$ for
all blocks $B\in\pi$.
\end{enumerate}
\end{theo}

\begin{defi}{\em M\"obius substitutional inverse}. Let
$\mathscr{C}$ be a c-operad. Define the M\"obius (substitutional)
inverse $\mathscr{C}^{\langle -1\rangle}$; $$ \mathscr{C}^{\langle
-1\rangle}[U]=\{[\hat{0},\{c_U\}]:[\hat{0},\{c_U\}]\mbox{ an
interval of $\Pop_{\mathscr{C}}[U],$ } c_U\in \mathscr{C}[U]\},$$
for a bijection $f:U\rightarrow U'$, the isomorphism
 $\Cop^{\langle -1\rangle}[f]:\coprod_{c\in \Cop[U]}[\hat{0},\{c\}]\rightarrow
\coprod_{c'\in \Cop[U']}[\hat{0},\{c'\}]$ is given by
 \be \label{actionop}\Cop^{\langle -1\rangle}[f]a_1=E(\Cop)[f]a_1.\eeq
 where $a_1$ is an assembly  $a_1\leq_{\eta} \{c_U\}$, for some $c_U\in \Cop[U]$.
 It
 is clear that if $a_1\in[\hat{0},\{c_U\}]$, then $\Cop^{\langle -1\rangle}[f]a_1\in
 [\hat{0},\{\Cop[f]c_U\}]$.
 \end{defi}

As in the case of $c$-monoids, by M\"obius inversion on the posets
$\Pop_{\Cop}[U]$ and $\Pop_{\Cop}[U,\sigma]$, we obtain
\begin{eqnarray} \Mob \Cop^{\langle -1\rangle}(x)&=&(\Cop(x))^{\langle
-1\rangle}\\ \Ch \Cop^{\langle
-1\rangle}(\xx)&=&(\Ch\Cop(\xx))^{\langle
-1\rangle}.\end{eqnarray}
\begin{ex} The species $E_+$ is a $c$-operad (the operad
$\mathscr{C}\mrm{om}$ in \cite{G-K}). It induces the M\"obius
species $E_+^{\langle -1 \rangle}[U]=\{\Pi[U]\},$  where $\Pi[U]$
is the classical poset of partitions ordered by refinement.
\begin{eqnarray*} \Ch E_+^{\langle
-1\rangle}(\xx)=\sum_{n=1}^{\infty}\sum_{\alpha\vdash
n}\mu(\Pi_{\alpha}[n])\frac{p_{\alpha}(\xx)}{z_{\alpha}}&=&(\Ch
E_+(\xx))^{\langle -1\rangle}\\
&=&\left(e^{\sum_{n=1}^{\infty}\frac{p_n(\xx)}{n}}-1\right)^{\sinverse}
=\ln\left(\prod_{n=1}^{\infty}(1+p_n)^{\mu(n)/n}\right),\end{eqnarray*}
from this, we recover  Hanlon result about the M\"obius function
of $\Pi_{\sigma}[U]$ (\cite{Hanlon} Theorem 4.)\\The c-operads are
closed under pointing ~\cite{Julia-Miguel}. The species of pointed
sets $E^{\bullet}[U]=(XD)E[U]=U$ is a c-operad
(\cite{Julia-Miguel}, Example 3.13. (3)). It induces the poset of
pointed partitions $\Pop_{E^{\bullet}}[U]$, with $n=|U|$ maximal
elements. $E^{\bullet}(x)=xe^x$ and $\Ch
E^{\bullet}(\xx)=p_1(\xx)e^{\sum_{n=0}^{\infty}\frac{p_n(\xx)}{n}}$.
We obtain that $$f(x)=\Mob
(E^{\bullet})^{\sinverse}(x)=(xe^x)^{\sinverse},$$ and
$$g(\xx)=\Ch (E^{\bullet})^{\sinverse}(\xx)=
\left(p_1(\xx)e^{\sum_{n=0}^{\infty}\frac{p_n(\xx)}{n}}\right)^{\sinverse},$$
are the solutions of the equations: $$f(x)=xe^{-f(x)},$$ and
$$g(\xx)=p_1(\xx)\exp\left({\sum_{n=0}^{\infty}-g(x_1^n,x_2^n,\dots)/n}\right).$$
By Lagrange inversion in one variable we obtain
$|(E^{\bullet})^{\sinverse}[n]|_{\mu}=(-1)^{n-1}n^{n-1}$. Because
all the intervals in $(E^{\bullet})^{\sinverse}[n]$ are
isomorphic, and there are $n$ of them, we have
$\mu[\hat{0},\{([n],1)\}]=(-1)^{n-1}n^{n-2}$. The M\"obius
cardinality for $\mrm{Fix}E^{\bullet}[\alpha]$ can be computed by
using the plethystic Lagrange inversion formula.  The $n$-pointed
set species $E^{\bullet^n}=(XD)^n E$ is a $c$-operad, inducing the
family of posets of $n$-pointed partitions. The generating
function of $E^{\bullet^n}$ is equal to $(xD)^ne^x=\phi_n(x)e^x$,
$\phi_n(x)$ being the exponential polynomials.
\end{ex}

\section{Koszul duality for operads}
\subsection{Partial Composition}\label{partial}
Let $(\Ope,\eta)$ be a tensor, $\gm$, or $\dg-$operad. For each
finite set, the morphism
\be\eta_U:\Ope(\Ope)[U]=\bigoplus_{\pi\in\Pi[U]}
\left(\bigotimes_{B\in\pi}\Ope[B]\right)\otimes\Ope[\pi]\rightarrow
\Ope[U],\eeq \noindent decomposes as a direct sum of morphisms
$\eta_U^{\pi}$, $\pi\in\Pi[U]$, \be
\eta_{U}^{\pi}:\left(\bigotimes_{B\in\pi}\Ope[B]\right)\otimes\Ope[\pi]\rightarrow\Ope[U].\eeq
Choosing the partition with only one non-singleton block
$U_1\subset U$; $\pi=\{U_1\}\cup \{\{u\}|u\in U_2\}$,
$U_2=U\backslash U_1$, the natural transformation $\eta_U^{\pi}$
goes from $\Ope[U_1]\otimes\Ope[U_2\uplus\{U_1\}]$ to $\Ope[U]$.
By the definition of derivative, for each proper decomposition
$U_1\uplus U_2=U$, we have a natural morphism from
$\Ope[U_1]\otimes \Ope'[U_2]$ to $\Ope[U]$. Using the definition
of product, we obtain a natural transformation (partial
composition) from the product $\Ope\cdot\Ope'$ to $\Ope$. The
product $\Ope\cdot\Ope'$ is clearly isomorphic to the species
$\Sc_{\Ope}^{\underline{2}}$ of $\Ope$-enriched trees with exactly
two internal vertices. By abuse of language we denote this partial
composition with the same symbol $\star$ used as `ghost element'
in the definition of derivative. Iterating this procedure we
obtain the natural transformations (partial compositions) denoted
by $\star_1$ and $\star_2$, from
$\Ope\cdot(\Ope\cdot\Ope')'=\Ope^2\cdot\Ope''+\Ope\cdot\Ope'^2$ to
$\Ope\cdot\Ope'$. The composed of those partial compositions
$\star_1\circ\star_2$ and $\star_2\circ\star_1$ go from
$\Ope^2\cdot\Ope''+\Ope\cdot(\Ope')^2$ to $\Ope$. The
associativity of $\eta$ gives us the identity
$\star_1\circ\star_2=\star_2\circ\star_1$ (see figure
(\ref{commuta})). Conversely, from a partial composition
$\star:\Ope\cdot\Ope'\rightarrow\Ope$ such that the derived
partial compositions $\star_1$ and $\star_2$ commute, we can give
$\Ope$ an operad structure, $\eta:\Ope(\Ope)\rightarrow \Ope$.

 For a tree $t$ in $\Sc_{\Ope}[U]$ and a pair
of internal vertices $v$ and $v'$, such that $v$ is a son of $v'$,
a partial composition can be performed through $v$ and the
resulting tree will be denoted by $v(t)\in \Sc_{\Ope}[U]$.
\begin{figure}
\begin{center}
% Generated with LaTeXDraw 1.9.5
% Thu Jun 12 21:59:28 VET 2008
% \usepackage[usenames,dvipsnames]{pstricks}
% \usepackage{epsfig}
% \usepackage{pst-grad} % For gradients
% \usepackage{pst-plot} % For axes
\scalebox{1} % Change this value to rescale the drawing.
{
\begin{pspicture}(0,-2.3)(9.56,2.3)
\definecolor{color1496}{rgb}{0.8,0.0,0.8}
\psline[linewidth=0.04cm](1.62,0.36)(0.54,-0.84)
\psline[linewidth=0.04cm](1.72,0.22)(1.48,-1.14)
\psline[linewidth=0.04cm](1.9,0.32)(2.6,-1.1)
\psline[linewidth=0.04cm](1.84,0.56)(2.96,1.68)
\psline[linewidth=0.04cm](2.98,1.7)(3.28,0.56)
\psline[linewidth=0.04cm](2.98,1.66)(2.62,0.52)
\pscircle[linewidth=0.04,linecolor=color1496,dimen=outer](0.42,-0.98){0.2}
\pscircle[linewidth=0.04,linecolor=color1496,dimen=outer](1.48,-1.3){0.2}
\pscircle[linewidth=0.04,linecolor=color1496,dimen=outer](2.7,-1.24){0.2}
\pscircle[linewidth=0.04,linecolor=color1496,dimen=outer](2.58,0.34){0.2}
\pscircle[linewidth=0.04,linecolor=color1496,dimen=outer](3.34,0.38){0.2}
\psarc[linewidth=0.04]{-cc}(2.95,1.57){0.43}{184.63547}{349.5085}
\usefont{T1}{ptm}{m}{n}
\rput(3.6114063,1.67){$\Ope$}
\psarc[linewidth=0.04]{cc-}(1.68,0.2){0.5}{180.0}{4.763642}
\usefont{T1}{ptm}{m}{n}
\rput(0.93140626,0.29){$\Ope$}
\usefont{T1}{ptm}{m}{n}
\rput(5.141406,-0.13){$\stackrel{\star}{\longmapsto}$}
\psline[linewidth=0.04cm](7.66,0.78)(6.52,-0.2)
\psline[linewidth=0.04cm](7.64,0.74)(7.1,-0.4)
\psline[linewidth=0.04cm](6.1,0.2)(7.68,0.8)
\psline[linewidth=0.04cm](7.66,0.78)(7.86,-0.32)
\psline[linewidth=0.04cm](7.66,0.78)(8.66,-0.08)
\pscircle[linewidth=0.04,linecolor=color1496,dimen=outer](5.94,0.1){0.2}
\pscircle[linewidth=0.04,linecolor=color1496,dimen=outer](7.06,-0.58){0.2}
\pscircle[linewidth=0.04,linecolor=color1496,dimen=outer](6.38,-0.3){0.2}
\pscircle[linewidth=0.04,linecolor=color1496,dimen=outer](7.86,-0.52){0.2}
\pscircle[linewidth=0.04,linecolor=color1496,dimen=outer](8.82,-0.18){0.2}
\psarc[linewidth=0.04]{-cc}(7.69,0.71){0.43}{160.46335}{6.8427734}
\usefont{T1}{ptm}{m}{n}
\rput(8.191406,1.05){$\Ope$}
\pscircle[linewidth=0.04,linecolor=color1496,dimen=outer](1.77,0.41){0.19}
\usefont{T1}{ptm}{m}{n}
\rput(1.7514062,0.37){$\star$}
\usefont{T1}{ptm}{m}{n}
\rput(2.5514061,0.35){$b$}
\usefont{T1}{ptm}{m}{n}
\rput(5.9214063,0.13){$a$}
\usefont{T1}{ptm}{m}{n}
\rput(8.801406,-0.15){$c$}
\usefont{T1}{ptm}{m}{n}
\rput(7.021406,-0.53){$e$}
\usefont{T1}{ptm}{m}{n}
\rput(1.4314063,-1.29){$d$}
\usefont{T1}{ptm}{m}{n}
\rput(7.851406,-0.53){$b$}
\usefont{T1}{ptm}{m}{n}
\rput(0.38140625,-0.95){$a$}
\usefont{T1}{ptm}{m}{n}
\rput(3.3014061,0.39){$c$}
\usefont{T1}{ptm}{m}{n}
\rput(2.6614063,-1.21){$e$}
\usefont{T1}{ptm}{m}{n}
\rput(6.331406,-0.29){$d$}
\psframe[linewidth=0.04,framearc=0.15,dimen=outer](9.56,2.3)(0.0,-2.3)
\usefont{T1}{ptm}{m}{n}
\rput(5.0914063,-1.79){$\star:\Ope\cdot \Ope'\rightarrow \Ope$}
\end{pspicture}

}\end{center}\caption{graphical representation of the partial
composition $\star:\Ope\cdot\Ope'\rightarrow\Ope$, for an operad
$\Ope$.}
\end{figure}
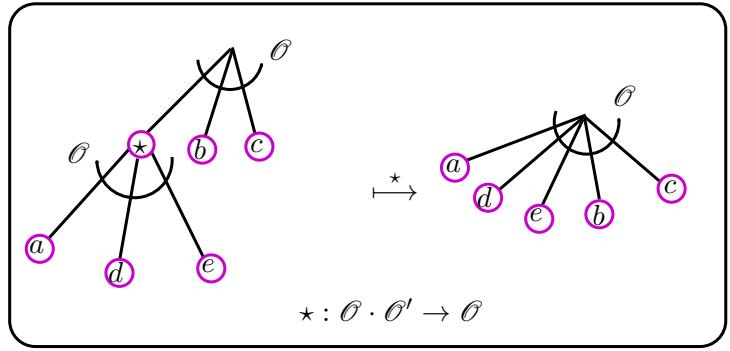

 \begin{defi} {\em Quadratic Operads}. Let $G$ be a tensor species of the form
$G=X+G_{2^+}$. The species of enriched Schr\"oder trees $\Sc_G$ is
the free operad generated by $G$. Let $R$ be a subspecies of
$\Sc_G^{\underline{2}}$. Define $\Rel_{\Ope}=\led R \rid$ to be
the operad ideal generated by R in  $\Sc_G$. The operad
$\Ope=\Sc_G/\Rel_{\Ope}$ will be called the quadratic operad
generated by $G$ with quadratic relations in $R$, and denoted by
$\Ope={\mathcal O}(G,R)$. There is a natural graduation on $\Ope$
given by
$\Ope^{\grader}=\Ope^{\underline{0}}+\Ope^{\underline{1}}+
\sum_{k=2}^{\infty}\Sc_G^{\underline{k}}/\Rel_{\Ope}^{\underline{k}},$
where $\Ope^{\underline{0}}=X$ and $\Ope^{\underline{1}}=G_{2^+}$.
The {\em quadratic dual cooperad} of $\Ope$ (see \cite{Fresse}) is
defined by

\be \Ope^{\coop}:=\mathcal{O}(G^*,R^{\perp})^*.\eeq When $G$ is
concentrated in cardinality $2$, the {\em quadratic dual}  (see
\cite{G-K}) of $\Ope$ is defined by \be
\Ope^{!}:=\mathcal{O}(G^*\odot\Lambda_2,\widetilde{R}^{\perp}),\eeq
\noindent where $\widetilde{R}$ is the kernel of the natural
transformation
$\star^{\underline{2}}:\Sc_{G\odot\Lambda_2}^{\underline{2}}\rightarrow
\Ope_3\odot\Lambda_3,$ $\star$ being the partial composition on
the operad $\Ope\odot\Lambda$.
\end{defi}
\begin{rem}\label{remark1}Observe that if $\Ope={\mathcal O}(G,R)$, $R$ is the
kernel of $$\star:G_{2^+}\cdot G_{2^+}'\rightarrow
\Ope^{\underline{2}}.$$ \noindent Hence, $\Ope$ is quadratic if
and only if $\Ope=\Sc_G/\led\mathrm{Ker}(\star)\rid$,
$G_{2^+}=\Ope^{\underline{1}}$.
\end{rem}

\subsection{Bar construction for operads} We follow here B. Fresse
(see \cite{Fresse}) on his generalization of the original
Ginzburg-Kapranov definition of Koszul operads. Let
$\Ope=\mathcal{O}(G,R)$ be a quadratic operad and $\Ope^{\grader}$
its corresponding graded operad. We are going to construct a
$\dg$-species by defining a differential on the inverse
${\Ope^{\grader}}^{\langle\mils
1\rangle}=\Sc_{\mils\Ope^{\grader}}$. Explicitly, using equations
(\ref{49}) and (\ref{explic}) \be
\Sc_{\mils\Ope}[U]=\bigoplus_{t\in
\Sc_E[U]}\bigotimes_{v\in\mathrm{Iv}(t)}\mil \Ope[\pi_v]=
\bigoplus_{t\in
\Sc_E[U]}s^{-|\mathrm{Iv}(t)|}\left(\bigotimes_{v\in\mathrm{Iv}(t)}
\Ope[\pi_v]\right)\otimes \Lambda[\mathrm{Iv}(t)]\eeq

Hence, a decomposable element of $\Sc_{\mils
\Ope^{\grader}}^{\underline{k}}[U]$ can be written in the form
$v_1\wedge v_2\wedge\dots\wedge v_{r-1}\wedge v_r\otimes t$ for
some $r\geq 0$, where $t$ is an $\Ope$-enriched Schr\"oder tree
$t\in \Sc_{\Ope}^{\underline{r}}[U]$ and $\{v_1,v_2,\dots,v_r\}$
is the set of internal vertices of $t$ in any order such that the
root is indexed as last element $v_r$. Define
$\tilde{d}:\Sc_{\mils \Ope^{\grader}}^{\underline{k}}\rightarrow
\Sc_{\mils \Ope^{\grader}}^{\underline{k+1}}\mbox{,
}k=0,1,2,\dots,$
\begin{equation}\label{Baod}
 \tilde{d}(v_1\wedge v_2\wedge \dots\wedge v_{r-1}\wedge
v_r\otimes t)= \sum_{i=1}^{r-1}(-1)^{i-1}v_1\wedge
v_2\wedge\dots\wedge\widehat{v_i}\wedge\dots\wedge v_r\otimes
v_i(t).\end{equation} \noindent Clearly $\tilde{d}^2=0$, and this
will be called the  bar construction for operads and denoted by
$\Bao(\Ope)$; $\Bao(\Ope)=({\Ope^{\grader}}^{\langle \mil
1\rangle},d).$ It is easy to see that
$\Bao(\Ope)^{\underline{0}}=\Sc_G$, and that the cohomology
$H^{0}\Bao(\Ope)$ is isomorphic to the quadratic dual cooperad
$\Ope^{\coop}$.

\begin{figure}
\begin{center}
% Generated with LaTeXDraw 1.9.5
% Wed Aug 27 12:15:14 BOT 2008
% \usepackage[usenames,dvipsnames]{pstricks}
% \usepackage{epsfig}
% \usepackage{pst-grad} % For gradients
% \usepackage{pst-plot} % For axes
\scalebox{1} % Change this value to rescale the drawing.
{
\begin{pspicture}(0,-5.92)(10.8,5.92)
\definecolor{color1496}{rgb}{0.8,0.0,0.8}
\psline[linewidth=0.04cm](1.72,3.66)(0.64,2.46)
\psline[linewidth=0.04cm](1.82,3.52)(1.64,2.24)
\psline[linewidth=0.04cm](2.0,3.54)(2.7,2.2)
\psline[linewidth=0.04cm](2.02,3.9)(3.06,4.94)
\psline[linewidth=0.04cm](3.08,5.0)(3.38,3.86)
\psline[linewidth=0.04cm](3.08,4.96)(2.72,3.82)
\pscircle[linewidth=0.04,linecolor=color1496,dimen=outer](0.52,2.32){0.2}
\pscircle[linewidth=0.04,linecolor=color1496,dimen=outer](1.61,2.03){0.23}
\pscircle[linewidth=0.04,linecolor=color1496,dimen=outer](2.8,2.06){0.2}
\pscircle[linewidth=0.04,linecolor=color1496,dimen=outer](2.68,3.64){0.2}
\pscircle[linewidth=0.04,linecolor=color1496,dimen=outer](3.44,3.68){0.2}
\psarc[linewidth=0.04]{cc-cc}(3.05,4.87){0.43}{184.63547}{349.5085}
\usefont{T1}{ptm}{m}{n}
\rput(3.7114062,4.97){$\Ope$}
\psarc[linewidth=0.04]{cc-cc}(1.78,3.5){0.5}{180.0}{4.763642}
\usefont{T1}{ptm}{m}{n}
\rput(1.0314063,3.59){$\Ope$}
\usefont{T1}{ptm}{m}{n}
\rput(5.7214065,2.89){$\stackrel{\star_1}{\longrightarrow}$}
\psline[linewidth=0.04cm](8.68,4.32)(7.56,3.38)
\psline[linewidth=0.04cm](8.66,4.28)(8.12,3.14)
\psline[linewidth=0.04cm](7.12,3.74)(8.7,4.34)
\psline[linewidth=0.04cm](8.68,4.32)(8.88,3.22)
\psline[linewidth=0.04cm](8.68,4.32)(9.68,3.46)
\pscircle[linewidth=0.04,linecolor=color1496,dimen=outer](6.96,3.64){0.2}
\pscircle[linewidth=0.04,linecolor=color1496,dimen=outer](8.08,2.96){0.2}
\pscircle[linewidth=0.04,linecolor=color1496,dimen=outer](8.88,3.02){0.2}
\pscircle[linewidth=0.04,linecolor=color1496,dimen=outer](9.78,3.34){0.2}
\psarc[linewidth=0.04]{-cc}(8.71,4.25){0.43}{160.46335}{6.8427734}
\usefont{T1}{ptm}{m}{n}
\rput(9.231406,4.47){$\Ope$}
\pscircle[linewidth=0.04,linecolor=color1496,dimen=outer](1.93,3.75){0.25}
\usefont{T1}{ptm}{m}{n}
\rput(1.9514062,3.77){$\star_1$}
\usefont{T1}{ptm}{m}{n}
\rput(2.6514063,3.65){$b$}
\usefont{T1}{ptm}{m}{n}
\rput(6.9414062,3.67){$a$}
\usefont{T1}{ptm}{m}{n}
\rput(9.741406,3.35){$c$}
\usefont{T1}{ptm}{m}{n}
\rput(8.041407,3.01){$e$}
\usefont{T1}{ptm}{m}{n}
\rput(0.8514063,0.61){$d$}
\usefont{T1}{ptm}{m}{n}
\rput(8.871407,3.01){$b$}
\usefont{T1}{ptm}{m}{n}
\rput(0.48140624,2.35){$a$}
\usefont{T1}{ptm}{m}{n}
\rput(3.4414062,3.71){$c$}
\usefont{T1}{ptm}{m}{n}
\rput(2.7614062,2.09){$e$}
\usefont{T1}{ptm}{m}{n}
\rput(6.431406,2.01){$d$}
\psframe[linewidth=0.04,framearc=0.15,dimen=outer](10.8,5.92)(0.0,-5.92)
\usefont{T1}{ptm}{m}{n}
\rput(2.2414062,5.49){\psframebox[linewidth=0.04]{$\Ope\cdot
\Ope'^2$}}
\pscircle[linewidth=0.04,linecolor=color1496,dimen=outer](0.9,0.6){0.2}
\pscircle[linewidth=0.04,linecolor=color1496,dimen=outer](2.04,0.6){0.2}
\psline[linewidth=0.04cm](1.48,1.84)(0.98,0.74)
\psline[linewidth=0.04cm](1.62,1.82)(1.96,0.76)
\usefont{T1}{ptm}{m}{n}
\rput(2.0014062,0.61){$f$}
\usefont{T1}{ptm}{m}{n}
\rput(1.6114062,2.05){$\star_2$}
\usefont{T1}{ptm}{m}{n}
\rput(7.391406,3.23){$\star_2$}
\usefont{T1}{ptm}{m}{n}
\rput(7.181406,1.63){$f$}
\pscircle[linewidth=0.04,linecolor=color1496,dimen=outer](7.24,1.64){0.2}
\pscircle[linewidth=0.04,linecolor=color1496,dimen=outer](6.46,2.02){0.2}
\psline[linewidth=0.04cm](7.2,3.06)(6.52,2.2)
\psline[linewidth=0.04cm](7.38,2.98)(7.26,1.84)
\psarc[linewidth=0.04]{-cc}(7.22,2.76){0.46}{177.51045}{317.72632}
\usefont{T1}{ptm}{m}{n}
\rput(7.731406,2.51){$\Ope$}
\psarc[linewidth=0.04]{-cc}(1.59,1.41){0.43}{182.48955}{338.9625}
\usefont{T1}{ptm}{m}{n}
\rput(2.2314062,1.43){$\Ope$}
\psline[linewidth=0.04cm](1.9,-3.68)(0.84,-4.82)
\psline[linewidth=0.04cm](2.24,-3.74)(2.9,-5.08)
\psline[linewidth=0.04cm](2.26,-3.38)(3.3,-2.24)
\psline[linewidth=0.04cm](3.28,-2.28)(3.58,-3.42)
\psline[linewidth=0.04cm](3.28,-2.32)(2.92,-3.46)
\pscircle[linewidth=0.04,linecolor=color1496,dimen=outer](0.7,-4.98){0.2}
\pscircle[linewidth=0.04,linecolor=color1496,dimen=outer](3.0,-5.22){0.2}
\pscircle[linewidth=0.04,linecolor=color1496,dimen=outer](2.9,-3.64){0.2}
\psarc[linewidth=0.04]{-cc}(3.25,-2.41){0.43}{184.63547}{349.5085}
\psarc[linewidth=0.04]{cc-}(1.98,-3.78){0.5}{180.0}{4.763642}
\usefont{T1}{ptm}{m}{n}
\rput(1.2314062,-3.69){$\Ope$}
\usefont{T1}{ptm}{m}{n}
\rput(2.1514063,-3.53){$\star_1$}
\usefont{T1}{ptm}{m}{n}
\rput(2.8514063,-3.63){$b$}
\usefont{T1}{ptm}{m}{n}
\rput(1.2914063,-5.13){$d$}
\usefont{T1}{ptm}{m}{n}
\rput(0.68140626,-4.93){$a$}
\usefont{T1}{ptm}{m}{n}
\rput(2.9614062,-5.19){$e$}
\pscircle[linewidth=0.04,linecolor=color1496,dimen=outer](1.34,-5.14){0.2}
\pscircle[linewidth=0.04,linecolor=color1496,dimen=outer](2.18,-5.2){0.2}
\psline[linewidth=0.04cm](1.96,-3.76)(1.36,-4.98)
\psline[linewidth=0.04cm](2.08,-3.76)(2.16,-5.04)
\usefont{T1}{ptm}{m}{n}
\rput(2.1414063,-5.19){$f$}
\usefont{T1}{ptm}{m}{n}
\rput(3.6414063,-3.57){$c$}
\pscircle[linewidth=0.04,linecolor=color1496,dimen=outer](3.64,-3.58){0.2}
\usefont{T1}{ptm}{m}{n}
\rput(3.7714062,-2.35){$\Ope$}
\usefont{T1}{ptm}{m}{n}
\rput(9.191406,1.71){\psframebox[linewidth=0.04]{$\Ope\cdot
\Ope'$}} \psline[linewidth=0.04cm](8.64,-2.86)(8.1,-4.0)
\psline[linewidth=0.04cm](7.1,-3.4)(8.68,-2.8)
\psline[linewidth=0.04cm](8.66,-2.82)(8.86,-3.92)
\psline[linewidth=0.04cm](8.66,-2.82)(9.66,-3.68)
\pscircle[linewidth=0.04,linecolor=color1496,dimen=outer](6.94,-3.5){0.2}
\pscircle[linewidth=0.04,linecolor=color1496,dimen=outer](8.06,-4.18){0.2}
\pscircle[linewidth=0.04,linecolor=color1496,dimen=outer](8.86,-4.12){0.2}
\pscircle[linewidth=0.04,linecolor=color1496,dimen=outer](9.76,-3.8){0.2}
\psarc[linewidth=0.04]{-cc}(8.69,-2.89){0.43}{160.46335}{6.8427734}
\usefont{T1}{ptm}{m}{n}
\rput(9.211407,-2.67){$\Ope$}
\usefont{T1}{ptm}{m}{n}
\rput(6.9214063,-3.47){$a$}
\usefont{T1}{ptm}{m}{n}
\rput(9.721406,-3.79){$c$}
\usefont{T1}{ptm}{m}{n}
\rput(8.001407,-4.19){$e$}
\usefont{T1}{ptm}{m}{n}
\rput(8.851406,-4.13){$b$}
\usefont{T1}{ptm}{m}{n}
\rput(7.1914062,-3.87){$d$}
\usefont{T1}{ptm}{m}{n}
\rput(7.4214063,-4.29){$f$}
\pscircle[linewidth=0.04,linecolor=color1496,dimen=outer](7.48,-4.28){0.2}
\pscircle[linewidth=0.04,linecolor=color1496,dimen=outer](7.22,-3.86){0.2}
\psline[linewidth=0.04cm](8.62,-2.86)(7.32,-3.74)
\psline[linewidth=0.04cm](8.64,-2.82)(7.52,-4.1)
\usefont{T1}{ptm}{m}{n}
\rput(1.7514062,-1.55){\psframebox[linewidth=0.04]{$\Ope\cdot
\Ope'$}}
\usefont{T1}{ptm}{m}{n}
\rput(8.591406,-1.69){\psframebox[linewidth=0.04]{$\Ope$}}
\usefont{T1}{ptm}{m}{n}
\rput{-90.0}(2.9201562,3.9807813){\rput(3.4314063,0.55){$\stackrel{\star2}{\longrightarrow}$}}
\usefont{T1}{ptm}{m}{n}
\rput(5.661406,-1.67){$\stackrel{\star_1}{\longrightarrow}$}
\usefont{T1}{ptm}{m}{n}
\rput{-90.0}(8.1301565,7.0307813){\rput(7.561406,-0.53){$\stackrel{\star_2}{\longrightarrow}$}}
\pscircle[linewidth=0.04,linecolor=color1496,dimen=outer](2.09,-3.55){0.25}
\pscircle[linewidth=0.04,linecolor=color1496,dimen=outer](7.41,3.21){0.25}
\end{pspicture}
}

\end{center}\caption{Graphical representation of the commutation
of partial compositions.}\label{commuta}
\end{figure}
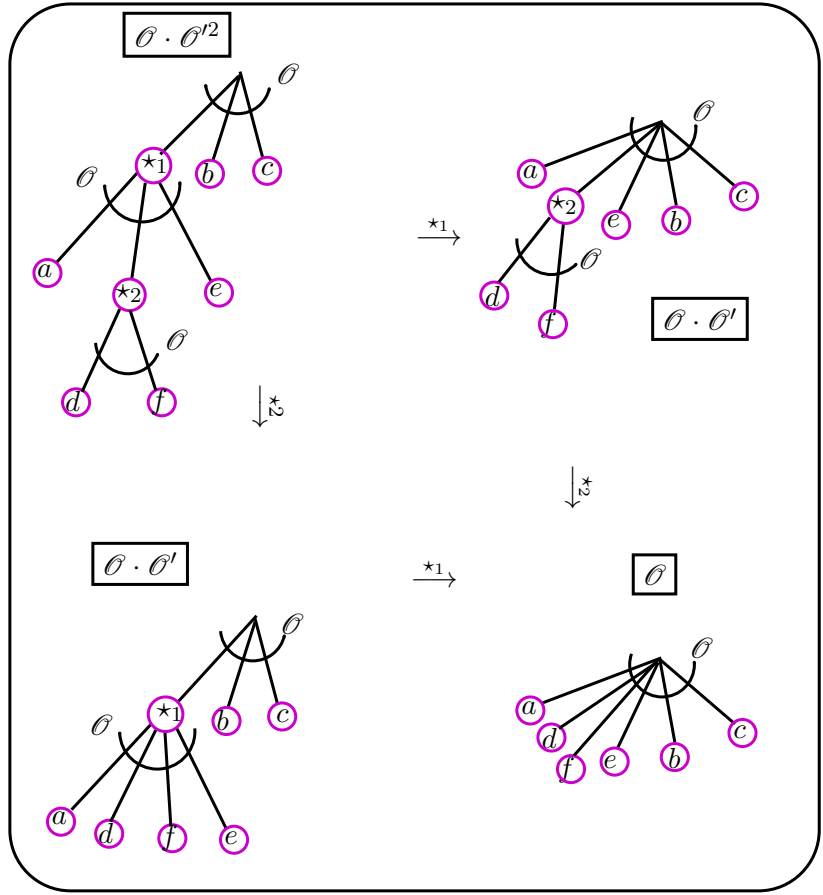
\begin{defi} A quadratic operad $\Ope$ is called Koszul if the
$\mathbf{g}$-species $\Ho \Bao(\Ope)$ is concentrated in degree
zero.
\begin{prop}If $\Ope$ is Koszul, then
\begin{eqnarray}
\Ope^{\coop}(x)&=&(\Ope^{\grader}(x))^{\langle -1\rangle}\\
\Ch\Ope^{\coop}(\xx)&=&(\Ch\Ope^{\grader}(\xx))^{\langle
-1\rangle}.
\end{eqnarray}

\end{prop}
\begin{proof}Similar to the proof of proposition (\ref{genk}).
\end{proof}

\end{defi}
\subsection {The operads of $M$-enriched and small trees} We now consider the species $\Arb$, of rooted trees where all the
vertices (internal and leaves) are labelled. It is the solution of
the implicit equation \be \mathscr{A}=XE(\Arb).\eeq

More generally, for an arbitrary species $M$, $|M[\emptyset]|=1$,
the species $\Arb_M$ of $M$-enriched trees is defined as the
solution of the implicit equation
\begin{equation}\label{etree}
\Arb_M=XM(\Arb_M).
\end{equation}

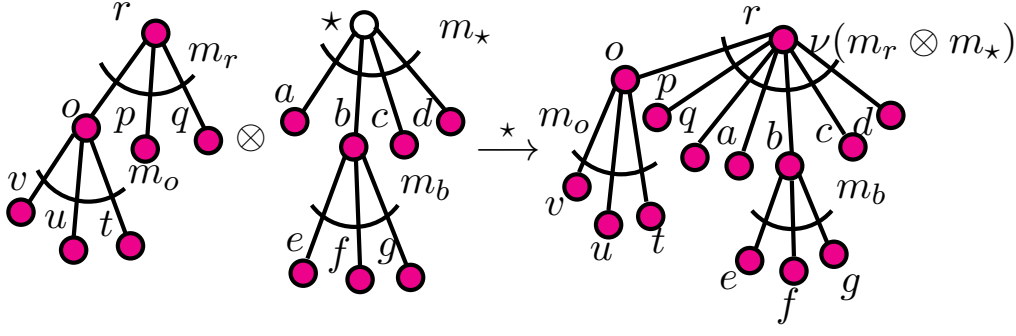
\begin{figure}\begin{center}
% Generated with LaTeXDraw 1.9.5
% Wed Aug 27 10:13:04 BOT 2008
% \usepackage[usenames,dvipsnames]{pstricks}
% \usepackage{epsfig}
% \usepackage{pst-grad} % For gradients
% \usepackage{pst-plot} % For axes
\scalebox{1.4} % Change this value to rescale the drawing.
{
\begin{pspicture}(0,-1.74375)(10.8828125,1.4221874)
\pscircle[linewidth=0.04,dimen=outer,fillstyle=solid,fillcolor=Magenta](0.7609375,-1.03625){0.14}
\pscircle[linewidth=0.04,dimen=outer,fillstyle=solid,fillcolor=Magenta](0.2609375,-0.67625){0.14}
\pscircle[linewidth=0.04,dimen=outer,fillstyle=solid,fillcolor=Magenta](4.3409376,0.18375){0.14}
\pscircle[linewidth=0.04,dimen=outer,fillstyle=solid,fillcolor=Magenta](3.4809375,-1.31625){0.14}
\pscircle[linewidth=0.04,dimen=outer,fillstyle=solid,fillcolor=Magenta](2.0409374,0.00375){0.14}
\pscircle[linewidth=0.04,dimen=outer,fillstyle=solid,fillcolor=Magenta](2.9409375,-1.25625){0.14}
\pscircle[linewidth=0.04,dimen=outer,fillstyle=solid,fillcolor=Magenta](3.4209375,-0.05625){0.14}
\pscircle[linewidth=0.04,dimen=outer,fillstyle=solid](3.5209374,1.10375){0.14}
\pscircle[linewidth=0.04,dimen=outer,fillstyle=solid,fillcolor=Magenta](0.8809375,0.12375){0.14}
\pscircle[linewidth=0.04,dimen=outer,fillstyle=solid,fillcolor=Magenta](1.5409375,1.02375){0.14}
\pscircle[linewidth=0.04,dimen=outer,fillstyle=solid,fillcolor=Magenta](2.8609376,0.18375){0.14}
\pscircle[linewidth=0.04,dimen=outer,fillstyle=solid,fillcolor=Magenta](3.9609375,-1.29625){0.14}
\pscircle[linewidth=0.04,dimen=outer,fillstyle=solid,fillcolor=Magenta](1.4409375,-0.07625){0.14}
\pscircle[linewidth=0.04,dimen=outer,fillstyle=solid,fillcolor=Magenta](1.3009375,-0.99625){0.14}
\pscircle[linewidth=0.04,dimen=outer,fillstyle=solid,fillcolor=Magenta](3.9009376,-0.03625){0.14}
\usefont{T1}{ptm}{m}{n}
\rput(2.7623436,0.45375){$a$}
\usefont{T1}{ptm}{m}{n}
\rput(3.3123438,0.27375){$b$}
\usefont{T1}{ptm}{m}{n}
\rput(3.6623437,0.21375){$c$}
\usefont{T1}{ptm}{m}{n}
\rput(4.072344,0.27375){$d$}
\usefont{T1}{ptm}{m}{n}
\rput(2.8623438,-0.96625){$e$}
\usefont{T1}{ptm}{m}{n}
\rput(3.2823439,-1.08625){$f$}
\usefont{T1}{ptm}{m}{n}
\rput(3.7423437,-1.04625){$g$}
\usefont{T1}{ptm}{m}{n}
\rput(1.2223438,1.23375){$r$}
\usefont{T1}{ptm}{m}{n}
\rput(0.59234375,-0.74625){$u$}
\usefont{T1}{ptm}{m}{n}
\rput(1.7723438,0.19375){$q$}
\usefont{T1}{ptm}{m}{n}
\rput(3.2123437,1.13375){$\star$}
\usefont{T1}{ptm}{m}{n}
\rput(1.0723437,-0.78625){$t$}
\usefont{T1}{ptm}{m}{n}
\rput(1.2523438,0.19375){$p$}
\psline[linewidth=0.04cm](1.4409375,0.94375)(0.9209375,0.22375)
\psline[linewidth=0.04cm](1.5209374,0.90375)(1.4609375,0.04375)
\psline[linewidth=0.04cm](1.6009375,0.94375)(2.0009375,0.12375)
\psline[linewidth=0.04cm](0.7609375,0.06375)(0.3409375,-0.57625)
\psline[linewidth=0.04cm](0.8409375,0.00375)(0.7609375,-0.91625)
\psline[linewidth=0.04cm](0.9209375,0.02375)(1.2609375,-0.89625)
\psline[linewidth=0.04cm](3.4409375,1.04375)(2.9409375,0.28375)
\psline[linewidth=0.04cm](3.5009375,1.00375)(3.4409375,0.06375)
\psline[linewidth=0.04cm](3.5809374,1.00375)(3.8809376,0.06375)
\psline[linewidth=0.04cm](3.6209376,1.08375)(4.2809377,0.30375)
\psline[linewidth=0.04cm](3.4409375,-0.15625)(3.4809375,-1.19625)
\psline[linewidth=0.04cm](3.3409376,-0.13625)(2.9809375,-1.11625)
\psline[linewidth=0.04cm](3.5009375,-0.11625)(3.9209375,-1.17625)
\usefont{T1}{ptm}{m}{n}
\rput(6.592344,0.19375){$q$}
\usefont{T1}{ptm}{m}{n}
\rput(0.23234375,-0.38625){$v$}
\usefont{T1}{ptm}{m}{n}
\rput(0.73234373,0.31375){$o$}
\usefont{T1}{ptm}{m}{n}
\rput(2.4423437,0.03375){$\otimes$}
\usefont{T1}{ptm}{m}{n}
\rput(2.0923438,0.81375){$m_r$}
\usefont{T1}{ptm}{m}{n}
\rput(4.4823437,1.03375){$m_{\star}$}
\usefont{T1}{ptm}{m}{n}
\rput(1.5223438,-0.34625){$m_o$}
\usefont{T1}{ptm}{m}{n}
\rput(4.1023436,-0.44625){$m_b$}
\pscircle[linewidth=0.04,dimen=outer,fillstyle=solid,fillcolor=Magenta](8.520938,0.24375){0.14}
\pscircle[linewidth=0.04,dimen=outer,fillstyle=solid,fillcolor=Magenta](7.6009374,-1.23625){0.14}
\pscircle[linewidth=0.04,dimen=outer,fillstyle=solid,fillcolor=Magenta](7.1809373,-1.13625){0.14}
\pscircle[linewidth=0.04,dimen=outer,fillstyle=solid,fillcolor=Magenta](7.5609374,-0.23625){0.14}
\pscircle[linewidth=0.04,dimen=outer,fillstyle=solid,fillcolor=Magenta](7.0809374,-0.23625){0.14}
\pscircle[linewidth=0.04,dimen=outer,fillstyle=solid,fillcolor=Magenta](7.9809375,-1.07625){0.14}
\pscircle[linewidth=0.04,dimen=outer,fillstyle=solid,fillcolor=Magenta](8.160937,-0.05625){0.14}
\usefont{T1}{ptm}{m}{n}
\rput(6.9623437,0.05375){$a$}
\usefont{T1}{ptm}{m}{n}
\rput(7.412344,0.05375){$b$}
\usefont{T1}{ptm}{m}{n}
\rput(7.882344,0.07375){$c$}
\usefont{T1}{ptm}{m}{n}
\rput(8.252344,0.19375){$d$}
\usefont{T1}{ptm}{m}{n}
\rput(6.9823437,-1.32625){$e$}
\usefont{T1}{ptm}{m}{n}
\rput(7.5423436,-1.56625){$f$}
\usefont{T1}{ptm}{m}{n}
\rput(8.1423435,-1.38625){$g$}
\usefont{T1}{ptm}{m}{n}
\rput(8.732344,0.93375){$\nu(m_r\otimes m_{\star})$}
\usefont{T1}{ptm}{m}{n}
\rput(8.242344,-0.48625){$m_b$}
\psarc[linewidth=0.04](1.5009375,0.98375){0.54}{206.56505}{318.36646}
\psarc[linewidth=0.04](0.9009375,-0.03625){0.54}{206.56505}{309.80557}
\pscircle[linewidth=0.04,dimen=outer,fillstyle=solid,fillcolor=Magenta](5.8409376,-0.79625){0.14}
\pscircle[linewidth=0.04,dimen=outer,fillstyle=solid,fillcolor=Magenta](5.5409374,-0.43625){0.14}
\pscircle[linewidth=0.04,dimen=outer,fillstyle=solid,fillcolor=Magenta](6.6609373,-0.15625){0.14}
\pscircle[linewidth=0.04,dimen=outer,fillstyle=solid,fillcolor=Magenta](6.0009375,0.58375){0.14}
\pscircle[linewidth=0.04,dimen=outer,fillstyle=solid,fillcolor=Magenta](7.5009375,0.96375){0.14}
\pscircle[linewidth=0.04,dimen=outer,fillstyle=solid,fillcolor=Magenta](6.3009377,0.22375){0.14}
\pscircle[linewidth=0.04,dimen=outer,fillstyle=solid,fillcolor=Magenta](6.2409377,-0.71625){0.14}
\usefont{T1}{ptm}{m}{n}
\rput(7.202344,1.19375){$r$}
\usefont{T1}{ptm}{m}{n}
\rput(5.7723436,-1.04625){$u$}
\usefont{T1}{ptm}{m}{n}
\rput(6.3123436,-0.94625){$t$}
\usefont{T1}{ptm}{m}{n}
\rput(5.3323436,-0.62625){$v$}
\usefont{T1}{ptm}{m}{n}
\rput(5.8923435,0.85375){$o$}
\psarc[linewidth=0.04](3.6009376,1.16375){0.54}{206.56505}{325.00797}
\psline[linewidth=0.04cm](7.6209373,0.94375)(8.420938,0.32375)
\psline[linewidth=0.04cm](7.5609374,0.88375)(8.080937,0.06375)
\psline[linewidth=0.04cm](7.5209374,0.86375)(7.5809374,-0.11625)
\psline[linewidth=0.04cm](7.1209373,-0.09625)(7.4409375,0.86375)
\psline[linewidth=0.04cm](6.7009373,-0.01625)(7.4209375,0.88375)
\psline[linewidth=0.04cm](7.4809375,-0.33625)(7.2409377,-1.01625)
\psline[linewidth=0.04cm](7.6409373,-0.29625)(7.9209375,-0.97625)
\psarc[linewidth=0.04](7.5109377,1.07375){0.59}{188.1301}{329.03625}
\usefont{T1}{ptm}{m}{n}
\rput(5.4223437,0.19375){$m_o$}
\psarc[linewidth=0.04](7.5809374,-0.33625){0.5}{220.60129}{319.3987}
\psarc[linewidth=0.04](3.4309375,-0.30625){0.51}{216.46924}{323.53076}
\usefont{T1}{ptm}{m}{n}
\rput(4.882344,0.01375){$\stackrel{\star}{\longrightarrow}$}
\usefont{T1}{ptm}{m}{n}
\rput(6.3923435,0.51375){$p$}
\psline[linewidth=0.04cm](7.6009374,-1.09625)(7.5809374,-0.37625)
\psline[linewidth=0.04cm](7.38,0.94375)(6.4,0.28375)
\psline[linewidth=0.04cm](7.4,1.02375)(6.12,0.60375)
\psline[linewidth=0.04cm](5.56,-0.29625)(5.9,0.50375)
\psline[linewidth=0.04cm](5.84,-0.65625)(5.96,0.46375)
\psline[linewidth=0.04cm](6.2,-0.57625)(6.04,0.46375)
\psarc[linewidth=0.04](5.96,0.26375){0.58}{217.87498}{305.5377}
\end{pspicture}
}

\end{center}
\caption{The partial composition for the operad of $M$-enriched
rooted trees}
\end{figure}

The species $\Arb_M$ is explicitly defined as follows
\be\label{explicit} \Arb_M[U]=\bigoplus_{t\in
\Arb[U]}\bigotimes_{u\in U}M[t^{-1}(u)],\eeq \noindent where for
each vertex $u$ of the rooted tree $t\in \Arb[U]$, $t^{-1}(u)$
denotes the fiber (set of sons) of $u$ in $t$. An $M$-enriched
tree $t^M_U$ is a decomposable element of $\Arb_M[U]$.  From
equation (\ref{explicit}) each one of it is of the form
$t_U^M=(t_U,\otimes_{u\in U} m_{u})$, where $t_U\in\Arb[U]$ and
for each $u\in U$, $m_{u}\in M[t^{-1}(u)]$.

The (set) species of rooted trees $\Arb$ has a c-operad structure
$\eta:\Arb(\Arb)\rightarrow \Arb$. An element of $\Arb(\Arb)[U]$
is a pair $(a,t')$ where $a=\{t_B\}_{B\in\pi}$ is a forest of
rooted trees and $t'$ is a rooted tree on the blocks of the
partition $\pi$. The tree $t=\eta(a,t')$ will have all the edges
in the forest $a$ plus a few more defined as follows: for every
pair of trees $t_{B_1}$ and $t_{B_2}$ such that $B_1$ and $B_2$
form an edge in $t'$, insert an edge between the roots of
$t_{B_1}$ and $t_{B_2}$.
 The operadic structure of $\Arb=\Arb_E$ can be generalized to the
 species of rooted trees enriched with a monoid $(M,\nu)$, $\nu:M.M\rightarrow M$
 (see \cite{Miguel}, and
\cite{Julia-Miguel} for a set-operad version of this
construction). For that end, it is enough to describe an
appropriate partial composition
$\star:\Arb_M\cdot\Arb_M'\rightarrow\Arb_M$. Let $t^M_{U_1}\otimes
t^M_{U_2},\, U_1\uplus U_2=U$ be an element of
$(\Arb_M\cdot\Arb_M')[U]$, $t^M_{U_1}\in\Arb_M[U_1]$ and
$t^M_{U_2}\in\Arb_{M}'[U_2]=\Arb_{M}[U_2\uplus\{\star\}]$ being
$M$-enriched trees. We are going to describe the tree
$t^M_U=\star(t^M_{U_1}\otimes t^M_{U_2})\in \Arb_M[U]$. Assume
first that the ghost element of $t^M_{U_2}$ is the root. Then
$t^M_U$ is obtained by imbedding $t^M_{U_1}$ into $t^M_{U_2}$
using the following procedure. Place the tree $t^M_{U_1}$ into
$t_{U_2}^M$ by replacing the (ghost) root of it by the root $r$ of
$t^M_{U_1}$ and enriching its fiber
$t^{-1}_{U}(r)=t_{U_1}^{-1}(r)\uplus t_{U_2}^{-1}(\star)$ with the
vector $\nu(m_{r}\otimes m_{\star})$, $m_{r}\in
M[t_{U_1}^{-1}(r)]$ and $m_{\star}\in M[t_{U_2}^{-1}(\star)]$
being the structures enriching the fibers of the roots of
$t_{U_1}$ and $t_{U_2}$ respectively. When the ghost element of
$t^M_{U_2}$ is not its root, imbed $t^M_{U_1}$ in $t^M_{U_2}$ by
using the previous procedure with the subtree of $t^M_{U_2}$
formed by the descendants of $\star$ ($\star$ will be the root of
it). By the associativity of $\nu$ we have the commutation of the
partial compositions.

An $M$-enriched tree will be called {\em small} if all the
vertices different from the root are leaves. Denote by $\stree_M$
the species of $M$-enriched small trees. $\stree_M$ is clearly
isomorphic to $X M$.
\begin{figure}
\begin{center}
% Generated with LaTeXDraw 1.9.5
% Mon May 19 08:47:54 VET 2008
% \usepackage[usenames,dvipsnames]{pstricks}
% \usepackage{epsfig}
% \usepackage{pst-grad} % For gradients
% \usepackage{pst-plot} % For axes
\scalebox{1} {
\begin{pspicture}(0,-2.3)(16.461876,2.3)
\definecolor{color1496}{rgb}{0.8,0.0,0.8}
\psline[linewidth=0.04cm](2.2,0.56)(1.12,-0.64)
\psline[linewidth=0.04cm](2.3,0.42)(2.06,-0.94)
\psline[linewidth=0.04cm](2.48,0.52)(3.18,-0.9)
\psline[linewidth=0.04cm](2.42,0.76)(3.54,1.88)
\psline[linewidth=0.04cm](3.56,1.9)(3.86,0.76)
\psline[linewidth=0.04cm](3.56,1.86)(3.2,0.72)
\pscircle[linewidth=0.04,linecolor=color1496,dimen=outer](1.0,-0.78){0.2}
\pscircle[linewidth=0.04,linecolor=color1496,dimen=outer](2.06,-1.1){0.2}
\pscircle[linewidth=0.04,linecolor=color1496,dimen=outer](3.28,-1.04){0.2}
\pscircle[linewidth=0.04,linecolor=color1496,dimen=outer](3.16,0.54){0.2}
\pscircle[linewidth=0.04,linecolor=color1496,dimen=outer](3.92,0.58){0.2}
\psarc[linewidth=0.04]{-cc}(3.53,1.77){0.43}{232.76517}{349.5085}
\usefont{T1}{ptm}{m}{n}
\rput(4.001406,1.87){$M$}
\psarc[linewidth=0.04]{cc-}(2.46,0.06){0.5}{212.47119}{344.0546}
\usefont{T1}{ptm}{m}{n}
\rput(2.9214063,0.09){$M$}
\usefont{T1}{ptm}{m}{n}
\rput(5.101406,0.07){$\stackrel{\star}{\longmapsto}$}
\psline[linewidth=0.04cm](7.98,0.98)(6.84,0.0)
\psline[linewidth=0.04cm](7.96,0.94)(7.42,-0.2)
\psline[linewidth=0.04cm](6.42,0.4)(8.0,1.0)
\psline[linewidth=0.04cm](7.98,0.98)(8.18,-0.12)
\psline[linewidth=0.04cm](7.98,0.98)(8.98,0.12)
\pscircle[linewidth=0.04,linecolor=color1496,dimen=outer](6.26,0.3){0.2}
\pscircle[linewidth=0.04,linecolor=color1496,dimen=outer](6.7,-0.1){0.2}
\pscircle[linewidth=0.04,linecolor=color1496,dimen=outer](9.14,0.02){0.2}
\psarc[linewidth=0.04]{-cc}(8.01,0.91){0.43}{198.43495}{6.8427734}
\usefont{T1}{ptm}{m}{n}
\rput(8.241406,1.23){$\nu(M M)$}
\usefont{T1}{ptm}{m}{n}
\rput(6.2414064,0.33){$a$}
\usefont{T1}{ptm}{m}{n}
\rput(9.121407,0.05){$c$}
\usefont{T1}{ptm}{m}{n}
\rput(6.6514063,-0.09){$d$}
\usefont{T1}{ptm}{m}{n}
\rput(6.1314063,0.73){$X$}
\pscircle[linewidth=0.04,linecolor=color1496,dimen=outer](7.38,-0.38){0.2}
\pscircle[linewidth=0.04,linecolor=color1496,dimen=outer](8.18,-0.32){0.2}
\usefont{T1}{ptm}{m}{n}
\rput(7.3414063,-0.33){$e$}
\usefont{T1}{ptm}{m}{n}
\rput(8.171406,-0.33){$b$}
\pscircle[linewidth=0.04,linecolor=color1496,dimen=outer](2.35,0.61){0.19}
\usefont{T1}{ptm}{m}{n}
\rput(2.3314064,0.55){$\star$}
\usefont{T1}{ptm}{m}{n}
\rput(3.1314063,0.55){$b$}
\usefont{T1}{ptm}{m}{n}
\rput(2.0114062,-1.09){$d$}
\usefont{T1}{ptm}{m}{n}
\rput(0.96140623,-0.75){$a$}
\usefont{T1}{ptm}{m}{n}
\rput(3.8814063,0.59){$c$}
\usefont{T1}{ptm}{m}{n}
\rput(3.2414062,-1.01){$e$}
\psframe[linewidth=0.04,framearc=0.15,dimen=outer](15.1,2.3)(0.0,-2.3)
\usefont{T1}{ptm}{m}{n}
\rput(4.811406,-1.61){$\star:XM\cdot M \rightarrow  XM$}
\usefont{T1}{ptm}{m}{n}
\rput(1.2514062,0.03){$X$}
\psline[linewidth=0.04cm](11.0,0.5)(10.26,-0.14)
\psline[linewidth=0.04cm](11.82,0.38)(11.34,-0.82)
\psline[linewidth=0.04cm](11.9,0.4)(12.02,-0.84)
\psline[linewidth=0.04cm](11.16,0.78)(12.28,1.9)
\psline[linewidth=0.04cm](12.3,1.92)(12.6,0.78)
\psline[linewidth=0.04cm](12.3,1.88)(11.94,0.74)
\pscircle[linewidth=0.04,linecolor=color1496,dimen=outer](10.13,-0.25){0.21}
\pscircle[linewidth=0.04,linecolor=color1496,dimen=outer](11.38,-1.02){0.2}
\pscircle[linewidth=0.04,linecolor=color1496,dimen=outer](12.02,-1.02){0.2}
\pscircle[linewidth=0.04,linecolor=color1496,dimen=outer](11.9,0.56){0.2}
\pscircle[linewidth=0.04,linecolor=color1496,dimen=outer](12.66,0.6){0.2}
\psarc[linewidth=0.04]{-cc}(12.27,1.79){0.43}{235.7843}{330.5241}
\usefont{T1}{ptm}{m}{n}
\rput(12.7014065,1.81){$M$}
\usefont{T1}{ptm}{m}{n}
\rput(12.281406,-0.03){$M$}
\usefont{T1}{ptm}{m}{n}
\rput(14.021406,0.05){$\stackrel{\star}{\longmapsto} 0$}
\pscircle[linewidth=0.04,linecolor=color1496,dimen=outer](11.09,0.63){0.19}
\usefont{T1}{ptm}{m}{n}
\rput(11.871407,0.47){$\star$}
\usefont{T1}{ptm}{m}{n}
\rput(11.071406,0.61){$b$}
\usefont{T1}{ptm}{m}{n}
\rput(11.331407,-1.01){$d$}
\usefont{T1}{ptm}{m}{n}
\rput(10.101406,-0.23){$a$}
\usefont{T1}{ptm}{m}{n}
\rput(12.621407,0.61){$c$}
\usefont{T1}{ptm}{m}{n}
\rput(11.981406,-0.99){$e$}
\usefont{T1}{ptm}{m}{n}
\rput(12.501407,-1.67){$\star:X^2 M M '\rightarrow  0$}
\usefont{T1}{ptm}{m}{n}
\rput(10.4114065,0.35){$X$}
\usefont{T1}{ptm}{m}{n}
\rput(11.371407,1.41){$X$}
\psarc[linewidth=0.04](11.84,-0.08){0.32}{155.55605}{341.56506}
\psline[linewidth=0.04cm](9.6,2.28)(9.58,-2.24)
\end{pspicture}
}
\end{center}\caption{The partial composition for the operad of
small trees.}\label{partialstree}
\end{figure}
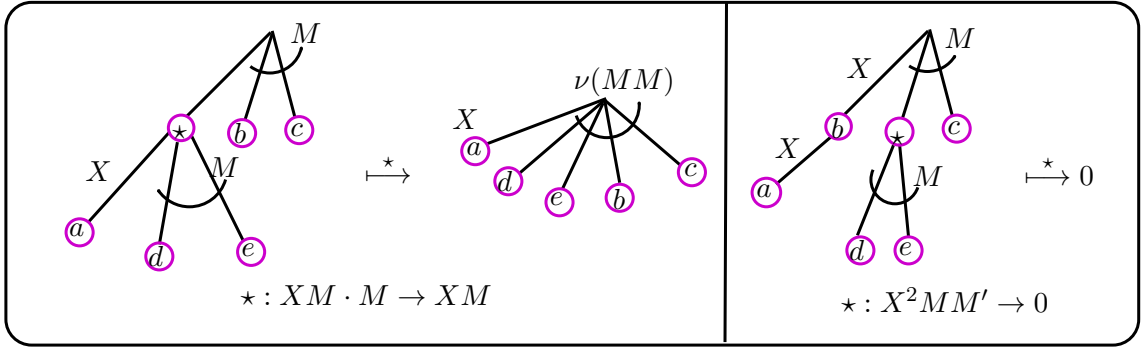

By the product rule for the derivative (see \cite{Joyal2}), we
have that $\stree_M\stree'_M=XM(M+XM')=XM^2+X^2MM'$. Hence, if $M$
is a monoid $(M,\nu)$, $\stree_M$ has a natural operad structure
given by the partial composition $\star=I_X.\nu+0$

\be\label{starstree} \star:XM^2+X^2MM'\longrightarrow XM\eeq The
commutation of partial compositions follows from the associativity
of the product $\nu$. The proof of the following proposition is
easy, and left to reader.
\begin{prop} The operad $\stree_M$ is quadratic if and only if $M$ is a
quadratic monoid. Moreover, we have that if $M=\mathcal{M}(F,R)$,
then $\stree_M=\mathcal{O}(XF,XR+X^2FF')$ and conversely.
\end{prop}

\begin{theo}\label{main}Let $M$ be a species (of any kind) of the
form $M=1+M_+$. Then, the species $\Sc_{XM}$ and $\Arb_{\LL(M_+)}$
are isomorphic. Moreover, if $M$ is a quadratic monoid, we have
the following isomorphism of $\dg$-tensor species \be
\label{Bao-ba}\Bao(\stree_M)=\Arb_{\Ba(M)}\,.\eeq
\end{theo}
\begin{proof}  An $XM$-enriched tree $t$ in $\Sc_{ XM}[U]$ is a rooted tree,
where each internal node $v\in \mathrm{Iv}(t)$ is decorated with
an element of $M_+[\pi_v-\{v'\}]\equiv \KK\{v'\} \otimes
M_+[\pi_v-\{v'\}]\subseteq (XM_+)[\pi_v]=\bigoplus_{v'\in
\pi_v}\KK\{v'\} \otimes M_+[\pi_v-\{v'\}],$ $\pi_v$ standing for
the set of sons of $v$. So, at each internal vertex $v$, a
preferred son $v'=p(v)$ is chosen and an element $m_v$ is put on
the rest of them ($m_v\in M[\pi_v-\{p(v)\}]$). By starting at the
root of t, we can construct a unique path from the root to a leaf
(the {\em distinguished leaf} of $t$), by letting the successor of
an internal node be its preferred son. We shall call this path the
{\em main spine} of $t$. For each internal vertex $v$ in the main
spine, and for each
 $w\in\pi_v-\{p(v)\} $, the descendants of $w$ form a tree in
$\Sc_{XM}[U_{v'}]$. The $XM$-enriched tree $t$ is then an element
of the species $X [\LL(M)\circ \Sc_{XM}][U]$ (see figure
\ref{AL}). Since the $M$-enriched trees generates the vector space
$\Sc_{XM}[U]$ for every finite set $U$, $\Sc_{XM}$ satisfies the
implicit equation \be \Sc_{ XM}=X[\LL( M_+)\circ \Sc_{ XM}].\eeq
This is the same implicit equation defining the species
 $\Arb_{\LL( M_+)}$. By the
 theorem of implicit equations for species (\cite{Joyal2}, see also \cite{BLL}),
 straightforwardly extended to this context, we obtain
 the result. As a consequence of that we obtain

 \be\label{gid}(\stree_M^{\grader})^{\langle\mils 1\rangle}=\Sc_{\mils XM^{\grader}}=\Sc_{X(\mils
M^{\grader})}=\Arb_{\LL(\mils M^{\grader})}=\Arb_{(
M^{\grader})^{\mils 1}}.\eeq
\begin{figure}\begin{center}
% Generated with LaTeXDraw 1.9.5
% Thu Sep 04 10:56:59 BOT 2008
% \usepackage[usenames,dvipsnames]{pstricks}
% \usepackage{epsfig}
% \usepackage{pst-grad} % For gradients
% \usepackage{pst-plot} % For axes
\scalebox{1.2} % Change this value to rescale the drawing.
{
\begin{pspicture}(0,-5.37)(9.26,5.37)
\definecolor{color2203}{rgb}{0.4,0.0,0.6}
\rput{-90.0}(5.7009373,-2.3590624){\pscircle[linewidth=0.04,dimen=outer,fillstyle=solid,fillcolor=Magenta](1.6709375,-4.03){0.14}}
\rput{-90.0}(7.9409375,1.5609375){\pscircle[linewidth=0.04,dimen=outer,fillstyle=solid,fillcolor=Magenta](4.7509375,-3.19){0.14}}
\rput{-90.0}(6.7409377,0.4009375){\pscircle[linewidth=0.04,dimen=outer,fillstyle=solid,fillcolor=Magenta](3.5709374,-3.17){0.14}}
\pscircle[linewidth=0.04,dimen=outer,fillstyle=solid,fillcolor=Magenta](3.2109375,-4.53){0.14}
\pscircle[linewidth=0.04,dimen=outer,fillstyle=solid,fillcolor=Magenta](3.6909375,-4.79){0.14}
\pscircle[linewidth=0.04,dimen=outer,fillstyle=solid,fillcolor=Magenta](4.5109377,-4.09){0.14}
\pscircle[linewidth=0.04,dimen=outer,fillstyle=solid,fillcolor=Magenta](7.1509376,-4.21){0.14}
\pscircle[linewidth=0.04,dimen=outer,fillstyle=solid,fillcolor=Magenta](2.5509374,-2.65){0.14}
\pscircle[linewidth=0.04,dimen=outer,fillstyle=solid,fillcolor=Magenta](7.2109375,-2.61){0.14}
\pscircle[linewidth=0.04,dimen=outer,fillstyle=solid,fillcolor=Magenta](2.3909376,-4.29){0.14}
\pscircle[linewidth=0.04,dimen=outer,fillstyle=solid,fillcolor=Magenta](7.8709373,-3.99){0.14}
\pscircle[linewidth=0.04,dimen=outer,fillstyle=solid,fillcolor=Magenta](6.1509376,-4.47){0.14}
\pscircle[linewidth=0.04,dimen=outer,fillstyle=solid,fillcolor=Magenta](6.1509376,-3.21){0.14}
\pscircle[linewidth=0.04,dimen=outer,fillstyle=solid,fillcolor=Magenta](1.1309375,-3.49){0.14}
\pscircle[linewidth=0.04,dimen=outer,fillstyle=solid,fillcolor=Magenta](4.6909375,-1.59){0.14}
\usefont{T1}{ptm}{m}{n}
\rput(4.3323436,-1.54){$a$}
\usefont{T1}{ptm}{m}{n}
\rput(2.2023437,-2.52){$b$}
\usefont{T1}{ptm}{m}{n}
\rput(0.77234375,-3.68){$c$}
\usefont{T1}{ptm}{m}{n}
\rput(1.4023438,-4.26){$d$}
\usefont{T1}{ptm}{m}{n}
\rput(2.0923438,-4.34){$e$}
\usefont{T1}{ptm}{m}{n}
\rput(3.2123437,-3.18){$f$}
\usefont{T1}{ptm}{m}{n}
\rput(2.8723438,-4.7){$g$}
\usefont{T1}{ptm}{m}{n}
\rput(3.4623437,-5.06){$h$}
\usefont{T1}{ptm}{m}{n}
\rput(4.632344,-4.44){$i$}
\usefont{T1}{ptm}{m}{n}
\rput(4.4223437,-3.08){$j$}
\usefont{T1}{ptm}{m}{n}
\rput(5.822344,-3.22){$k$}
\usefont{T1}{ptm}{m}{n}
\rput(5.7923436,-4.5){$m$}
\usefont{T1}{ptm}{m}{n}
\rput(6.842344,-2.76){$n$}
\usefont{T1}{ptm}{m}{n}
\rput(6.862344,-4.4){$o$}
\usefont{T1}{ptm}{m}{n}
\rput(3.6223438,-1.76){$M_+$}
\usefont{T1}{ptm}{m}{n}
\rput(7.722344,-4.2){$p$}
\psline[linewidth=0.04cm](4.6,-1.63)(2.64,-2.57)
\psline[linewidth=0.04cm](4.6,-1.69)(3.58,-3.05)
\psline[linewidth=0.04cm](4.68,-1.71)(4.74,-3.05)
\psline[linewidth=0.04cm](4.76,-1.65)(6.12,-3.09)
\psline[linewidth=0.04cm](4.8,-1.57)(7.1,-2.57)
\psline[linewidth=0.04cm](2.46,-2.69)(1.22,-3.45)
\psline[linewidth=0.04cm](2.52,-2.75)(1.76,-3.93)
\psline[linewidth=0.04cm](2.58,-2.77)(2.4,-4.17)
\psline[linewidth=0.04cm](3.48,-3.25)(3.22,-4.39)
\psline[linewidth=0.04cm](3.58,-3.27)(3.68,-4.67)
\psline[linewidth=0.04cm](3.64,-3.23)(4.42,-4.03)
\psline[linewidth=0.04cm](6.16,-3.31)(6.14,-4.35)
\psline[linewidth=0.04cm](7.16,-2.71)(7.12,-4.09)
\psline[linewidth=0.04cm](7.28,-2.69)(7.8,-3.89)
\psarc[linewidth=0.02](4.37,-1.98){0.55}{170.21759}{317.8624}
\psarc[linewidth=0.02](5.46,-2.51){0.5}{303.69006}{347.00537}
\psarc[linewidth=0.02](5.54,-2.01){0.5}{326.30994}{14.036243}
\psarc[linewidth=0.02](2.01,-3.24){0.55}{170.83765}{268.02505}
\psarc[linewidth=0.02](2.58,-3.39){0.5}{230.71059}{278.74615}
\psarc[linewidth=0.02](3.46,-3.79){0.5}{228.81407}{308.65982}
\psarc[linewidth=0.02](6.08,-3.53){0.5}{258.69006}{296.56506}
\psarc[linewidth=0.02](7.32,-3.01){0.5}{230.71059}{323.97263}
\usefont{T1}{ptm}{m}{n}
\rput(1.4823438,-2.94){$M_+$}
\usefont{T1}{ptm}{m}{n}
\rput(2.9023438,-3.68){$M_+$}
\usefont{T1}{ptm}{m}{n}
\rput(4.5223436,-3.54){$M_+$}
\usefont{T1}{ptm}{m}{n}
\rput(4.0823436,-4.1){$M_+$}
\usefont{T1}{ptm}{m}{n}
\rput(6.5823436,-3.72){$M_+$}
\usefont{T1}{ptm}{m}{n}
\rput(7.9023438,-3.08){$M_+$}
\usefont{T1}{ptm}{m}{n}
\rput(6.242344,-1.64){$M_+$}
\usefont{T1}{ptm}{m}{n}
\rput(6.202344,-2.46){$M_+$}
\usefont{T1}{ptm}{m}{n}
\rput(4.6714063,-0.96){$\Arb_{\LL(M_+)}$}
\rput{-90.0}(2.8009374,9.900937){\pscircle[linewidth=0.04,dimen=outer,fillstyle=solid,fillcolor=Magenta](6.3509374,3.55){0.14}}
\rput{-90.0}(2.5009375,9.360937){\pscircle[linewidth=0.04,dimen=outer,fillstyle=solid,fillcolor=Magenta](5.9309373,3.43){0.14}}
\psline[linewidth=0.04cm](5.8909373,3.89)(6.2509375,3.61)
\rput{-90.0}(2.0009375,8.940937){\pscircle[linewidth=0.04,dimen=outer,fillstyle=solid,fillcolor=Magenta](5.4709377,3.47){0.14}}
\psline[linewidth=0.06cm,linecolor=color2203](5.8509374,3.91)(5.5509377,3.57)
\psline[linewidth=0.04cm](5.8709373,3.87)(5.9109373,3.55)
\psdots[dotsize=0.16,dotangle=-90.0](5.8709373,3.91)
\psline[linewidth=0.06cm,linecolor=color2203](3.1709375,1.69)(3.01,0.97)
\psline[linewidth=0.06cm,linecolor=color2203](3.0909376,2.23)(3.1709375,1.73)
\pscircle[linewidth=0.04,dimen=outer,fillstyle=solid,fillcolor=Magenta](3.9709375,1.85){0.14}
\pscircle[linewidth=0.04,dimen=outer,fillstyle=solid,fillcolor=Magenta](2.9709375,0.85){0.14}
\psdots[dotsize=0.16](3.1709375,1.69)
\psline[linewidth=0.04cm](3.1109376,2.21)(3.87,1.89)
\pscircle[linewidth=0.04,dimen=outer,fillstyle=solid,fillcolor=Magenta](3.6109376,0.77){0.14}
\pscircle[linewidth=0.04,dimen=outer,fillstyle=solid,fillcolor=Magenta](3.9309375,1.25){0.14}
\psline[linewidth=0.04cm](3.1709375,1.69)(3.81,1.29)
\psline[linewidth=0.04cm](3.1709375,1.69)(3.55,0.85)
\psarc[linewidth=0.027999999](3.4571137,2.216939){0.32711384}{296.3828}{329.333}
\psarc[linewidth=0.027999999](3.3171139,1.5569389){0.32711384}{253.52519}{3.3664606}
\psdots[dotsize=0.16](3.0909376,2.25)
\psline[linewidth=0.055999998cm,linecolor=color2203](1.9909375,2.33)(1.19,2.15)
\psline[linewidth=0.06cm,linecolor=color2203](2.4509375,2.71)(2.0109375,2.37)
\pscircle[linewidth=0.04,dimen=outer,fillstyle=solid,fillcolor=Magenta](1.0909375,2.13){0.14}
\pscircle[linewidth=0.04,dimen=outer,fillstyle=solid,fillcolor=Magenta](1.2909375,1.43){0.14}
\pscircle[linewidth=0.04,dimen=outer,fillstyle=solid,fillcolor=Magenta](1.8309375,1.19){0.14}
\psdots[dotsize=0.16](1.9909375,2.33)
\psline[linewidth=0.04cm](1.9709375,2.33)(1.35,1.49)
\psline[linewidth=0.04cm](2.0109375,2.31)(1.85,1.29)
\psdots[dotsize=0.16](2.5109375,2.75)
\pscircle[linewidth=0.04,dimen=outer,fillstyle=solid,fillcolor=Magenta](2.6109376,1.83){0.14}
\psarc[linewidth=0.027999999](1.9371139,2.0969388){0.32711384}{191.67723}{292.2593}
\psarc[linewidth=0.027999999](2.52,2.5340528){0.31}{246.89244}{301.99594}
\psline[linewidth=0.04cm](2.59,1.93)(2.51,2.73)
\pscircle[linewidth=0.04,dimen=outer,fillstyle=solid,fillcolor=Magenta](4.0709376,2.39){0.14}
\psline[linewidth=0.04cm](3.2909374,3.05)(2.55,2.79)
\psline[linewidth=0.04cm](3.2709374,3.05)(3.11,2.31)
\psline[linewidth=0.04cm](3.29,3.0468514)(3.99,2.47)
\usefont{T1}{ptm}{m}{n}
\rput(3.8623438,2.98){$M_+$}
\psarc[linewidth=0.027999999](3.35,2.9640527){0.36}{173.43134}{341.3335}
\usefont{T1}{ptm}{m}{n}
\rput(2.3023438,1.98){$M_+$}
\psline[linewidth=0.06cm,linecolor=color2203](4.4709377,3.03)(4.5909376,2.51)
\pscircle[linewidth=0.04,dimen=outer,fillstyle=solid,fillcolor=Magenta](5.4309373,2.65){0.14}
\pscircle[linewidth=0.04,dimen=outer,fillstyle=solid,fillcolor=Magenta](4.6509376,2.39){0.14}
\psdots[dotsize=0.16](4.4909377,2.97)
\psline[linewidth=0.04cm](4.3109374,3.51)(4.4509373,3.03)
\psline[linewidth=0.04cm](4.5109377,2.97)(5.29,2.67)
\psline[linewidth=0.1cm,linecolor=color2203](4.9909377,4.41)(4.2909374,3.51)
\psline[linewidth=0.1cm,linecolor=color2203](3.2309375,3.11)(2.2109375,3.23)
\psline[linewidth=0.1cm,linecolor=color2203](4.2709374,3.49)(3.3109374,3.13)
\psdots[dotsize=0.16](3.2909374,3.11)
\pscircle[linewidth=0.04,dimen=outer,fillstyle=solid,fillcolor=Magenta](2.0909376,3.25){0.14}
\psdots[dotsize=0.16](4.3109374,3.51)
\psdots[dotsize=0.24](5.0509377,4.45)
\psline[linewidth=0.04cm](5.0909376,4.45)(5.8509374,3.93)
\usefont{T1}{ptm}{m}{n}
\rput(1.8323437,3.28){$a$}
\usefont{T1}{ptm}{m}{n}
\rput(0.98234373,1.86){$b$}
\usefont{T1}{ptm}{m}{n}
\rput(1.2323438,1.22){$c$}
\usefont{T1}{ptm}{m}{n}
\rput(1.8423438,0.92){$d$}
\usefont{T1}{ptm}{m}{n}
\rput(2.6323438,1.6){$e$}
\usefont{T1}{ptm}{m}{n}
\rput(2.6723437,0.9){$f$}
\usefont{T1}{ptm}{m}{n}
\rput(3.3523438,0.82){$g$}
\usefont{T1}{ptm}{m}{n}
\rput(3.9023438,0.94){$h$}
\usefont{T1}{ptm}{m}{n}
\rput(4.1523438,1.78){$i$}
\usefont{T1}{ptm}{m}{n}
\rput(4.322344,2.22){$j$}
\usefont{T1}{ptm}{m}{n}
\rput(4.882344,2.36){$k$}
\usefont{T1}{ptm}{m}{n}
\rput(5.532344,2.38){$m$}
\usefont{T1}{ptm}{m}{n}
\rput(5.6223435,3.26){$n$}
\usefont{T1}{ptm}{m}{n}
\rput(6.0623436,3.24){$o$}
\usefont{T1}{ptm}{m}{n}
\rput(4.662344,3.46){$M_+$}
\usefont{T1}{ptm}{m}{n}
\rput(5.182344,3.0){$M_+$}
\usefont{T1}{ptm}{m}{n}
\rput(2.8823438,2.48){$M_+$}
\usefont{T1}{ptm}{m}{n}
\rput(3.7623436,1.68){$M_+$}
\usefont{T1}{ptm}{m}{n}
\rput(6.4423437,3.96){$M_+$}
\pscustom[linewidth=0.02,linestyle=dashed,dash=0.16cm 0.16cm] {
\newpath
\moveto(1.8709375,3.01) \lineto(1.9009376,3.01)
\curveto(1.9159375,3.01)(1.9659375,3.005)(2.0009375,3.0)
\curveto(2.0359375,2.995)(2.0959375,2.99)(2.1209376,2.99)
\curveto(2.1459374,2.99)(2.2059374,3.0)(2.2409375,3.01)
\curveto(2.2759376,3.02)(2.3359375,3.06)(2.3609376,3.09)
\curveto(2.3859375,3.12)(2.4359374,3.21)(2.4609375,3.27)
\curveto(2.4859376,3.33)(2.5109375,3.43)(2.5109375,3.47)
\curveto(2.5109375,3.51)(2.5109375,3.58)(2.5109375,3.61)
\curveto(2.5109375,3.64)(2.5109375,3.68)(2.5109375,3.71) }
\usefont{T1}{ptm}{m}{n}
\rput(2.0423439,3.6){$X$}
\usefont{T1}{ptm}{m}{n}
\rput(6.6023436,3.48){$p$}
\usefont{T1}{ptm}{m}{n}
\rput(5.802344,4.42){$M_+$}
\pscustom[linewidth=0.016,linestyle=dashed,dash=0.16cm 0.16cm] {
\newpath
\moveto(4.87,4.519259) \lineto(4.86,4.4757404)
\curveto(4.855,4.4539814)(4.85,4.410463)(4.85,4.388704)
\curveto(4.85,4.366945)(4.85,4.323426)(4.85,4.3016667)
\curveto(4.85,4.2799077)(4.85,4.232037)(4.85,4.205926)
\curveto(4.85,4.179815)(4.855,4.1275926)(4.86,4.1014814)
\curveto(4.865,4.0753703)(4.875,4.031852)(4.88,4.0144444)
\curveto(4.885,3.997037)(4.905,3.9535186)(4.92,3.9274075)
\curveto(4.935,3.9012964)(4.965,3.8577778)(4.98,3.8403704)
\curveto(4.995,3.822963)(5.02,3.7837963)(5.03,3.762037)
\curveto(5.04,3.7402778)(5.065,3.6880555)(5.08,3.6575925)
\curveto(5.095,3.6271296)(5.12,3.5618517)(5.13,3.5270371)
\curveto(5.14,3.4922223)(5.155,3.4356482)(5.16,3.4138892)
\curveto(5.165,3.39213)(5.185,3.352963)(5.2,3.3355553)
\curveto(5.215,3.3181481)(5.26,3.287685)(5.29,3.2746296)
\curveto(5.32,3.261574)(5.36,3.2354627)(5.37,3.2224073)
\curveto(5.38,3.2093518)(5.41,3.1875927)(5.43,3.1788893)
\curveto(5.45,3.1701856)(5.495,3.1527777)(5.52,3.1440742)
\curveto(5.545,3.1353705)(5.595,3.1223147)(5.62,3.1179628)
\curveto(5.645,3.1136112)(5.695,3.1049073)(5.72,3.1005554)
\curveto(5.745,3.0962036)(5.795,3.0787964)(5.82,3.0657406)
\curveto(5.845,3.052685)(5.895,3.030926)(5.92,3.0222223)
\curveto(5.945,3.0135186)(5.99,2.996111)(6.01,2.9874072)
\curveto(6.03,2.9787035)(6.095,2.97)(6.14,2.97)
\curveto(6.185,2.97)(6.275,2.97)(6.32,2.97)
\curveto(6.365,2.97)(6.47,2.97)(6.53,2.97)
\curveto(6.59,2.97)(6.685,2.97)(6.72,2.97)
\curveto(6.755,2.97)(6.83,2.9830554)(6.87,2.9961112)
\curveto(6.91,3.009167)(6.975,3.0439816)(7.0,3.065741)
\curveto(7.025,3.0875)(7.06,3.1397223)(7.07,3.1701853)
\curveto(7.08,3.2006483)(7.095,3.261574)(7.1,3.292037)
\curveto(7.105,3.3225)(7.11,3.3964813)(7.11,3.44)
\curveto(7.11,3.4835186)(7.11,3.5575)(7.11,3.5879629)
\curveto(7.11,3.6184258)(7.11,3.6880555)(7.11,3.7272222)
\curveto(7.11,3.766389)(7.105,3.8490741)(7.1,3.8925927)
\curveto(7.095,3.9361112)(7.055,4.0275)(7.02,4.0753703)
\curveto(6.985,4.123241)(6.91,4.232037)(6.87,4.292963)
\curveto(6.83,4.353889)(6.765,4.436574)(6.74,4.4583335)
\curveto(6.715,4.4800925)(6.635,4.5410185)(6.58,4.5801854)
\curveto(6.525,4.619352)(6.42,4.6846294)(6.37,4.7107406)
\curveto(6.32,4.7368517)(6.235,4.78037)(6.2,4.7977777)
\curveto(6.165,4.815185)(6.1,4.8369446)(6.07,4.841296)
\curveto(6.04,4.8456483)(5.95,4.85)(5.89,4.85)
\curveto(5.83,4.85)(5.735,4.85)(5.7,4.85)
\curveto(5.665,4.85)(5.565,4.85)(5.5,4.85)
\curveto(5.435,4.85)(5.325,4.85)(5.28,4.85)
\curveto(5.235,4.85)(5.16,4.8456483)(5.13,4.841296)
\curveto(5.1,4.8369446)(5.045,4.819537)(5.02,4.8064814)
\curveto(4.995,4.793426)(4.955,4.762963)(4.94,4.7455554)
\curveto(4.925,4.728148)(4.905,4.6889815)(4.9,4.6672225)
\curveto(4.895,4.645463)(4.895,4.6062965)(4.9,4.5888886)
\curveto(4.905,4.5714817)(4.905,4.536667)(4.9,4.5192595)
\curveto(4.895,4.501852)(4.885,4.4757404)(4.87,4.44963) }
\psline[linewidth=0.03cm,linestyle=dashed,dash=0.16cm
0.16cm](3.49,3.67)(3.83,4.79)
\psline[linewidth=0.03cm,linestyle=dashed,dash=0.16cm
0.16cm](3.83,4.81)(4.18,3.99)
\psline[linewidth=0.03cm,linestyle=dashed,dash=0.16cm
0.16cm](3.83,4.81)(4.87,4.65)
\usefont{T1}{ptm}{m}{n}
\rput(3.7614062,5.02){$\LL$}
\usefont{T1}{ptm}{m}{n}
\rput(2.7314062,0.2){$M_+(\Sc_{XM})$}
\usefont{T1}{ptm}{m}{n}
\rput(6.291406,1.56){$M_+(\Sc_{XM})$}
\usefont{T1}{ptm}{m}{n}
\rput(7.6914062,2.76){$M_+(\Sc_{XM})$}
\psarc[linewidth=0.027999999](5.31,4.364053){0.36}{289.83572}{341.3335}
\psarc[linewidth=0.027999999](5.89,3.9440527){0.36}{256.85217}{341.3335}
\usefont{T1}{ptm}{m}{n}
\rput(3.7823439,2.24){$M_+$}
\psarc[linewidth=0.027999999](4.33,3.5440528){0.36}{246.32767}{314.12067}
\psarc[linewidth=0.027999999](4.79,3.0840528){0.36}{279.04608}{322.75238}
\pscustom[linewidth=0.0139999995,linestyle=dashed,dash=0.16cm
0.16cm] {
\newpath
\moveto(3.4510639,3.628205) \lineto(3.4114892,3.5968587)
\curveto(3.3917022,3.5811858)(3.3372872,3.5602884)(3.3026595,3.5550642)
\curveto(3.2680318,3.5498397)(3.1987765,3.5341668)(3.1641488,3.5237179)
\curveto(3.1295211,3.513269)(3.0751061,3.4871473)(3.0553193,3.4714744)
\curveto(3.035532,3.4558015)(2.97617,3.4192307)(2.9365957,3.398333)
\curveto(2.8970215,3.377436)(2.8426065,3.335641)(2.8277662,3.3147438)
\curveto(2.8129253,3.2938461)(2.7782977,3.252051)(2.7585108,3.2311535)
\curveto(2.7387235,3.2102563)(2.6942022,3.17891)(2.6694682,3.1684613)
\curveto(2.644734,3.1580126)(2.6002128,3.1214423)(2.5804255,3.0953205)
\curveto(2.5606384,3.0691986)(2.5210638,3.0274036)(2.501277,3.0117307)
\curveto(2.4814897,2.9960577)(2.441915,2.9647112)(2.4221277,2.9490385)
\curveto(2.4023404,2.9333656)(2.3578193,2.8915703)(2.333085,2.8654487)
\curveto(2.3083508,2.8393269)(2.258883,2.807981)(2.234149,2.8027563)
\curveto(2.2094147,2.797532)(2.155,2.7870834)(2.1253192,2.7818592)
\curveto(2.0956385,2.7766345)(2.0214362,2.766186)(1.976915,2.7609613)
\curveto(1.9323937,2.755737)(1.8631383,2.7505126)(1.8384042,2.7505126)
\curveto(1.8136702,2.7505126)(1.7493618,2.7452884)(1.7097872,2.7400641)
\curveto(1.6702127,2.7348397)(1.5712765,2.7191668)(1.511915,2.708718)
\curveto(1.4525533,2.6982691)(1.328883,2.6721473)(1.2645745,2.6564744)
\curveto(1.2002661,2.6408014)(1.0864894,2.593782)(1.0370213,2.5624359)
\curveto(0.98755324,2.5310898)(0.90345746,2.4631732)(0.8688298,2.4266028)
\curveto(0.8342021,2.3900323)(0.7896808,2.311667)(0.77978724,2.2698717)
\curveto(0.76989365,2.2280767)(0.76,2.1235895)(0.76,2.060897)
\curveto(0.76,1.998205)(0.76989365,1.8832691)(0.77978724,1.8310254)
\curveto(0.7896808,1.7787818)(0.8144148,1.6586218)(0.8292553,1.5907049)
\curveto(0.84409577,1.5227884)(0.86388296,1.4235257)(0.8688298,1.3921796)
\curveto(0.87377656,1.3608334)(0.8985106,1.2929169)(0.9182979,1.2563465)
\curveto(0.9380851,1.219776)(0.9776596,1.1518592)(0.9974469,1.1205127)
\curveto(1.0172342,1.0891665)(1.0518618,1.0264746)(1.0667021,0.99512815)
\curveto(1.0815425,0.9637817)(1.1112235,0.9219873)(1.1260638,0.91153806)
\curveto(1.1409042,0.9010895)(1.200266,0.86451906)(1.2447873,0.8383972)
\curveto(1.2893084,0.8122754)(1.3882446,0.77048033)(1.4426596,0.7548071)
\curveto(1.4970746,0.73913455)(1.6009575,0.7130127)(1.6504256,0.70256406)
\curveto(1.6998936,0.6921155)(1.8087234,0.68166685)(1.868085,0.68166685)
\curveto(1.9274467,0.68166685)(2.0610106,0.67644286)(2.1352127,0.6712183)
\curveto(2.2094147,0.66599363)(2.3726597,0.6503204)(2.4617023,0.63987184)
\curveto(2.5507445,0.6294232)(2.7139893,0.6189746)(2.7881916,0.6189746)
\curveto(2.8623939,0.6189746)(2.9910107,0.6033014)(3.0454254,0.5876282)
\curveto(3.0998404,0.57195556)(3.2136168,0.5406091)(3.2729788,0.5249359)
\curveto(3.3323407,0.5092627)(3.4411705,0.47269166)(3.4906385,0.45179445)
\curveto(3.5401065,0.4308972)(3.6242023,0.41)(3.65883,0.41)
\curveto(3.6934576,0.41)(3.787447,0.41)(3.8468084,0.41)
\curveto(3.9061701,0.41)(4.010053,0.4152246)(4.0545745,0.4204486)
\curveto(4.099096,0.42567262)(4.1831913,0.46746767)(4.222766,0.5040381)
\curveto(4.2623405,0.5406091)(4.3217025,0.6294226)(4.3414893,0.68166625)
\curveto(4.3612766,0.7339099)(4.390958,0.84884584)(4.4008512,0.91153806)
\curveto(4.4107447,0.97423095)(4.4255853,1.0891669)(4.430532,1.1414105)
\curveto(4.4354787,1.1936542)(4.4404254,1.2824683)(4.4404254,1.3190387)
\curveto(4.4404254,1.3556092)(4.4503193,1.4444232)(4.4602127,1.4966669)
\curveto(4.470106,1.5489105)(4.48,1.6429486)(4.48,1.6847436)
\curveto(4.48,1.7265387)(4.48,1.7944549)(4.48,1.8205768)
\curveto(4.48,1.8466986)(4.48,1.8989422)(4.48,1.9250641)
\curveto(4.48,1.951186)(4.470106,1.9982053)(4.4602127,2.0191028)
\curveto(4.4503193,2.04)(4.4354787,2.0974684)(4.430532,2.1340387)
\curveto(4.4255853,2.1706092)(4.4206386,2.2333014)(4.4206386,2.2594233)
\curveto(4.4206386,2.285545)(4.405798,2.3377883)(4.390958,2.3639102)
\curveto(4.3761168,2.390032)(4.336542,2.4475)(4.3118086,2.478846)
\curveto(4.2870746,2.5101922)(4.2475,2.5676603)(4.23266,2.593782)
\curveto(4.217819,2.6199038)(4.193085,2.6721473)(4.1831913,2.6982691)
\curveto(4.173298,2.724391)(4.1535106,2.7766345)(4.143617,2.8027563)
\curveto(4.1337233,2.8288782)(4.0990953,2.896795)(4.074362,2.9385898)
\curveto(4.049628,2.9803846)(3.9952128,3.0744233)(3.9655318,3.1266668)
\curveto(3.935851,3.1789105)(3.866596,3.2833972)(3.8270216,3.335641)
\curveto(3.787447,3.3878846)(3.728085,3.4610257)(3.7082977,3.481923)
\curveto(3.6885107,3.5028205)(3.653883,3.5446155)(3.6390424,3.5655127)
\curveto(3.624202,3.5864103)(3.5945213,3.6229808)(3.5796807,3.6386538)
\curveto(3.5648403,3.654327)(3.5252662,3.67)(3.500532,3.67)
\curveto(3.4757977,3.67)(3.44117,3.654327)(3.4114892,3.6073077) }
\psframe[linewidth=0.04,framearc=0.1,dimen=outer](9.26,5.37)(0.0,-5.37)
\pscustom[linewidth=0.04] {
\newpath
\moveto(3.48,-1.73) \lineto(3.52,-1.71) }
\pscustom[linewidth=0.02,linestyle=dashed,dash=0.16cm 0.16cm] {
\newpath
\moveto(4.272421,4.0120687) \lineto(4.3305264,3.9931033)
\curveto(4.359579,3.9836206)(4.4225264,3.9693966)(4.456421,3.9646552)
\curveto(4.4903154,3.9599137)(4.5581055,3.9267242)(4.592,3.8982759)
\curveto(4.6258945,3.8698275)(4.6791577,3.8318965)(4.6985264,3.8224137)
\curveto(4.717895,3.812931)(4.7614737,3.7892241)(4.785684,3.775)
\curveto(4.809895,3.7607758)(4.8534737,3.7323275)(4.872842,3.7181034)
\curveto(4.8922105,3.7038794)(4.9309473,3.675431)(4.950316,3.661207)
\curveto(4.9696846,3.6469827)(5.003579,3.6090517)(5.0181055,3.5853448)
\curveto(5.032632,3.5616379)(5.061684,3.514224)(5.0762105,3.490517)
\curveto(5.090737,3.4668102)(5.1197896,3.4241378)(5.134316,3.4051723)
\curveto(5.1488423,3.3862069)(5.182737,3.3387933)(5.2021055,3.310345)
\curveto(5.221474,3.2818964)(5.250526,3.2344825)(5.2602105,3.215517)
\curveto(5.269895,3.1965516)(5.3086314,3.1633618)(5.337684,3.149138)
\curveto(5.366737,3.134914)(5.424842,3.1206896)(5.4538946,3.1206896)
\curveto(5.4829473,3.1206896)(5.5313683,3.1159484)(5.5507374,3.111207)
\curveto(5.5701056,3.1064653)(5.6088424,3.0922415)(5.6282105,3.082759)
\curveto(5.647579,3.0732758)(5.7008424,3.0448272)(5.734737,3.0258617)
\curveto(5.7686315,3.0068963)(5.8122106,2.9689655)(5.8218946,2.95)
\curveto(5.831579,2.9310346)(5.8509474,2.874138)(5.860632,2.836207)
\curveto(5.870316,2.7982757)(5.88,2.7366378)(5.88,2.712931)
\curveto(5.88,2.6892242)(5.88,2.6418104)(5.88,2.6181037)
\curveto(5.88,2.5943966)(5.88,2.5469828)(5.88,2.5232759)
\curveto(5.88,2.4995692)(5.8557897,2.4284484)(5.831579,2.3810346)
\curveto(5.8073683,2.3336205)(5.7395787,2.2435346)(5.6959996,2.2008622)
\curveto(5.6524205,2.1581898)(5.57979,2.0823276)(5.5507374,2.0491378)
\curveto(5.5216846,2.0159483)(5.468421,1.963793)(5.44421,1.9448276)
\curveto(5.42,1.9258621)(5.3522105,1.8926728)(5.3086314,1.8784485)
\curveto(5.2650523,1.8642242)(5.182737,1.85)(5.144,1.85)
\curveto(5.1052637,1.85)(5.0423155,1.8547412)(5.0181055,1.8594828)
\curveto(4.993895,1.8642242)(4.9115796,1.8831897)(4.853474,1.897414)
\curveto(4.7953687,1.9116381)(4.722737,1.940086)(4.7082105,1.9543103)
\curveto(4.693684,1.9685346)(4.6549473,2.0112066)(4.630737,2.0396552)
\curveto(4.6065264,2.0681036)(4.567789,2.1155176)(4.5532627,2.134483)
\curveto(4.5387363,2.1534486)(4.5242105,2.2008622)(4.5242105,2.2293103)
\curveto(4.5242105,2.2577589)(4.519368,2.3288794)(4.5145264,2.3715518)
\curveto(4.5096846,2.4142241)(4.480632,2.494828)(4.456421,2.5327587)
\curveto(4.4322104,2.5706897)(4.3934736,2.6465516)(4.3789473,2.6844828)
\curveto(4.364421,2.722414)(4.3498945,2.7935345)(4.3498945,2.8267243)
\curveto(4.3498945,2.8599138)(4.3353686,2.9310346)(4.3208423,2.9689655)
\curveto(4.306316,3.0068963)(4.277263,3.082759)(4.262737,3.1206896)
\curveto(4.2482104,3.1586206)(4.2094736,3.2439654)(4.185263,3.2913795)
\curveto(4.1610527,3.3387933)(4.1174736,3.4336207)(4.098105,3.4810345)
\curveto(4.078737,3.5284486)(4.0545263,3.599569)(4.049684,3.623276)
\curveto(4.044842,3.6469827)(4.04,3.7038794)(4.04,3.7370691)
\curveto(4.04,3.770259)(4.04,3.8413794)(4.04,3.8793104)
\curveto(4.04,3.9172413)(4.0738945,3.9788795)(4.1077895,4.0025864)
\curveto(4.141684,4.026293)(4.1997895,4.05)(4.224,4.05)
\curveto(4.2482104,4.05)(4.282105,4.0405173)(4.3111577,4.0120687)
} \psarc[linewidth=0.02](3.76,-3.55){0.5}{303.69006}{347.00537}
\end{pspicture}
}
\end{center}
\caption{The isomorphism $\Sc_{ X M}=X\cdot [\LL( M)\circ \Sc_{ X
M}]=\Arb_{\LL(M_+)}$.}\label{AL}
\end{figure}

 Equation (\ref{Bao-ba}) is equivalent to the implicit equation
 \be \label{bao}\Bao(\stree_M)=X\Ba(M)(\Bao(\stree_M)).\eeq
We rewrite it as follows

\be (\Sc_{\mils XM^{\grader}},\tilde{d})=X\cdot[(\LL(\mils
M_+^{\grader}),d)((\Sc_{\mils XM^{\grader}},\tilde{d}))]
 \eeq

\noindent where $\tilde{d}$ is as in equation (\ref{Baod}) and $d$
is as in (\ref{Badif}). By equation (\ref{gid}) we only have to
prove that the differentials match. To that end, choose a
Schr\"oder tree $t\in \Sc_{\mils XM^{\grader}}[U]$, and assume
that $t$ has $k$ internal vertices. Order them by indexing the
root as last element $v_k$ and the internal vertices on the main
spine $\{v_1,v_2,\dots,v_j\}$ enumerated as follows: $v_1=p(v_k)$
is he preferred son of the root, and $v_{i+1}=p(v_i)$,
$i=1,2,\dots,j-1$ (see figure (\ref{BaoBar})). Apply the same
procedure recursively on each subtree attached to the
(non-preferred) sons of the vertices on the main spine. By the
definition of partial composition on small trees (see equation
(\ref{starstree}) and figure (\ref{partialstree})) is easy to see
that the sum defining $\tilde{d}(t)$ restricted to the first $j$
terms (the vertices on the main spine) is equal to $d$ applied to
$m_{v_1}\otimes m_{v_2}\otimes\dots\otimes m_{v_j}$. Denote by
$\pi=\{v_{j+1},v_{j+2},\dots,v_s\}=\biguplus_{i=2}^{j}\tilde{\pi}_{v_i},$
$\tilde{\pi}_{v_i}=\pi_{v_i}-\{v_{i+1}\}$, $i=1,\dots,j$, the set
of sons of the vertices on the main spine that are outside of it,
 and by $t_{v_i}$,
the corresponding subtree of the descendants of $v_i$ in $t,$
$i=j+1,j+2,\dots, s.$ We have
\begin{eqnarray}
\tilde{d}(t)&=&d(m_{v_1}\otimes m_{v_2}\otimes\dots\otimes
m_{v_j})\otimes \otimes_{i=j+1}^{s}t_{v_i}\\&+&\sum_{i=j+1}^{s}\pm
m_{v_1}\otimes m_{v_2}\otimes\dots\otimes m_{v_j}\otimes
t_{v_j+1}\otimes\dots\otimes \tilde{d}(t_{v_i})\otimes\dots
\otimes t_{v_s}. \end{eqnarray} Then, we get \be (\Sc_{\mils
XM^{\grader}},\tilde{d})[U]=\bigoplus_{u\in
U}\bigoplus_{\pi\in\Pi[U-\{u\}]}(\LL(\mil
M_+^{\grader}),d)[\pi]\otimes\bigotimes_{B\in\pi}(\Sc_{\mils
XM^{\grader}},\tilde{d})[B].\eeq

\end{proof}

\begin{cor}\label{genchapoton}Let $M$ be a quadratic monoid. Then
we have
 \be \label{dualidad}\stree_M^{\coop}=\Arb_{M^{\coop .}}.\eeq
Moreover, $M$ is a Koszul monoid if and only if $\stree_M$ is a
Koszul operad.
\end{cor}

 \begin{proof}Let $G$ be a $\dg$-tensor species of the form
$G=1+G_+$. From the properties of the functor $\Ho$ we have that
\be \Ho\Arb_{G}=\Ho X\cdot\left(\Ho
G\circ\Ho\Arb_G\right)=X\cdot\left(\Ho G\circ\Ho\Arb_G\right),\eeq
\noindent hence, $\Ho\Arb_{G}=\Arb_{\Ho G}$.

From equation (\ref{Bao-ba}) and the explicit form for the species
of enriched rooted trees (equation (\ref{etree})),
\begin{eqnarray}\label{Homology}\Ho \Bao(\stree_M)&=&\Arb_{\Ho \Ba(M)}\\
 \Ho\Bao(\stree_M)[U]&=&\bigoplus_{t\in\Arb[U]}\bigotimes_{u\in
U}\left(\Ho \Ba(M)\right)[t^{-1}u].\end{eqnarray} Then, by the
definition of tensor product in the category $\mgr$
\be\label{Homi}
H^i\Bao(\stree_{M})[U]=\bigoplus_{t\in\Arb[U]}\bigoplus_{\sum_{u\in
U}j_u=i}\left(\bigotimes_{u\in
U}H^{j_u}\Ba(M)[t^{-1}u]\right).\eeq From this it follows that
$H^0\Bao(\stree_M)=\Arb_{H^0 \Ba(M)},$ and hence equation
(\ref{dualidad}). Moreover if $\Ho\Ba(M)$ is concentrated in
degree zero, $\Ho\Bao(\stree_M)$ is also. Conversely, if
$\Ho\Bao(\stree_M)$ is concentrated in degree zero, by restricting
the direct sum in (\ref{Homi}) to the set $\stree_E[U]$ of small
trees, we obtain that for $i\neq 0$,
$X\cdot\left(H^i\Ba(M)\right)=0$. Then $H^i\Ba(M)=0$.\\
\end{proof}

\begin{cor}\label{genchapoton2} The monoid $M$ is Koszul if and only if $\Arb_{M}$ and $\Arb_{M^{!.}}$ are
both Koszul operads.
\end{cor}
\begin{proof}By duality according to Fresse
(see \cite{Fresse}, Lemmas 5.2.9. and 5.2.10.), $\stree_M$ is a
Koszul operad if and only if the operad
$(\stree_M^{\coop})^*=(\Arb_{M^{\coop .}})^*=\Arb_{M^{!. }}$ is
so. The result follows from corollary \ref{genchapoton} and from
the fact that $M$ and $M^!$ are simultaneously Koszul.
\end{proof}
Observe that if $M$ is generated by a species $F_1$ concentrated
in cardinality $1$, then $\stree_M$ is generated by $XF_1$,
concentrated in cardinality $2$. The quadratic dual of $\stree_M$,
according to Ginzburg-Kapranov \cite{G-K}, is equal to
$(\Arb_{M^{\coop}})^*\odot\Lambda=\Arb_{M^!}\odot\Lambda$.

We now restate Vallette result (see \cite{Vallette} Theorem 9), in
this context and using our terminology.

\begin{theo} Let $\Cop$ be a quadratic $c$-operad generated by
a species concentrated in some cardinality $k$,
$\Cop=\mathcal{O}(G_k,R)$. Then, $\Cop$ is Koszul if and only if
the M\"obius species $\Cop^{\langle -1\rangle}$ is Cohen-Macaulay.
\end{theo}
We obtain that, if $\Cop$ is as in the previous theorem, then
\begin{eqnarray} \Mob \Cop^{\langle -1\rangle}(x)&=&(\Cop^{\coop})^{\grader}(x)
\\ \Ch \Cop^{\langle -1\rangle}(\xx)&=&(\Ch\Cop^{\coop})^{\grader}(\xx).\end{eqnarray}
\noindent we have also
\begin{eqnarray}\label{odim}\sum_{c\in
\Cop^{\underline{k}}[n]}\mu(\hat{0},\{c\})&=&(-1)^k \mathrm{dim}(\Cop^{\coop })^{\underline{k}}[n],\\
\label{otr}\sum_{c\in
\Cop^{\underline{k}}[n],\Cop[\sigma]c=c}\mu([\hat{0},\{c\}]_{\sigma})&=&(-1)^k\mathrm{tr}(\Cop^{\coop
})^{\underline{k}}[\sigma].\end{eqnarray}

It is easy to check that if $M$ is a c-monoid, the operad $\Arb_M$
is a $c$-operad. The following corollary is trivial from the
previous theory.
\begin{cor}Let $M=\mathcal{M}(F_k,R)$ be a quadratic $c$-monoid generated
by a species $F_k$ concentrated in cardinality $k$. Then, the
following statements are equivalent
\begin{enumerate}
\item The M\"obius species $M^{-1}$ is Cohen-Macaulay. \item $M$
is Koszul.
 \item The
M\"obius species $\Arb_M^{\langle -1\rangle}$ is Cohen-Macaulay.
\item The operad $\Arb_M$ is Koszul.
\end{enumerate}
\end{cor}

\begin{figure}\begin{center}
% Generated with LaTeXDraw 1.9.5
% Thu Sep 04 10:35:17 BOT 2008
% \usepackage[usenames,dvipsnames]{pstricks}
% \usepackage{epsfig}
% \usepackage{pst-grad} % For gradients
% \usepackage{pst-plot} % For axes
\scalebox{1} % Change this value to rescale the drawing.
{
\begin{pspicture}(0,-4.65)(15.2,4.65)
\definecolor{color11b}{rgb}{1.0,0.0,0.6}
\definecolor{color190}{rgb}{0.8,0.0,0.8}
\psdots[dotsize=0.12](8.9609375,2.11)
\psdots[dotsize=0.12](8.260938,1.37)
\psdots[dotsize=0.12](6.5409374,0.25)
\psdots[dotsize=0.12](5.7609377,0.15)
\psdots[dotsize=0.12](4.7609377,0.15)
\psline[linewidth=0.04cm](8.940937,2.13)(8.260938,1.39)
\psline[linewidth=0.04cm](6.5409374,0.27)(5.7409377,0.17)
\psline[linewidth=0.04cm](5.7209377,0.15)(4.7609377,0.17)
\pscircle[linewidth=0.04,dimen=outer,fillstyle=solid,fillcolor=color11b](3.1609375,0.71){0.14}
\psline[linewidth=0.04cm](4.7209377,0.19)(3.2809374,0.65)
\psline[linewidth=0.04cm](8.940937,2.09)(9.68,0.95)
\psline[linewidth=0.04cm](8.9609375,2.15)(10.440937,1.71)
\psline[linewidth=0.04cm](8.980938,2.15)(11.160937,2.59)
\psline[linewidth=0.04cm](8.240937,1.35)(8.220938,0.07)
\psline[linewidth=0.04cm](8.260938,1.37)(9.500937,-0.05)
\psline[linewidth=0.04cm](6.5409374,0.29)(6.3809376,-2.27)
\psline[linewidth=0.04cm](6.5409374,0.27)(7.7209377,-1.47)
\psline[linewidth=0.04cm](5.7209377,0.19)(4.0409374,1.79)
\psline[linewidth=0.04cm](5.7409377,0.21)(5.9409375,2.37)
\psline[linewidth=0.04cm](4.7409377,0.09)(4.6809373,-1.59)
\psline[linewidth=0.04cm](4.7209377,0.13)(3.1609375,-0.37)
\psline[linewidth=0.04cm](6.5409374,0.29)(8.220938,1.35)
\psframe[linewidth=0.04,linecolor=color190,linestyle=dashed,dash=0.16cm
0.16cm,dimen=outer](4.9209375,2.79)(3.2209375,1.85)
\usefont{T1}{ptm}{m}{n}
\rput(3.9923437,2.32){$\mathscr{B}\mathrm{ar}(\stree_M)$}
\usefont{T1}{ptm}{m}{n}
\rput(8.092343,1.66){$v_1$}
\usefont{T1}{ptm}{m}{n}
\rput(6.0623436,-0.1){$v_{j-1}$}
\usefont{T1}{ptm}{m}{n}
\rput(6.5623436,0.62){$v_{j-2}$}
\usefont{T1}{ptm}{m}{n}
\rput(4.632344,0.48){$v_j$} \psdots[dotsize=0.05](6.7809377,0.97)
\psdots[dotsize=0.05](7.4409375,1.37)
\psdots[dotsize=0.05](7.1009374,1.19)
\usefont{T1}{ptm}{m}{n}
\rput(8.632343,2.4){$v_k$}
\usefont{T1}{ptm}{m}{n}
\rput(6.3123436,1.22){$\mil M_+$}
\psarc[linewidth=0.04](5.3509374,1.1){0.67}{23.198591}{187.69604}
\psarc[linewidth=0.04](4.5509377,-0.12){0.61}{161.56505}{310.9144}
\usefont{T1}{ptm}{m}{n}
\rput(5.1523438,-0.36){$\mil M_+$}
\psarc[linewidth=0.04](9.070937,2.3){0.67}{257.61923}{21.037512}
\usefont{T1}{ptm}{m}{n}
\rput(10.072344,2.8){$\mil M_+$}
\psarc[linewidth=0.04](6.7609377,-0.29){0.5}{208.61046}{14.534455}
\usefont{T1}{ptm}{m}{n}
\rput(7.2323437,-0.02){$\mil M_+$}
\psarc[linewidth=0.04](8.150937,1.32){0.67}{257.61923}{340.46335}
\usefont{T1}{ptm}{m}{n}
\rput(8.832344,1.26){$\mil M_+$}
\usefont{T1}{ptm}{m}{n}
\rput(3.0923438,1.1){$X$}
\psframe[linewidth=0.04,linecolor=color190,linestyle=dashed,dash=0.16cm
0.16cm,dimen=outer](5.6009374,-1.53)(3.9009376,-2.47)
\usefont{T1}{ptm}{m}{n}
\rput(4.6723437,-2.0){$\mathscr{B}\mathrm{ar}(\stree_M)$}
\psframe[linewidth=0.04,linecolor=color190,linestyle=dashed,dash=0.16cm
0.16cm,dimen=outer](12.860937,2.97)(11.160937,2.03)
\usefont{T1}{ptm}{m}{n}
\rput(11.9323435,2.5){$\mathscr{B}\mathrm{ar}(\stree_M)$}
\psframe[linewidth=0.04,linecolor=color190,linestyle=dashed,dash=0.16cm
0.16cm,dimen=outer](12.160937,1.99)(10.4609375,1.05)
\usefont{T1}{ptm}{m}{n}
\rput(11.232344,1.52){$\mathscr{B}\mathrm{ar}(\stree_M)$}
\psframe[linewidth=0.04,linecolor=color190,linestyle=dashed,dash=0.16cm
0.16cm,dimen=outer](11.360937,0.97)(9.660937,0.03)
\usefont{T1}{ptm}{m}{n}
\rput(10.4323435,0.5){$\mathscr{B}\mathrm{ar}(\stree_M)$}
\psframe[linewidth=0.04,linecolor=color190,linestyle=dashed,dash=0.16cm
0.16cm,dimen=outer](10.740937,-0.05)(9.040937,-0.99)
\usefont{T1}{ptm}{m}{n}
\rput(9.812344,-0.52){$\mathscr{B}\mathrm{ar}(\stree_M)$}
\psframe[linewidth=0.04,linecolor=color190,linestyle=dashed,dash=0.16cm
0.16cm,dimen=outer](4.0009375,-0.41)(2.3009374,-1.35)
\usefont{T1}{ptm}{m}{n}
\rput(3.0723438,-0.88){$\mathscr{B}\mathrm{ar}(\stree_M)$}
\psframe[linewidth=0.04,linecolor=color190,linestyle=dashed,dash=0.16cm
0.16cm,dimen=outer](7.0009375,3.31)(5.3009377,2.37)
\usefont{T1}{ptm}{m}{n}
\rput(6.072344,2.84){$\mathscr{B}\mathrm{ar}(\stree_M)$}
\psframe[linewidth=0.04,linecolor=color190,linestyle=dashed,dash=0.16cm
0.16cm,dimen=outer](7.2809377,-2.21)(5.5809374,-3.15)
\usefont{T1}{ptm}{m}{n}
\rput(6.3523436,-2.68){$\mathscr{B}\mathrm{ar}(\stree_M)$}
\psframe[linewidth=0.04,linecolor=color190,linestyle=dashed,dash=0.16cm
0.16cm,dimen=outer](9.140938,0.13)(7.4409375,-0.81)
\usefont{T1}{ptm}{m}{n}
\rput(8.212344,-0.34){$\mathscr{B}\mathrm{ar}(\stree_M)$}
\psframe[linewidth=0.04,linecolor=color190,linestyle=dashed,dash=0.16cm
0.16cm,dimen=outer](9.440937,-1.41)(7.7409377,-2.35)
\usefont{T1}{ptm}{m}{n}
\rput(8.512343,-1.88){$\mathscr{B}\mathrm{ar}(\stree_M)$}
\psdots[dotsize=0.22](4.76,0.13) \psdots[dotsize=0.22](5.72,0.17)
\psdots[dotsize=0.22](6.52,0.25) \psdots[dotsize=0.22](8.22,1.39)
\psdots[dotsize=0.22](8.94,2.11)
\psframe[linewidth=0.034,framearc=0.1,dimen=outer](15.2,4.65)(0.0,-4.65)
\end{pspicture}
}

\end{center}
\caption{Graphical representation of the equation
$\Bao(\stree_M)=X\cdot\Ba(M)(\Bao(\stree_M))$.}\label{BaoBar}
\end{figure}

\begin{ex} By the previous corollaries we obtain
the following. \begin{itemize}\item Since $E$ is a Koszul monoid,
the operads $\stree_{E}$, $\stree_{\Lambda}$, and
$\Arb=(\stree_{\Lambda}^{\coop})^*=\stree_{E}^!$ are Koszul. We
also have that $\Arb_{\Lambda}=(\stree_{E}^{\coop})^*$ is Koszul.
\item The monoid $\mathrm{Cosh}$ is Koszul and hence its quadratic
dual $\sum_{j=0}^{\infty}\mathcal{S}_{\tau_{2j}}$. Then,
$\stree_{\mathrm{Cosh}}$ and $\Arb_{\mathrm{Cosh}}$, respectively
the species of small trees and rooted trees where each vertex has
an even number of sons, are Koszul. So are the operads
$\stree_{\sum_{j=0}^{\infty}\mathcal{S}_{\tau_{2j}}}$, and
$\Arb_{\sum_{j=0}^{\infty}\mathcal{S}_{\tau_{2j}}}.$ \item $E\seg
E$ and $E\seb E$  are Koszul monoids. The operads $\stree_{E\seg
E}$, $\stree_{E\seb E}$, $\Arb_{E\seg E}$ and $\Arb_{E\seb E}$
 are Koszul.
 \item Since the $c-$monoids $E$, $E_{(k)}$, $E_{(k_1)}\seg
 E_{(k_2)},$
 $E_{(k_1)}\seg E_{(k_2)}\seg E_{(k_3)},\dots,$ are
 Koszul, the M\"obius species $\Arb^{\langle -1\rangle}$,
 $\Arb_{E_{(k)}}^{\langle -1\rangle}$,
 $\Arb_{E_{(k_1)}\seg E_{(k_2)}}^{\langle -1\rangle},\Arb_{E_{(k_1)}\seg E_{(k_2)}\seg E_{(k_3)}}
 ^{\langle -1\rangle},\dots,$ are
 Cohen-Macaulay.
\end{itemize}
\end{ex}
Koszulness for associative algebras, and hence for monoids in
species, is a property that is closed under many operations (see
\cite{Backelin}). Segre and Manin product are two among them. So,
corollaries \ref{genchapoton} and \ref{genchapoton2} give us a
wide class of Koszul operads.


\begin{thebibliography}{1}
\providecommand{\natexlab}[1]{#1}
\providecommand{\url}[1]{\texttt{#1}} \expandafter\ifx\csname
urlstyle\endcsname\relax
  \providecommand{\doi}[1]{doi: #1}\else
  \providecommand{\doi}{doi: \begingroup \urlstyle{rm}\Url}\fi

  \bibitem  {Marcelo} M. Aguiar and S. Mahajan, Monoidal
functors, species and Hopf algebras, in preparation.
\bibitem{Andre} D. Andr\'e, D\'eveloppement de sec x and tg x, C. R. Math. Acad.
Sci. Paris 88 (1879), 965-979

\bibitem{Backelin}J. Backelin, and R. Fröberg,  Koszul algebras,
Veronese subrings and rings with linear resolutions. Rev. Roumaine
Math. Pures Appl. 30 (1985), no. 2, 85-97.

\bibitem{BLL} F. Bergeron, G. Labelle and P. Leroux,
            Combinatorial Species
              and Tree-Like Structures, Encyclopedia of Mathematics
               and its Applications,
              volume 67, Cambridge University Press, Cambridge,
              1998.
\bibitem{Bjorner} A. Bj\"orner, A. Garsia, and R. Stanley, An
introduction to Cohen-Macaulay partially ordered sets in Ordered
sets (ed. I. Rival; Reidel, Dordrecht, 1982) 583-615.

\bibitem{Carlitz0} L. Carlitz, Enumeration of up-down sequences, Discrete Math.
4 (1973), 273-286.

 \bibitem{Carlitz}L. Carlitz, R. Scoville, and T. Vaughan,
Enumeration of pairs of sequences by rises, falls and levels.
Manuscripta Math. 19 (1976), no. 3, 211-243.

\bibitem{Chapoton1} F. Chapoton, An endofunctor de la cat\'egorie
des op\'erades, in ``Dialgebras and related Operads'', Lecture
Notes in Math., No. 1763, Springer-Verlag, 2001, 105-110.

\bibitem{Chapoton}F. Chapoton, and M. Livervet, Relating Two Hopf Algebras Built from An
Operad, by appear.


\bibitem {shroeder} R. Ehrenborg and M. M\'{e}ndez, Schröder parenthesizations and
chordates, J. Combin. Theory Ser. A 67 (1994), no. 2, 127-139.

\bibitem{Foulkes0}H. O. Foulkes, Enumeration of permutations with prescribed updown
and inversion sequences, Discrete Math. 15 (1976), 235-252.
\bibitem{Foulkes}H. Foulkes,
Tangent and secant numbers and representations of symmetric
groups, Discrete Math. 15 (1976), no. 4, 311-324.

\bibitem{Fresse}B. Fresse, Koszul duality of operads and homology of
partition posets, in ``Homotopy theory and its applications
(Evanston, 2002)'', Contemp. Math. 346 (2004), 115-215.



\bibitem{G-K}V. Ginzburg and M. Kapranov,
Koszul duality for operads, Duke Math. J. 76 (1994), no. 1,
203-272.

\bibitem{Hanlon}P. Hanlon, The fixed point partition lattices,
Pacific J. Math. 96 (1981), 319-341.

\bibitem{Joyal1}A. Joyal, Foncteurs analytiques et esp\`{e}ces de structures. in ``Combinatoire
 \'{e}num\'{e}rative'', Lecture Notes in Math., No. 1234, Springer-Verlag,
 1986, 126-159.

\bibitem{Joyal2} A. Joyal,  Une th\'eorie combinatoire des
s\'eries formelles, Adv. Math., 42 (1981), 1-82.

\bibitem{Manin}Y. Manin, Some remarks on Koszul algebras and
quantum groups. Ann. Inst. Fourier, 37, \# 4,191-205,1987.

\bibitem{Macdonald} I. Macdonald, Symmetric functions and Hall
polynomials (second edition), Oxford Mathematical Monographs,
Oxford University Press, 1995.

\bibitem{Maclane} S. Mac Lane, Categories for the Working
Mathematician (second edition), Graduate Texts in Mathematics
\textbf{5}, Springer Verlag,1998.

\bibitem {Miguel} M. M\'endez,
Monoides, c-monoides, especies de M\"{o}bius y co\'{a}lgebras,
Ph.D. thesis, Universidad Central de Venezuela, 1989.\\
http://www.ivic.ve/matematicas/documentos/tesisMiguelMendez.pdf

\bibitem{Rota1}S. Joni, G.-C. Rota, and B. Sagan.
From sets to functions: three elementary examples. Discrete Math.
37 (1981), no. 2-3, 193-202.

\bibitem{Julia-Miguel}M. Méndez and J. Yang,  M\"{o}bius species. Adv. Math. 85 (1991), no. 1, 83-128.

\bibitem{Mullin}R. Mullin, and G.-C. Rota, On the foundations of combinatorial
  theory. III. Theory of binomial enumeration. 1970 Graph Theory and its Applications
 (Proc. Advanced Sem., Math. Research Center, Univ. of Wisconsin, Madison, Wis., 1969)
 pp. 167-213.

\bibitem{Livernet}M. Livernet, A rigidity theorem for pre-Lie
algebras. J Pure Appl. Algebra, 207(1), 1-18,2006.


\bibitem{quadal} A. Polishchuk and L. Positselski, Quadratic
Algebras, University Lecture Series AMS, Vol. 37, Rhode Island,
2005.

\bibitem{Priddy} S. Priddy, Koszul resolutions. Trans. Amer. Math.
Soc. 152, 39-60, 1970.

\bibitem{Reiner}D. Reiner, The combinatorics of polynomial
sequences. Studies in Appl. Math. 58 (1978), no. 2, 95-117

\bibitem{Rota}G-C. Rota, On the foundations of combinatorial theory
 I. Theory of Möbius functions,
 Z. Wahrscheinlichkeitstheorie und Verw. Gebiete 2 (1964), 340-368.

\bibitem{Solomon} L. Solomon, A decomposition of the group algebra
of a finite Coxeter group, J. Algebra, 9 (1968), 220-239.


\bibitem{Stanley-Andre}R. Stanley,
Alternating permutations and symmetric functions.  J. Combin.
Theory Ser. A Vol. 114 (2007), no. 3, 436-460.

\bibitem{Stanley-Binomial}R. Stanley, Binomial posets, Möbius inversion,
and permutation enumeration, J. Combinatorial Theory (A) 20
(1976), 336-356.

\bibitem{Stanley0} R. Stanley, Enumerative combinatorics. Vol. 1,
Cambridge Studies in Advanced Mathematics,  49, Cambridge
University Press, Cambridge, 1997.

\bibitem{Stanley} R. Stanley, Enumerative combinatorics. Vol. 2.
Cambridge Studies in Advanced Mathematics, 62. Cambridge
University Press, Cambridge, 1999.

\bibitem{Stanley-Hipparcus}R. Stanley, Hipparchus, Plutarch, Schröder and
Hough, American Mathematical Monthly Vol. 104 (1997), 344-50.

\bibitem{Stanleyg}R. Stanley, Some aspects of groups acting on
finite posets, J. Combinatorial Theory (A) 32 (1982), 132-161.
\bibitem{Gilbert}G. Labelle, Une nouvelle d\'emonstration combinatoire des formules
d'inversion de Lagrange, Advances in Mathematics, Acad. Press,
Vol. 42, no. 3 (1981 ) 217-247.

\bibitem{Vallette} B. Vallette, Homology of generalized partition posets,
Journal of Pure and Applied Algebra, Vol. 208, no. 2 (February
2007) 699-725.

\bibitem{Wachs}M. Wachs, Poset Topology, Geometric combinatorics, 497-615, IAS/Park City Math.
Ser., 13, Amer. Math. Soc., Providence, RI, 2007.

\end{thebibliography}
\end{document}